\documentclass[%
 aip,
 amsmath,amssymb,
 reprint,%
nofootinbib
]{revtex4-1}

\usepackage{graphicx}% Include figure files
\usepackage{dcolumn}% Align table columns on decimal point
\usepackage{bm}% bold math
\usepackage[utf8]{inputenc}
\usepackage[T1]{fontenc}
\usepackage{mathptmx}
\usepackage{cleveref}
\usepackage{amsthm}
\usepackage[english]{babel} 
\usepackage{booktabs}

\usepackage[caption=false]{subfig} 
\usepackage{tikz}
\newtheorem{definition}{Definition}
\newtheorem{remark}{Remark}
\newtheorem{theorem}{Theorem}
\newtheorem{proposition}{Proposition}

\DeclareMathOperator{\conv}{conv}

\DeclareMathOperator{\newt}{Newt}
\DeclareMathOperator{\nvol}{NVol}
\DeclareMathOperator{\init}{init}
\newcommand{\boldzero}{\mathbf{0}}
\newcommand{\boldc}{\mathbf{c}}
\newcommand{\bolde}{\mathbf{e}}
\newcommand{\boldf}{\mathbf{f}}
\newcommand{\boldp}{\mathbf{p}}
\newcommand{\boldx}{\mathbf{x}}
\newcommand{\boldy}{\mathbf{y}}
\newcommand{\boldalpha}{\boldsymbol{\alpha}}
\newcommand{\Q}{\mathbb{Q}}
\newcommand{\R}{\mathbb{R}}
\newcommand{\C}{\mathbb{C}}
\newcommand{\imag}{\mathbf{i}}
\newcommand{\adjp}{\nabla}
\newcommand{\inner}[2]{\left\langle #1 \,,\, #2 \right\rangle}
\newcommand{\term}[1]{\textbf{#1}}
\newcommand{\tech}[1]{\textsf{#1}}

\begin{document}

% \preprint{AIP/123-QED}

\title{Directed acyclic decomposition of Kuramoto equations}

\author{Tianran Chen}
\affiliation{Department of Mathematics, 
    Auburn University at Montgomery, Montgomery, Alabama USA\\
    Department of Mathematics, Michigan State University,
    East Lansing, Michigan USA}
\email{ti@nranchen.org}
\homepage{www.tianranchen.org}.
% \author{A. Author}
%  \altaffiliation[Also at ]{Physics Department, XYZ University.}%Lines break automatically or can be forced with \\
% \author{B. Author}%
%  \email{Second.Author@institution.edu.}
% \affiliation{ 
% Authors' institution and/or address%\\This line break forced with \textbackslash\textbackslash
% }%

% \author{C. Author}
%  \homepage{http://www.Second.institution.edu/~Charlie.Author.}
% \affiliation{%
% Second institution and/or address%\\This line break forced% with \\
% }%

\date{\today}% It is always \today, today,

\begin{abstract}
    The Kuramoto model is one of the most widely studied models for
    describing synchronization behaviors in a network of coupled oscillators,
    and it has found a wide range of applications.
    Finding all possible frequency synchronization configurations in a general
    non-uniform, heterogeneous, and sparse network is important yet challenging 
    due to complicated nonlinear interactions.
    From the view point of homotopy deformation, 
    we develop a general framework for decomposing a Kuramoto network into smaller directed acyclic subnetworks,
    which lays the foundation for a divide-and-conquer approach to
    studying the configurations of frequency synchronization of large Kuramoto networks.
\end{abstract}

\maketitle

% REQUIRED
% \begin{keywords}
%     synchronization, Kuramoto equations, adjacency polytope, polyhedral homotopy
% \end{keywords}

% REQUIRED
% \begin{AMS}
%     14Q99, 52B20, 65H10
% \end{AMS}

\begin{quotation}
    The spontaneous synchronization of a network of oscillators
    is an emergent phenomenon that naturally appears
    in many seemingly independent complex systems including
    mechanical, chemical, biological, and even social systems.
    The Kuramoto model is one of the most widely studied and successful
    mathematical models for analyzing synchronization behaviors.
    While much is known about the macro-scale question of whether or not
    a Kuramoto network can be synchronized, 
    detailed analysis of the possible configurations 
    of the oscillator once it has reached synchronization
    remains difficult for large networks
    partly due to the nonlinear interactions involved.
    In this work, we demonstrate that 
    by dividing the link between two oscillators into two one-way interactions,
    complex networks can indeed be decomposed into much simpler subnetwork.
    This is a crucial step toward fully understanding synchronization
    configurations in large networks.
\end{quotation}

\section{Introduction}\label{sec:intro}
Mathematical modeling and analysis of spontaneous synchronization
have found many important applications in 
physics, 
chemistry,
engineering,
biology,
and medical science.~\cite{dorfler_synchronization_2014}
Originally introduced to describe chemical oscillators,
the Kuramoto model~\cite{Kuramoto1975Self,Kuramoto2012Chemical} 
has become one of the most widely studied models for
describing synchronization behaviors in a wide range of situations.
This paper focuses on the study of frequency synchronization,
which describes a particular type of synchronization behavior where 
oscillators are tuned into the same frequency.
The central objective is to understand the set of \emph{all} such configurations 
on potentially non-uniform and non-homogeneous Kuramoto network.
With the appropriate frame of reference,
this problem is equivalent to the study of the full set of
solutions of the system of nonlinear equations
\begin{equation*}%\label{equ:kuramoto-sin}
    \omega_i - \sum_{j \in \mathcal{N}_G(i)} k_{ij} \sin(\theta_i - \theta_j)
    \;=\; 0
    \quad\text{for } i=1,\dots,n.
\end{equation*}

The main contribution of this paper is the development of a general framework for 
decomposing a Kuramoto network into subnetworks that can be studied more easily.
This decomposition is enabled by the key insights gained through
an abstract polytope that encodes the network topology.
    The way in which a network is broken down into subnetworks
    mirrors the process through which how the boundary of that polytope
    is broken down into facets.
    From the view point of dynamics, 
    this decomposition comes from
    the limit behavior of Kuramoto networks
    as the differences of the oscillators' natural frequencies are amplified to infinity.
The framework presented in this paper marks a crucial step toward a 
divide-and-conquer approach to studying large Kuramoto networks.
\vskip -0.5ex

Using a complex algebraic formulation, 
we deform the synchronization equations
under which a Kuramoto network
degenerates into a union of those of simpler subnetworks 
supported by directed acyclic subgraphs.
By extending our domain to complex phase angles, 
% $x_i = e^{r_i + \imag \theta_i}$ 
we can form a new rational system 
that also captures the synchronization configurations as solutions.
This procedure, detailed in \cref{sec:complex},
allows for the introduction of powerful tools stemming from complex algebraic geometry.
The main ideas are briefly illustrated through a simple example in~\cref{sec:main}.
Then, in \cref{sec:facet}, 
through a construction known as the \emph{adjacency polytope},
we decompose the Kuramoto network into
simpler subnetworks induced by facets of this polytope.
Each ``facet subnetwork'' corresponds to a directed acyclic subnetwork
of the original network.
We also explore the topological properties of these subnetworks.
% (\cref{thm:facet-network}).
These ``facet subnetworks'' preserves and reveals many important
properties of the original network.
% In particular, 
We demonstrate in \cref{thm:root-count} that the number of
synchronization configurations that the original network has 
is bounded by the total root count %number of synchronization configurations
of the facet subnetworks.
In \cref{thm:homotopy}, we demonstrate that the decomposition of a
network into facet subnetworks can be understood as a 
a smooth deformation of the synchronization equations 
that can degenerate into facet subsystems.
Among the facet subnetworks, 
simplest type is known as a ``primitive subnetwork''
and is analyzed in \cref{sec:primitive}.
Each primitive subnetwork has a unique synchronization configuration,
and hence form the basic building block of the decomposition.
Finally, we illustrate the decomposition scheme via concrete examples
in~\cref{sec:example} 
and conclude with remarks on future directions in~\cref{sec:conclusions}.

%=============================================================================
\section{Kuramoto model}\label{sec:kuramoto}
The Kuramoto model~\cite{Kuramoto1975Self} is a mathematical model 
used to study behavior of a network of coupled oscillators.
% It is originally proposed to study chemical oscillators
% but has since found numerous different applications in many different fields
% We briefly review the formulation of this model and the main equations in this section
% (see~\cite{dorfler_synchronization_2014} for a recent survey).
An oscillator is simply an object that can continuously vary between two states.
In isolation, each oscillator has its own natural frequency.
When we consider networks of coupled oscillators, however, 
rich and complex dynamic behaviors emerge. 
% A simple mechanical illustration of the coupled oscillator model is a
% spring network that consists of a set of weightless particles 
% constrained to move on the unit circle without friction or collision.
The oscillators are coupled with one another by idealized springs
with stiffness characterized by their \emph{coupling strength}.
For a pair of oscillators $(i,j)$,
the real number $k_{ij}$ quantifies the strength of coupling between them.
% In isolation, each particle has its own natural frequency.
% However, in this connected network the tug of war between the particles’ 
% tendency to oscillate in their own natural frequencies and the pulling of 
% their neighbors gives rise to rich dynamic behavior. 
The topology of a network of $N = n + 1$ oscillators 
is modeled by a graph $G = (V, E)$,
in which nodes $V = \{0,\dots,n\}$ represent the oscillators 
and edges $E$ represent their connections. 
The coupling strengths $K = \{k_{ij}\}$ 
and the natural frequencies $\omega = (\omega_0, \dots, \omega_n)$
capture the quantitative information of the network.
In its simplest form, the Kuramoto model is a differential equation 
that describes the nonlinear interactions in such a network 
given by
\begin{equation}\label{equ:kuramoto-ode}
    \dot{\theta}_i (t) \;=\;
    \omega_i - \sum_{j \in \mathcal{N}_G(i)} k_{ij} \sin(\theta_i(t) - \theta_j(t))
    \;, \text{for } i=0,\dots,n,
\end{equation}
where $\theta_i(t)$ is the phase angle of the oscillator $i$
as a function of time $t$,
and $\mathcal{N}_G(i)$ denotes the set of its neighboring nodes.

The network described by $(G,K,\omega)$ together with the above differential equation
is referred to as a \term{Kuramoto network}.
Central to this paper are the special configurations in which the angular velocity 
of all oscillators become perfectly aligned, 
known as \term{frequency synchronization configurations}. 
That is, $\frac{d\theta_i}{dt} = c$ for $i = 0, \dots , n$ and a constant $c$. 
By adopting a proper frame of reference, we can assume $c = 0$,
and the (frequency) synchronization configurations are simply equilibria of the
Kuramoto model~\eqref{equ:kuramoto-ode} given by
\begin{equation}\label{equ:kuramoto-sin}
    \omega_i - \sum_{j \in \mathcal{N}_G(i)} k_{ij} \sin(\theta_i(t) - \theta_j(t))
    \; = \; 0
    \quad\text{for } i=0,\dots,n.
\end{equation}
For the removal of the inherent degree of freedom induced by uniform translation,
the standard practice is to choose node 0 to be the \emph{reference node} 
and set $\theta_0 = 0$.
Assuming the couplings are symmetric, i.e., $k_{ij} = k_{ji}$,
the $n+1$ equations are then linearly dependent.
This allows for the elimination of one equation
and produces a system of $n$ equations in $n$ unknowns $\theta_1,\dots,\theta_n$. 
This \term{synchronization system} is the main focus of this paper.
% and will be henceforth referred to as the system of 
% \term{synchronization equations} induced by the Kuramoto network $(G,K,\omega)$.
The structure of the solution set to this system
has been the subject of intense research since the 1970s.
\cite{Baillieul1982,Kuramoto1975Self}
While earlier studies have focused on macro-scale and statistical analyses
of large (often infinite) networks,\cite{Kuramoto1984Cooperative}
recent research has gradually shifted toward precise analysis,
of small finite networks.\cite{Strogatz2000,AeyelsRogge2004Existence,MirolloStrogatz2005Spectrum,MehtaDaleoDorflerHauenstein2015Algebraic,ChenMarecekMehtaNeimerg2019Three,RoggeAeyels2004Stability,OchabGora2010Synchronization,DelabaysColettaJacquod2016Multistability,DelabaysColettaJacquod2017Multistability,ManikTimmeWitthaut2017Cycle}
The complex algebraic approach\cite{Baillieul1982} started by Baillieul and Byrnes
has been particularly successful in this regard 
and forms the foundation of this work.

The key question we set out to answer is whether or not 
the full set of solutions to this system can be understood
through a study of the simpler subnetworks of the original Kuramoto network.
% via a decomposition of this system into simpler subsystems.

%=============================================================================
\section{Complex algebraic formulation}\label{sec:complex}

Even though the original formulation of the synchronization system~\eqref{equ:kuramoto-sin} 
considers only real phase angles, it is useful to expand the domain to 
more general complex phase angles,
as this would allow us to apply powerful tools from complex algebraic geometry.
Consider complex phase angle $z_i = \theta_i -  r_i\imag$ for $i = 0,\dots,n$ 
where $\imag = \sqrt{-1}$.
Using the identity $\sin(z) = \frac{1}{2\imag} (e^{\imag z} - e^{-\imag z})$,
we define the new complex variables
\begin{equation}
    x_i = e^{\imag z_i} = e^{r_i+\imag \theta_i}
    \quad \text{for } i=1,\dots,n
    \quad \text{ and } x_0 = e^{0+0\imag} = 1.
\end{equation}
These variables represent the phases of the oscillator as points on the complex plane.
If $r_i = 0$, then $x_i$ lies on the unit circle as the original formulation requires.
If $r_i \ne 0$, $x_i$ deviates from the unit circle
and no longer represents real solutions of the original equations.
However, such \emph{extraneous solutions} (i.e. non-real solutions) 
can be identified easily.
With these new variables, the transcendental terms in~\eqref{equ:kuramoto-sin} 
can be converted to rational functions
\[
    \sin(z_i - z_j) =
    \frac{ e^{z_i \imag} e^{-z_j \imag} - e^{z_j \imag} e^{-z_i \imag} }{ 2 \imag } =
    \frac{1}{2 \imag} \left( \frac{x_i}{x_j} - \frac{x_j}{x_i} \right),
\]
and the system can be transformed into
the \term{algebraic synchronization system}
\begin{equation}\label{equ:kuramoto-rat}
    \omega_i - \sum_{j \in \mathcal{N}_G(i)} 
    a_{ij}' \left( \frac{x_i}{x_j} - \frac{x_j}{x_i} \right)
    \;=\;0 \quad\text{for } i=1,\dots,n
\end{equation}
in the complex variables $x_1,\dots,x_n$
where $a_{ij}' = \frac{k_{ij}}{2 \imag}$.
In this formulation, 
we also allow for complex coupling strength and natural frequencies.
This algebraic formulation has been used in 
Ref.~\onlinecite{Chen2019Unmixing,ChenDavisMehta2018Counting},
and it is similar to the formulation in Ref.~\onlinecite{Baillieul1982}.
However, the above formulation directly connects to the construction of ``adjacency polytopes.''
% Note here that in this formulation,  we no longer require the symmetry of 
% coupling strength, i.e., we no longer enforce $a_{ij}' = a_{ji}'$.

Denote the rational functions on the left hand sides of the above system by
$\boldf = (f_1,\dots,f_n)^\top$. 
It is clear that the two systems $\boldf = \boldzero$ and $M \boldf = \boldzero$ 
have the exact same solution set for any nonsingular $n \times n$ matrix $M = [M_{ij}]$.
Therefore, without loss of generality, 
we can consider the equivalent system $M \boldf = \boldzero$, which is of the form
\begin{equation}
    \label{equ:kuramoto-unmixed}
    c_k - \sum_{(i,j) \in \mathcal{E}(G)} 
    a_{ijk} \left( \frac{x_i}{x_j} - \frac{x_j}{x_i} \right)
    \;=\; 0
    \quad\text{for } k = 1,\dots,n 
\end{equation}
with $c_k = \sum_{i=1}^n M_{ki} \omega_i$.
This system is referred to as the \term{unmixed form}
of the algebraic synchronization system
(simply \term{unmixed synchronization system} hereafter).
Replacing a system of equations $\boldf = \boldzero$
by $M \, \boldf = \boldzero$ for a nonsingular square matrix $M$
is a common practice in numerical methods for solving nonlinear systems
(e.g., the \emph{randomization} procedure
in numerical algebraic geometry~\cite{SommeseWampler2005Numerical}).

In this form, every monomial appears in every equation,
and each equation no longer represents the balancing condition on a  single oscillator. 
Instead, each equation becomes a linear combination of all
balancing conditions in the network.
The rest of the paper focuses on this system.

%=============================================================================
\section{A example}\label{sec:main}

% \begin{wrapfigure}[11]{r}{0.37\textwidth}
%     \centering
%     \begin{tikzpicture}[scale=0.5]
%         \node[draw,circle] (0) at ( 0.0000, 3.00000) {$0$};
%         \node[draw,circle] (1) at (-2.5981,-1.50000) {$1$};
%         \node[draw,circle] (2) at ( 2.5981,-1.50000) {$2$};
%         \path[draw,thick,-latex] (1) -- (0);
%         \path[draw,thick,-latex] (2) -- (0);
%     \end{tikzpicture}
%     \caption{
%         A subnetwork of the network of
%         3 coupled oscillators.
%     }\label{fig:comp3-sub1}
% \end{wrapfigure}
% \begin{wrapfigure}[10]{r}{0.35\textwidth}
%     \centering
%     \begin{tikzpicture}[scale=0.5]
%         \node[draw,circle] (0) at ( 0.0000, 3.00000) {$0$};
%         \node[draw,circle] (1) at (-2.5981,-1.50000) {$1$};
%         \node[draw,circle] (2) at ( 2.5981,-1.50000) {$2$};
%         \path[draw,thick] (0) -- (1);
%         \path[draw,thick] (1) -- (2);
%         \path[draw,thick] (2) -- (0);
%     \end{tikzpicture}
%     \caption{
%         3 coupled oscillators.
%     }\label{fig:comp3}
% \end{wrapfigure}
We briefly outline the basic idea behind the 
proposed decomposition scheme with a simple example of a network of 
three coupled oscillators (\Cref{fig:comp3}).
Though this case is simple enough to be solved directly by algebraic manipulations,
it can still illuminate main idea of the proposed framework.
In this case, the unmixed synchronization system~\eqref{equ:kuramoto-unmixed} is
\begin{equation}
    \label{equ:unmixed-comp3}
    \small
    % \left\{
    \begin{aligned}
        c_1 &= 
            a_{101} \left(\frac{x_1}{x_0} - \frac{x_0}{x_1} \right) +
            a_{121} \left(\frac{x_1}{x_2} - \frac{x_2}{x_1} \right) +
            a_{201} \left(\frac{x_2}{x_0} - \frac{x_0}{x_2} \right) 
         \\
        c_2 &= 
            a_{102} \left(\frac{x_1}{x_0} - \frac{x_0}{x_1} \right) +
            a_{122} \left(\frac{x_1}{x_2} - \frac{x_2}{x_1} \right) +
            a_{202} \left(\frac{x_2}{x_0} - \frac{x_0}{x_2} \right) .
    \end{aligned}
    % \right.
\end{equation}
Existing root count results\cite{Baillieul1982,ChenDavisMehta2018Counting} 
dictate that there are no more than six complex synchronization configurations 
for this Kuramoto network of three oscillators.
Our goal is to understand these synchronization configurations
by examining simpler subnetworks of this network.

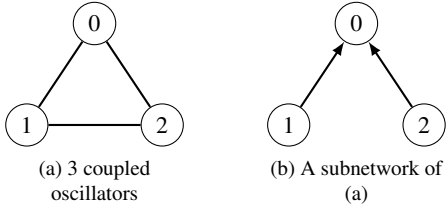
\begin{figure}[h]
    \centering
    % \begin{subfigure}[b]{0.4\textwidth}
    \subfloat[3 coupled oscillators]{
        \label{fig:comp3}
        \centering
        \begin{tikzpicture}[scale=0.3]
            \node[draw,circle] (0) at ( 0.0000, 3.00000) {$0$};
            \node[draw,circle] (1) at (-3.0000,-1.50000) {$1$};
            \node[draw,circle] (2) at ( 3.0000,-1.50000) {$2$};
            \path[draw,thick] (0) -- (1);
            \path[draw,thick] (1) -- (2);
            \path[draw,thick] (2) -- (0);
        \end{tikzpicture}
    }\quad\quad\quad%
    \subfloat[A subnetwork of (a)]{
        \label{fig:comp3-sub1}
        \centering
        \begin{tikzpicture}[scale=0.3]
            \node[draw,circle] (0) at ( 0.0000, 3.00000) {$0$};
            \node[draw,circle] (1) at (-3.0000,-1.50000) {$1$};
            \node[draw,circle] (2) at ( 3.0000,-1.50000) {$2$};
            \path[draw,thick,-latex] (1) -- (0);
            \path[draw,thick,-latex] (2) -- (0);
        \end{tikzpicture}
    }
    \caption{A network of three coupled oscillators and one of its subnetwork}
    \label{fig:ex1}
\end{figure}

\Cref{fig:comp3-dac} displays the proposed decomposition scheme of 
this Kuramoto network into six subnetworks, 
supported by directed acyclic graphs 
(i.e., asymmetric couplings)
with each subnetwork corresponding to one of the possible 
complex synchronization configurations.

\begin{figure}[t]
    \centering
    \begin{tikzpicture}[scale=0.3]
        \node[draw,circle] (0) at ( 0.0000, 3.00000) {$0$};
        \node[draw,circle] (1) at (-2.5981,-1.50000) {$1$};
        \node[draw,circle] (2) at ( 2.5981,-1.50000) {$2$};
        \path[draw,thick] (0) -- (1);
        \path[draw,thick] (1) -- (2);
        \path[draw,thick] (2) -- (0);
        \node (d1) at (-7 , 7) { \includegraphics[width=0.1\textwidth]{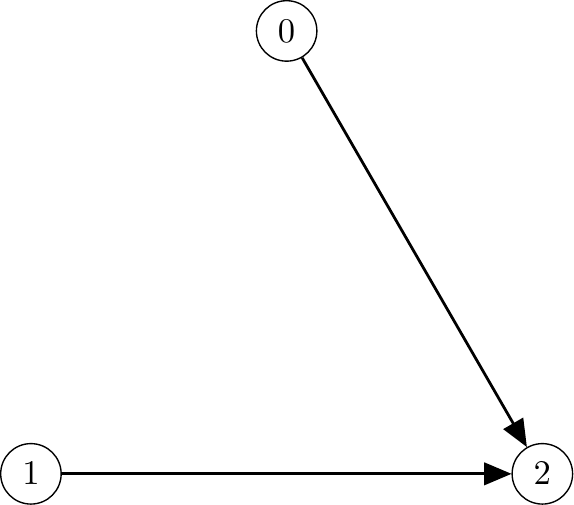} };
        \node (d2) at (-11, 0) { \includegraphics[width=0.1\textwidth]{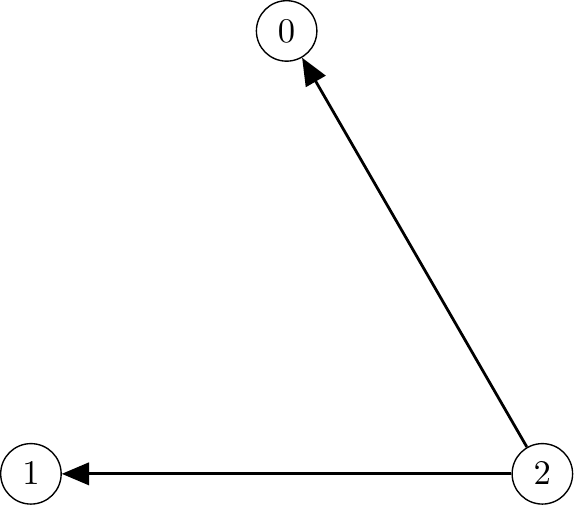} };
        \node (d3) at (-6 ,-6) { \includegraphics[width=0.1\textwidth]{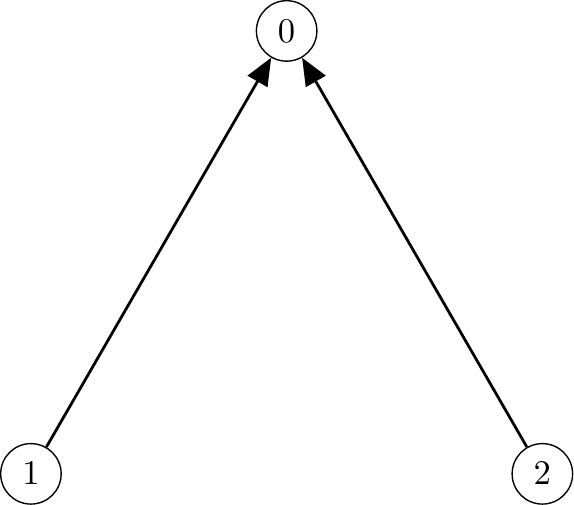} };
        \node (d4) at ( 7 , 7) { \includegraphics[width=0.1\textwidth]{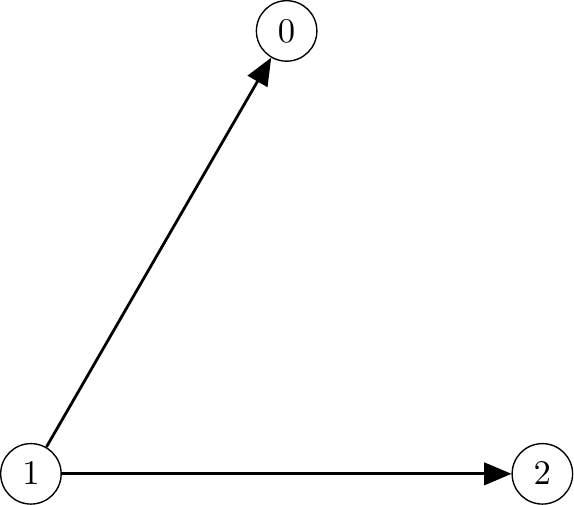} };
        \node (d5) at ( 11, 0) { \includegraphics[width=0.1\textwidth]{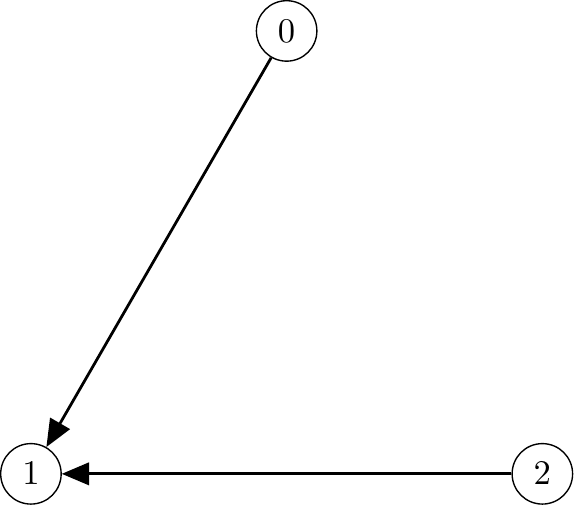} };
        \node (d6) at ( 6 ,-6) { \includegraphics[width=0.1\textwidth]{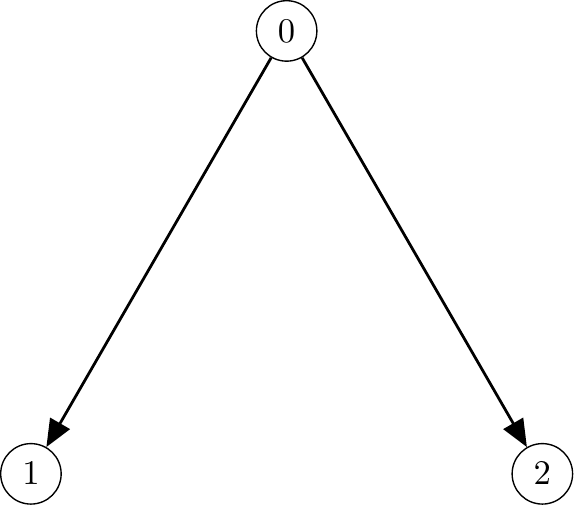} };
        \path[draw,dashed,-latex] (-1.5,2.0) -- (d1);
        \path[draw,dashed,-latex] ( 1.5,2.0) -- (d4);
        \path[draw,dashed,-latex] (-3.0,0.0) -- (d2);
        \path[draw,dashed,-latex] ( 3.0,0.0) -- (d5);
        \path[draw,dashed,-latex] (-5.0,3.5) -- (d1);
        \path[draw,dashed,-latex] (-1.5,-2.5) -- (d3);
        \path[draw,dashed,-latex] ( 1.5,-2.5) -- (d6);
    \end{tikzpicture}
    \caption{
        The decomposition of a network of three oscillators (center)
        into six subnetworks supported by directed acyclic graphs.
    }\label{fig:comp3-dac}
\end{figure}
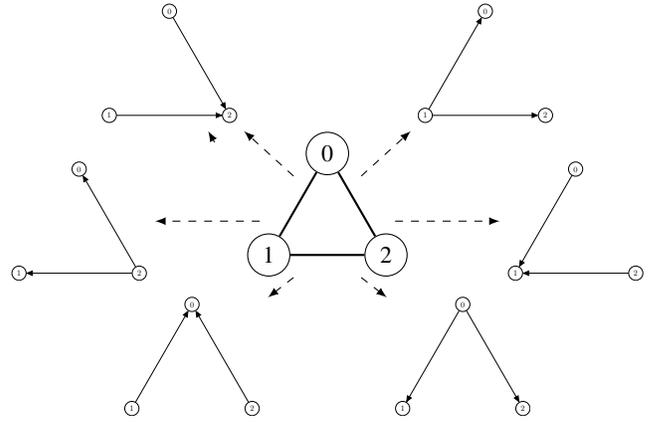

Each of the subnetworks has its own synchronization system.
For instance, the subnetwork displayed in~\Cref{fig:comp3-sub1}
corresponds to the unmixed synchronization system given by
\begin{equation}\label{equ:subsys3}
    \begin{aligned}
        c_1 &= a_{110} x_1 / x_0 + a_{120} x_2 / x_0 \\
        c_2 &= a_{210} x_1 / x_0 + a_{220} x_2 / x_0
    \end{aligned}
\end{equation}
which is a system of equations involving a subset of terms
in~\eqref{equ:unmixed-comp3}.
Here, each edge is interpreted as a directed edge,
and the directed edge $(i,j)$ corresponds to the term
$x_i / x_j$.
It is easy to verify that for generic choices of coefficients,
the above system has a unique complex solution.
The same analysis can be applied to the rest of the subnetworks 
presented in \Cref{fig:comp3-dac}, 
and we can verify that the synchronization
system induced by each subnetwork has a unique solution 
under the assumption of generic coefficients.
Moreover, these solutions are in one-to-one correspondence with
the six complex frequency synchronization configurations of the
original network in \Cref{fig:comp3} via a homotopy function
$H(x_1,x_2,t) : \C^2 \times \R \to \C^2$ given by
\begin{equation*}
    \small
    \left\{
    \begin{aligned}
        \frac{c_1}{t} - \left[
            a_{101} \left(\frac{x_1}{x_0} - \frac{x_0}{x_1} \right) +
            a_{121} \left(\frac{x_1}{x_2} - \frac{x_2}{x_1} \right) +
            a_{201} \left(\frac{x_2}{x_0} - \frac{x_0}{x_2} \right) 
        \right]
        \\
        \frac{c_2}{t} - \left[
            a_{102} \left(\frac{x_1}{x_0} - \frac{x_0}{x_1} \right) +
            a_{122} \left(\frac{x_1}{x_2} - \frac{x_2}{x_1} \right) +
            a_{202} \left(\frac{x_2}{x_0} - \frac{x_0}{x_2} \right) 
        \right]
    \end{aligned}
    \right.
    .
\end{equation*}
At $t=1$, the equation $H(x_1,x_2,1)=0$ is equivalent to the
original unmixed synchronization system~\eqref{equ:unmixed-comp3}.
Under a mild ``genericity'' assumption, it can be proved that as $t$
varies continuously from 1 to 0 (without reaching 0),
the complex solutions of $H(x_1,x_2,t) = 0$ also move smoothly,
forming smooth paths emanating from the solutions at $t=1$.

As $t$ approaches 0, $H$ becomes undefined, and the six solution paths 
degenerate into the complex frequency synchronizations configurations
of the six subnetworks in \Cref{fig:comp3} in the sense that
the six complex synchronization configurations of the six subnetworks
are the limit points of the following reparametrization of the six paths:
\begin{align*}
    \small
    x_1   &= y_1    &   \frac{x_1}{t} &= y_1  &   \frac{x_1}{t} &= y_1  &   x_1     &= y_1  &   x_1 t &= y_1    &   x_1 t &= y_1 \\
    x_2 t &= y_2    &   \frac{x_2}{t} &= y_2  &   x_2     &= y_2  &   \frac{x_2}{t} &= y_2  &   x_2   &= y_2    &   x_2 t &= y_2.
\end{align*}
That is, as $t \to 0$, along each path, $(y_1,y_2)$ converge to
the unique solution of the synchronization systems of 
one of the subnetworks shown in the decomposition \Cref{fig:comp3-dac}.

As we will demonstrate in~\cref{sec:primitive},
the unique solution of each of the six subsystems (e.g.~\eqref{equ:subsys3})
can be computed easily and explicitly.
This gives rise to a practical algorithm that can be used to
find all (complex) solutions to system ~\eqref{equ:unmixed-comp3}.
After finding the solutions to the six subsystems,
one can trace the smooth paths, 
starting from these solutions and reach the solutions to the original synchronization system~\eqref{equ:unmixed-comp3}.

The deformation induced by $H(x_1,x_2,t)$ works by
moving the parameter $t$ between 1 and 0 with $t=1$ corresponding to the original system.
This has the effect of amplifying the difference in natural frequencies of the oscillators.
As $t \to 0$, the difference in natural frequencies goes to infinity,
at which point synchronization of the entire network is no longer possible
(even considering complex configurations)
and the network breaks down into six simpler subnetworks.

One important feature of this decomposition scheme is that 
it is guaranteed to be able to find \emph{all} complex solutions to~\eqref{equ:unmixed-comp3},
which include \emph{all} real solutions to the original Kuramoto equations
\eqref{equ:kuramoto-sin}.
In the following sections, we develop the general construction.

%=============================================================================
\section{Adjacency polytope and facet networks}\label{sec:facet}

The decomposition illustrated above is constructed from the geometric
information encoded in a polytope --- 
the ``adjacency polytope.''
\cite{Chen2019Unmixing,ChenDavisMehta2018Counting}
In this section, we briefly review the definition.
A polytope is the convex hull of a finite lists of points in $\R^n$
(a bounded geometric object with finitely many flat faces).
% Equivalently, it is also the bounded intersection of a finite number of half spaces.
% Here we shall define a polytope that will capture the topological information
An adjacency polytope is a polytope that can 
capture the topology of a network,
count the number of synchronization configurations, 
and guide the decomposition of the network.

\begin{definition}[Chen\cite{Chen2019Unmixing}]\label{def:ap}
    Given a Kuramoto network $(G,K,\omega)$,
    the \term{adjacency polytope} of this network is the polytope
    \begin{equation}
        \adjp_G = 
        \conv
        \left\{
            \bolde_i - \bolde_j
            \mid (i,j) \in \mathcal{E}(G)
        \right\},
    \end{equation}
    and the \term{adjacency polytope bound} of this network is
    the \emph{normalized volume}
    % \footnote{%
    %     The normalized the volume of a polytope 
    %     in $\R^d$ is its volume times $d \, !$.
    %     % It is easy to verify that this normalized volume is always an integer.
    % }
    $\nvol(\adjp_G) = n \, ! \operatorname{vol}_n(\adjp_G)$.
\end{definition}

Here, $\bolde_i$ is the $i$-th standard basis of $\R^n$ and $\bolde_0 = \boldzero$.
The adjacency polytope is the convex hull of entries in the 
(directed) adjacency list of the graph as points in $\R^n$,
and it is a geometric encoding of the topology of the network.
This construction has been studied in the context of 
Kuramoto model and power flow equations.\cite{Chen2019Unmixing,ChenMehta2018Network}
% earlier works~\cite{Baillieul1982,LiSauerYorke1987Numerical,Marecek2014}.
Similar constructions have also appeared in other contexts.\cite{Matsui2011Roots,Higashitani2016Interlacing,DelucchiHoessly2016Fundamental}
The adjacency polytope bound is a simplification and relaxation of the 
Bernshtein-Kushnirenko-Khovanskii bound.\cite{Bernshtein1975Number}
Recently, the author proved that the this bound
remains a sharp upper bound for the total number of 
complex synchronization configurations for certain graphs.\cite{Chen2019Unmixing}
This bound is sharp in the sense that it is attainable when
generic coupling strength and natural frequencies are used.
The explicit formulae for this bound in the case of
cycle and tree graphs have also been established.\cite{ChenDavisMehta2018Counting}

%-------------------------------------------------------------------------------------------------
\subsection{Facet subnetworks and subsystems}

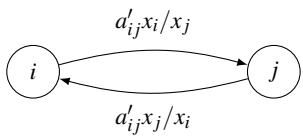
\begin{figure}[h]
    \centering
    \begin{tikzpicture}[scale=0.8,minimum size=0.7cm]
        \node[draw,circle] (i) at (0.0, 0.0) {$i$};
        \node[draw,circle] (j) at (4.0, 0.0) {$j$};
        \path[draw,-latex] (i) edge [bend left=15] node[above] { $a_{ij}' x_i/x_j$ } (j);
        \path[draw,-latex] (j) edge [bend left=15] node[below] { $a_{ij}' x_j/x_i$ } (i);
    \end{tikzpicture}
    \caption{Two directed edges}\label{fig:dir-edge}
\end{figure}

The main goal of this paper is to demonstrate that much more information 
of the original Kuramoto network can be extracted from this polytope.
Of particular importance are the facets of the adjacency polytope,
which are proper faces of maximum dimension.
Each facet gives rise 
to a directed acyclic subnetwork of the original network.
To define such generalized Kuramoto networks,
we first need to reinterpret synchronization equations
\eqref{equ:kuramoto-rat} and~\eqref{equ:kuramoto-unmixed}
from the view point of directed graphs.
In~\eqref{equ:kuramoto-rat}, an (undirected) edge $\{i,j\}$
corresponds to the term $a_{ij}'(x_i/x_j - x_j/x_i)$.
We can split this undirected edge into directed edges 
$(i,j)$ and $(j,i)$ with the corresponding terms 
$a_{ij}'x_i/x_j$ and $a_{ij}'x_j/x_i$ respectively,
as illustrated in~\Cref{fig:dir-edge}.
This allows us to define more general directed Kuramoto networks
induced by special directed graphs:

\begin{definition}\label{def:facet-network}
    For a facet $F$ of $\adjp_G$, we define 
    \begin{align*}
        \mathcal{V}_F &= \{ i \mid \bolde_i - \bolde_j \in F \text{ or } \bolde_j - \bolde_i \in F \text{ for some } j \} \\
        \mathcal{E}_F &= \{ (i,j) \mid \bolde_i - \bolde_j \in F \}.
    \end{align*}
    With these, we define the \term{facet subnetwork} associated with $F$ to be
    the (directed) Kuramoto network $(\mathcal{V}_F,\mathcal{E}_F)$
    and the corresponding \term{facet subsystem} to be the system of equations
    \begin{equation}\label{equ:facet}
        c_k - \sum_{(i,j) \in \mathcal{E}_F} a_{kij} 
        \left(
            \frac{x_i}{x_j}
        \right) = 0
        \quad\text{for each } k \in \mathcal{V}_F.
    \end{equation}
\end{definition}

A facet subsystem is a system of equations involving a subset of terms
in the unmixed synchronization system~\eqref{equ:kuramoto-unmixed}
of the original network.
The selection of the terms depends on the nodes and edges 
that appear in the corresponding facet subnetwork.
In this context, we consider the edge $(i,j)$ to be a directed edge.
That is, $(i,j) \in \mathcal{E}_F$ does not imply $(j,i) \in \mathcal{E}_F$.
As we shall prove, the directed edges $(i,j)$ and $(j,i)$
can never be in the same facet subnetwork.
% Comparing this system with~\eqref{equ:kuramoto-unmixed}, 
% we can interpret the facet system as a synchronization system
% defined on the facet subnetwork $((\mathcal{V}_F,\mathcal{E}_F),A,\omega)$
% with a topology described by a directed graph
% as opposed to the undirected topology used in~\eqref{equ:kuramoto-unmixed}
% and~\eqref{equ:kuramoto-rat} where an edge $(i,j)$ corresponds to the term 
% $a_{ij}'(x_i/x_j - x_j/x_i)$.

\begin{remark}
    The splitting of an undirected edge into two directed edges displayed in~\Cref{fig:dir-edge} 
    can also be interpreted as an explicit use of an imaginary interaction term.
    In considering the real frequency synchronization configurations
    defined by~\eqref{equ:kuramoto-sin}, we require $|x_i| = 1$ for all $i$
    since $x_i = e^{r_i + \imag \theta_i}$.
    In this case,
    \[
        a_{ij}' \frac{x_i}{x_j} = 
        \frac{k_{ij}}{2\imag} e^{\imag(\theta_i-\theta_j)} =
        \frac{k_{ij}}{2} [ \sin(\theta_i - \theta_j) - \imag \cos(\theta_i - \theta_j) ]
    \]
    which can be interpreted as the interaction term along the
    directed edge $(i,j)$.
    If both $(i,j)$ and $(j,i)$ are present,
    then the imaginary parts cancel each other out, leaving only the real part
    % \begin{gather*}
    %     \frac{k_{ij}}{2} [ \sin(\theta_i - \theta_j) - \imag \cos(\theta_i - \theta_j) ] -
    %     \frac{k_{ij}}{2} [ \sin(\theta_j - \theta_i) - \imag \cos(\theta_j - \theta_i) ] \\
    %     =
    %     k_{ij} \sin(\theta_i - \theta_j)
    % \end{gather*}
    $k_{ij} \sin(\theta_i - \theta_j)$ as the combined interaction term 
    that matches the original synchronization equations~\eqref{equ:kuramoto-sin}.
    This interpretation of the unidirectional interaction is similar to yet different from
    the interpretations used in the recent studies about Kuramoto models on directed graphs.
    \cite{DelabaysJacquodDorfler2018,RoggeAeyels2004Stability}
\end{remark}

\begin{remark}[Facet computation]
    In this section, we have established an one-to-one correspondence between
    the facets of an adjacency polytope and the facet subnetworks of a Kuramoto network.
    As we demonstrate in the following subsections,
    a great deal of information about all possible
    synchronization configurations can be extracted from the facet networks.
    The proposed analysis therefore hinges on finding all facets of adjacency polytopes.
    Fortunately, enumerating all facets of a convex polytope in high dimensions
    is a well-studied problem in convex geometry,
    and there exist several software packages that can carrying out such calculations
    efficiently.\cite{Fukuda1997,BuelerEngeFukuda2000Exact,Avis2018,polymake}
    Although this problem is known to be of class \#P in general,
    given the symmetries built into the definition of adjacency polytopes,
    it is plausible that facets of this type of polytope can be listed
    much more efficiently.
    For instance, the theory for exploiting the symmetry
    with respect to the origin 
    ($\boldx \in \nabla_G$ if and only if $-\boldx \in \nabla_G$)
    is well understood.
\end{remark}

\subsection{Topological properties of facet subnetworks}

A natural question to ask here is as follows:
what special topological properties do facet subnetworks posses? 
We now establish property restrictions and symmetries.
% We can verify that the special construction imposes certain restrictions.
% and limit the number of edges.

\begin{theorem}\label{thm:facet-network}
    Given a Kuramoto network, 
    % let $\adjp_G$ be its adjacency polytope.
    % Then for each facet $F$ of $\adjp_G$,
    % the facet system~\eqref{equ:facet}
    % is a synchronization system supported by a Kuramoto subnetwork 
    % the corresponding facet subnetwork $G_F = (\mathcal{V}_F,\mathcal{E}_F)$
    the facet subnetworks (\cref{def:facet-network}) have the following properties:
    \begin{itemize}
        \item Each facet subnetwork is acyclic;
        \item Each facet subnetwork contains all the nodes;
        \item Each facet subnetwork is weakly connected;
        \item Any two paths between a pair of nodes in facet subnetwork are of the same length; 
        \item Each path in a subnetwork contains no more than half of the edges of any cycle in $G$ containing it; and
        \item For each facet subnetwork, its transpose is also a facet subnetwork.
    \end{itemize}
\end{theorem}

\begin{proof}
    Let $\boldalpha = (\alpha_1,\dots,\alpha_n) \in \R^n$ be the 
    inner normal vector of a facet $F$ of $\adjp_G$,
    then there exists an $h \in \R$ such that 
    \begin{align*}
        \langle \boldp, \alpha \rangle &=   h \quad\text{for all } \boldp \in F \text{ while} \\
        \langle \boldp, \alpha \rangle &> h \quad\text{for all } \boldp \in \adjp_G \setminus F.
    \end{align*}
    However, $\boldzero$ is an interior point of $\adjp_G$. 
    Therefore $h < \langle \boldzero, \alpha \rangle = 0$.
    
    (Acyclic) 
    Suppose a facet subnetwork $G_F$ contains a directed cycle 
    $i_1 \to i_2 \to \cdots \to i_m \to i_1$ for some $\{i_1,\dots,i_m\} \subset \{0,\dots,n\}$.
    Let $i_{m+1} = i_1$,  then $\bolde_{i_k} - \bolde_{i_{k+1}} \in F$ 
    for $k = 1,\dots,m$.
    Since $\boldalpha$ is the inner normal of $F$, we have
    \[
        \inner{ \bolde_{i_k} - \bolde_{i_{k+1}} }{ \boldalpha } =
        \alpha_{i_k} - \alpha_{i_{k+1}} = h < 0
        \quad \text{for each } k = 1,\dots,m,
    \]
    and 
    \[
        \alpha_{i_1} < \alpha_{i_2} < \cdots < \alpha_{i_m} < \alpha_{i_1},
    \]
    which is a contradiction.
    % Therefore there cannot be any directed cycles.
    
    (Containing all nodes)
    For $n=1$, there are exactly two directed edges
    corresponding to the two facets of $\adjp_G$ 
    which is simply a line segment.
    The statement clearly holds.
    
    For $n>1$, since $\dim \adjp_G = n$, the facet $F$ is $(n-1)$-dimensional.
    Suppose there is a node $i \not\in G_F$.
    Without loss of generality,
    we can assume $i \ne 0$ by reassigning the reference node.
    Let $G'$ be the graph obtained from $G$ by removing the node $i$
    and all its edges.
    Then, $\adjp_{G'} \subset \adjp_G$, and
    the set of vertices of $\adjp_{G'}$ is a subset of the vertices of $\adjp_G$.
    It is straightforward to verify that $F$ remains a face of $\adjp_{G'}$.
    Therefore, $\dim F < \dim \adjp_{G'} \le n-1$,
    which contradicts with the assumption that $\dim F = n-1$.
    % We can thus conclude that $G_F$ contains all the nodes of $G$.
    
        (Weakly connected)
        Suppose $G_F$ is not weakly connected.
        Let $C_1 \ni 0$ be the nodes in one weakly connected component
        and let $C_2 = \mathcal{V}_F \setminus C_1$.
        By relabeling the nodes, we can assume $C_1 = \{0,1,\dots,m\}$
        and $C_2 = \{m+1,\dots,n\}$.
        Define
        \begin{align*}
            V_1 &= F \; \cap \; \{ \bolde_i - \bolde_j \;\mid\; i \in C_1 \} \\
            V_2 &= F \; \cap \; \{ \bolde_i - \bolde_j \;\mid\; i \in C_2 \}.
        \end{align*}
        Since $C_1$ and $C_2$ belong to different weakly connected components of $G_F$,
        i.e., there is no $\pm(\bolde_i - \bolde_j) \in F$ with $i \in C_1$ and $j \in C_2$,
        we can see that $V_1 \cap V_2 = \varnothing$.
        
        By construction, 
        $V_1 \subset \R^m \times \{ \boldzero \}$ and
        $V_2 \subset \{ \boldzero \} \times \R^{n-m}$.
        Since points in $V_2$ are points of the form $\bolde_i - \bolde_j$,
        we can see that they actually belong to the smaller subspace defined by
        $\langle \;\cdot\;, \mathbf{1} \rangle = 0$.
        Therefore, treating points in $V_1$ and $V_2$ as vectors, we get
        $\dim (\operatorname{span}(V_1) \le m)$ and
        $\dim (\operatorname{span}(V_2) < n - m)$.
        Consequently, the pyramid formed by $F = \conv (V_1 \cup V_2)$ and $\boldzero$
        is of a dimension strictly less than $m + n - m = n$,
        and therefore, the dimension of the facet $F$ itself must be strictly less than $n-1$.
        This contradicts with the fact that $\adjp_G$ is full-dimensional
        and $F$ is one of its facets.
    
    (Equal length) Suppose there are two paths
    $i_0 \to i_2 \to \dots \to i_m$ and $j_0 \to j_2 \to \dots \to j_\ell$  
    with the same starting and end points, that is, $i_0 = j_0$ and $i_m = j_\ell$.
    Then
    \[
        \inner{ \bolde_{i_k} - \bolde_{i_{k+1}} }{ \boldalpha } = h
        \quad \text{for each } k=0,\dots,m-1
    \]
    Summing these $m$ equations, we obtain
    \[
        \inner{ \bolde_{i_0} - \bolde_{i_m} }{ \boldalpha } = mh.
    \]
    Similarly, summing the $\ell$ equations along the path
    $j_0 \to j_2 \to \dots \to j_\ell$ produces 
    \[
        \inner{ \bolde_{j_0} - \bolde_{j_\ell} }{ \boldalpha } = \ell h.
    \]
    However, $j_0 = i_0$ and $i_m = j_\ell$. 
    Thus we must have $mh = \ell h$,
    i.e., $m = \ell$, and these two paths must have the same length.
    
    (Half cycle) 
    Suppose $G$ contains an undirected cycle
    $i_0 \leftrightarrow \dots \leftrightarrow i_m \leftrightarrow \dots \leftrightarrow i_\ell = i_0$.
    Among these edges, the path $i_0 \to i_2 \to \dots \to i_m$ is in $G_F$
    we can show that $m$ is no more than $\ell / 2$.
    By construction, $\bolde_{i_k} - \bolde_{i_{k+1}} \in F$ for each $k=0,\dots,m-1$,
    and hence,
    \[
        \inner{ \bolde_{i_k} - \bolde_{i_{k+1}} }{ \boldalpha } = h
        \quad\text{for each } k = 0,\dots,m-1.
    \]
    Summing these $m$ equations, we obtain
    \[
        \inner{ \bolde_{i_1} - \bolde_{i_m} }{ \boldalpha } = m h.
    \]
    Similarly, consider the path $i_\ell \to i_{\ell-1} \to \cdots \to i_m$, 
    which is in $G$. We must have
    \[
        \inner{ \bolde_{i_{k+1}} - \bolde_{i_{k}} }{ \boldalpha } \ge h
        \quad\text{for each } k = m,\dots,\ell-1.
    \]
    Summing these equations produces
    \[
        \inner{ \bolde_{i_1} - \bolde_{i_m} }{ \boldalpha } \ge (\ell - m) h
    \]
    Recall that $\inner{ \bolde_{i_1} - \bolde_{i_m} }{ \boldalpha } = m h$.
    Thus we must have $mh \ge (\ell - m)h$.
    Since $h<0$, this implies that $m \le \ell - m$, i.e., $2m \le \ell$.
    
    % The map $\bolde_i - \bolde_j \mapsto \bolde_j - \bolde_i$
    % is an permutation on the set of points defining $\adjp_G$
    % which has no fixed point.
    Since $\adjp_G$ is centrally symmetric,
    $-F = \{ -\boldx \mid \boldx \in F \} \subset \adjp_G$,
    and $\dim (-F) = \dim F = n-1$.
    % and $-F \cap F = \varnothing$.
    Moreover, for $\bolde_j - \bolde_i \in -F$,
    $\bolde_i - \bolde_j \in F$ by assumption, thus
    \[
        \inner{ \bolde_j - \bolde_i }{ -\boldalpha } =
        \inner{ \bolde_i - \bolde_j }{ \boldalpha } = h
    \]
    with the inner normal $\boldalpha$ of the facet $F$.
    Similarly,
    \[
        \inner{ -\boldp }{ -\boldalpha } =
        \inner{ \boldp }{ \boldalpha } >
        h 
    \]
    for any $-\boldp \in \adjp_G \setminus (-F)$.
    Therefore we can conclude that $-F$ is also a facet of $\adjp_G$.
    We can verify that the facet network $G_{-F}$ is the transpose of $G_F$ by definition.
\end{proof}

%--------------------------------------------------------------------------------
\subsection{Facet subnetworks as maximal flow networks}

The facet subnetworks can also be understood from the point of view of flow networks.
Indeed, they form subnetworks consisting of edges on which
certain pseudo-flow is maximized.

Given a vector $\boldalpha = (\alpha_1,\dots,\alpha_n) \in \R^n$,
we can define the pseudo-flow $f : V \times V \to \R$ given by
\begin{equation}
    f_{\boldalpha}(i,j) = \alpha_i - \alpha_j
\end{equation}
with $\alpha_0 = 0$.
Without loss of generality, we further require that the maximum flow be 1,
which can be used as the capacity for all edges.
By combining this setup with the topological properties established in \Cref{thm:facet-network},
we can see that according to this interpretation, 
a facet subnetwork is the subnetwork 
consisting of all vertices $\{0,1,\dots,n\}$, 
but only the edges with no residual capacity,
i.e., edges $(i,j)$ with $f_{\boldalpha}(i,j) = 1$.
Simply put, a facet subnetwork is the largest subnetwork 
on which the pseudo-flow $f_{\boldalpha}$ is maximized.
The problem of listing all facet subnetworks is therefore equivalent to 
the problem of finding all possible assignments of $\boldalpha$ 
that induce such a maximized flow-subnetworks.
Classic flow network algorithms based on linear programming problems
can thus be used here.
% Different assignment of $\boldalpha = (\alpha_1,\dots,\alpha_n) \in \R^n$
% may induce different subnetworks that maximized the pseudo-flow $f_{\boldalpha}$.
% For some assignments, the maximized subnetwork may not be largest (containing vertices $0,1,\dots,n$)

%--------------------------------------------------------------------------------
\subsection{Algebraic properties of facet subsystems}

    By adopting the complex formulation~\eqref{equ:kuramoto-rat}
    of the synchronization system,
    we paid the price of potentially introducing extraneous solutions
    (synchronization with imaginary phase angles).
    However, this effort allows us to employ more powerful tools
    from complex algebraic geometry.
A fundamental fact in complex algebraic geometry is that,
as far as complex root count is concerned,
the generic behavior of a family of algebraic systems
coincides with the maximal behavior.
That is, if we consider the facet system as a family of
algebraic systems parameterized by the coefficients,
then the generic complex root count coincides with the maximum root count.
This is an important consequence of the Bertini's Theorem,
and it forms the foundation of much of numerical algebraic geometry
(see Ref.~\onlinecite[Theorem 7.1.1]{SommeseWampler2005Numerical}).
Applying this to our context, we obtain the following proposition.

\begin{proposition}\label{pro:facet-root-count}
    For a generic choice of coefficients $\{c_k\}$ and $\{a_{ijk}\}$, 
    the total number of isolated\footnote{
        Isolated solutions here refer to geometrically isolated solutions.
        A solution of a system of equations is said to be geometrically isolated
        if there is a nonempty open set in which it is the only solution.
    }
    complex solutions to 
    a given facet subsystem~\eqref{equ:facet} induced by the facet
    $F$ of an adjacency polytope is a constant.
    Denote this constant by $\mathcal{N}(F)$,
    and then $\mathcal{N}(F)$ is also the maximum number of
    isolated complex solutions this facet subsystem could have.
\end{proposition}

More specifically, this proposition can be understood as a direct application of
the cheater's homotopy~\cite{LiSauerYorke1989Cheater}
or the parameter homotopy~\cite{MorganSommese1989Coefficient} theory
to the facet subsystems.
Moreover, Kushnirenko's Theorem~\cite{Kushnirenko1976Polyedres} 
provides us with the explicit formula
for the maximum root count in the form of normalized volume.

\begin{proposition}\label{pro:facet-vol}
    For a facet $F$ of an adjacency polytope $\adjp_G$,
    \[
        \mathcal{N}(F) = \nvol (\conv(F \cup \{ \boldzero \})) = n! \operatorname{vol}_n (\conv(F \cup \{ \boldzero \})).
    \]
\end{proposition}

\begin{proof}
    Since the directed graph $(\mathcal{V}_F,\mathcal{E}_F)$
    associated with a facet subnetwork is necessarily acyclic
    by~\cref{thm:facet-network},
    edges $(i,j)$ and $(j,i)$ cannot both be in $\mathcal{E}_F$.
    Consequently, the coefficients in the facet subsystem~\eqref{equ:facet} are independent.
    That is, generic choices of the parameters $\{ c_k \}$ and $\{ a_{ijk} \}$
    for the network correspond to independent generic choices of the coefficients
    for the facet subsystem.
    
    In addition, it is important to note that 
    the \emph{support} of the facet subsystem ---
    the collection of all the exponent vectors appeared in the system ---
    is the set of vertices in $F$ together with $\boldzero$.
    According to Kushnirenko's Theorem~\cite{Kushnirenko1976Polyedres},
    $\mathcal{N}(F) = \nvol(\conv(F \cup \{\boldzero\}))$.
\end{proof}

$\mathcal{N}(F)$ can be considered a generalization of the
adjacency polytope bound to facet subnetworks.
Moreover,
% While the root count $\mathcal{N}(F)$ for a facet system may not be directly useful by itself,
the sum of the root count for \emph{all} facet subsystems
gives us a root count for the whole system.
It is critical to note here that the factor of $n!$ in $\mathcal{N}(F)$ above is 
a part of the definition of normalized volume
and it is used to compensate for the shrinking of the Euclidean volume itself
(the volume of a minimum simplex with integral vertices is $1 / n!$ and goes to 0 as $n \to \infty$).
It does not indicate factorial growth in the root count for facet subsystems.
As \cref{sec:primitive} will demonstrate,
the smallest facet subsystem has exactly one solution regardless of the dimension.

\begin{theorem}\label{thm:root-count}
    Given a Kuramoto network, 
    let $\adjp_G$ be the adjacency polytope defined above.
    Then the total number of isolated complex synchronization configurations
    this network has is bounded above by the sum of generic complex
    root counts of the facet subsystems,
    that is, it is bounded by
    \begin{equation*}
        \sum_{F \in \mathcal{F}(\adjp_G)} \mathcal{N}(F)
    \end{equation*}
    where $\mathcal{F}(\adjp_G)$ is the set of facets of the polytope $\adjp_G$.
\end{theorem}

\begin{proof}
    $\boldzero$ is an interior point of $\adjp_G$.
    Thus the collection of pyramids 
    \[
        \{ 
            \conv(F \cup \{\boldzero\}) \mid
            F \in \mathcal{F}(\adjp_G)
        \}
    \]
    form a subdivision of $\adjp_G$.
    The normalized volume of $\adjp_G$ is the sum of the 
    normalized volume of these pyramid.
    That is,
    \[
        \nvol(\adjp_G) =
        \sum_{F \in \mathcal{F}(\adjp_G)} \nvol( \conv(F \cup \{\boldzero\}) )
        \qedhere
    \]
\end{proof}

When this result is combined with \cref{thm:facet-network},
it is plausible that the adjacency polytope bound $\nvol(\adjp_G)$
can be computed simply by listing all possible facet subnetworks
and computing the adjacency polytope bound for each of such simpler subnetworks.
This process is potentially easier than the $\#P$ problem of volume computation in general.
As we demonstrate in the examples in section~\ref{sec:example},
this can be done for certain classes of networks.

\subsection{Degeneration into facet subsystems}
The true value of the facet subsystems lies in their roles
as destinations for a deformation of the synchronization equations.
That is, we can form a continuous deformation (in the sense of homotopy) 
of the unmixed synchronization equations that can degenerate into simpler facet subsystems
and reduce the problem of finding synchronization configurations
to the problem of solving each individual facet subsystem.
This can be done through a specialized homotopy.

Consider the homotopy $H(\boldx,t) = (H_1,\dots,H_k)$ given by
\begin{equation}
    H_k = 
    \frac{c_k}{t} -
    \sum_{(i,j) \in \mathcal{E}(G)} a_{ijk} 
    \left(
        \frac{x_i}{x_j} -
        \frac{x_j}{x_i}
    \right)
    \;\text{for } k=1,\dots,n.
\end{equation}
The equation $H(\boldx,1) = \boldzero$ is exactly the 
unmixed synchronization system~\eqref{equ:kuramoto-unmixed}.
Since $H(\boldx,t)$ is smooth in $t$ for $t \in (0,1]$,
as $t$ varies between 0 and 1,
the equation $H(\boldx,t) = \boldzero$ represents a continuous deformation
of~\eqref{equ:kuramoto-unmixed}.
With proper choices of coefficients, the solutions of $H(\boldx,t) = \boldzero$ 
also move smoothly, as $t$ varies within $(0,1]$,
forming smooth \emph{solution paths}.
As $t \to 0$, however, this smooth deformation breaks down,
and the equation $H(\boldx,t) = \boldzero$ degenerates into
the facet subsystems.

\begin{theorem}\label{thm:homotopy}
    Given a Kuramoto network $(G,K,\omega)$ with generic
    coupling strengths and natural frequencies,
    % consider the homotopy $H(\boldx,t)$ defined above.
    the solution set of the system of equations $H(\boldx,t) = \boldzero$ 
    contains a finite number of smooth curves parameterized by $t$
    such that:
    \begin{enumerate}
        \item 
            The set of limit points of these curves as $t \to 1$
            contains all isolated complex synchronization configurations
            of this network; and
        \item
            The set of limit points of these curves as $t \to 0$
            reaches all complex synchronization configurations
            of the facet subsystems at ``toric infinity''
            in the sense that
            for each of these complex solutions $\boldy = (y_1,\dots,y_n)$,
            there exists a curve defined by $H(\boldx,t) = \boldzero$,
            along which $x_i(t) = y_i t^{\alpha_i} + o(t)$ for some $\alpha_i \in \Q$
            and $t$ sufficiently close to 0.
    \end{enumerate}
\end{theorem}

    In this case, ``generic'' coupling strength and natural frequencies
    can be understood as coefficients with arbitrarily small perturbation.
    Given any set of $K = [k_{ij}]$, $\omega = (\omega_1,\dots,\omega_n)$,
    and a threshold $\epsilon > 0$, there exists a new set of coefficients
    no more than $\epsilon$-distance away from $(K,\omega)$ 
    for which the above statement holds true.
    This can also be interpreted from a probabilistic point of view:
    if $K$ and $\omega$ are selected at random,
    then the above statement holds true with probability one.
    
    ``Toric infinity'' is a term from complex algebraic geometry that is used 
    to describe space outside $(\C \setminus \{0\})^n$ but inside a certain closure.
    In this context, as $t \to 0$, some coordinates in $\boldx = (x_1,\dots,x_n)$
    either goes to 0 or infinity.
    In either case, the point $\boldx$ escapes $(\C \setminus \{0\})^n$,
    known as the ``algebraic torus''.

\begin{proof}
    (1) This statement describes a direct consequence of the parameter homotopy method
    \cite{MorganSommese1989Coefficient}.
    If we consider the unmixed synchronization system~\eqref{equ:kuramoto-unmixed}
    as a family $\boldf(\boldx; \boldc)$ parameterized by 
    constant terms $\boldc = (c_1,\dots,c_n)$,
    then the parameter $\boldc / t$ remains generic
    for generic choices of $\boldc$ and every $t \in (0,1]$.
    According to the theory of parameter homotopy~\cite{MorganSommese1989Coefficient}, 
    the solution set of $H(\boldx,t) = \boldzero$ in $\C^n \times(0,1]$
    consists of smooth paths parameterized by $t$ whose limit points as $t \to 1$
    include all solutions of $H(\boldx,1) \equiv \boldf (\boldx; \boldc) = \boldzero$.
    
    (2) Let $C \subset \C^n \times (0,1]$ be a curve defined by 
    $H(\boldx,t) = \boldzero$. 
    By the previous part, $C$ can be expressed as a path $\boldx(t)$, 
    smoothly and analytically parameterized by $t$
    % (i.e., $\boldx$ is differentiable in $t$ for $t \in (0,1]$
    % and holomorphic in $t$, as a complex variable for $t$ in a sufficiently small
    % punctured disk centered at 0) 
    for $t \ne 0$.
    Although the smoothness breaks down at $t=0$,
    according to the theory of polyhedral homotopy method~\cite{HuberSturmfels1995Polyhedral}
    as well as the Newton-Puiseux Theorem,
    there exists a smooth function $\boldy(t) = (y_1,\dots,y_n)$ 
    and a vector $\boldalpha = (\alpha_1,\dots,\alpha_n) \in \Q^n$
    such that $\boldy(t)$ has a limit point in $\C^n$ as $t \to 0^+$ and
    \[
        \boldx(t) = \boldy(t) \cdot t^{\boldalpha} =
        (y_1(t) \, t^{\alpha_1}, \cdots, y_n(t) \, t^{\alpha_n})
    \]
    for sufficiently small positive $t$ values.
    Moreover, $\hat{\boldalpha} = (\boldalpha,1) \in \Q^{n+1}$ must be 
    an inner normal vector of a facet of the Newton polytope $\newt(H)$ of $H(\boldx,t)$ 
    which is the pyramid formed by $\adjp_G \subset \R^n \subset \R^{n+1}$ 
    together with $(\boldzero,-1) \in \R^{n+1}$, i.e.,
    $\hat{\boldalpha}$ is an inner normal vector of
    \begin{equation*}
        \newt(H) = \conv \left(
            \{ (\boldx,0) \mid \boldx \in \adjp_G \}
            \cup
            \{ (\boldzero,-1) \}
        \right).
    \end{equation*}
    In this case, $\boldalpha$ is an inner normal vector of a facet $F$ of $\adjp_G$.
    Consequently, the initial form $\init_{\alpha} (H)$ is exactly 
    the facet subsystem induced by $F$ as displayed in \eqref{equ:facet},
    % Then for sufficiently small $t \in (0,1]$,
    % there exists a smooth function $\boldy(t)$ such that
    % the solution path in question defined by $H(\boldx,t) = \boldzero$
    % can be parametrized by
    % \[
    %     \boldx(t) = \boldy(t) \cdot t^{\boldalpha} =
    %     (y_1(t) t^{\alpha_1}, \cdots, y_n t^{\alpha_n}),
    % \]
    and $\lim_{t \to 0^+} \boldy$ exists and is a solution to this system.
\end{proof}

The above theorem reduces the problem of solving a complicated synchronization system
\eqref{equ:kuramoto-unmixed} to the problem of solving simpler facet subsystems:
Once the solutions to each facet subsystem are found,
they become the starting points of the smooth paths defined by $H(\boldx,t) = \boldzero$.
Then, standard numerical ``path tracking'' algorithms
\cite{allgower_survey_1981,ChenLi2015Homotopy,davidenko_new_1953,SommeseWampler2005Numerical}
can be applied to trace these smooth paths 
and reach the desired synchronization configurations
defined by $H(\boldx,1) = \boldf(\boldx) = \boldzero$.
    This homotopy method has the advantage of being \emph{pleasantly parallel}
    in the sense that every configuration can be found independently.

\begin{remark}[Connection to polyhedral homotopy]\label{rmk:polyhedral}
    The homotopy constructed from the above theorem can be viewed as a highly specialized
    polyhedral homotopy method~\cite{HuberSturmfels1995Polyhedral} with 
    a special lifting function that takes the value of $-1$ on constant terms 
    and 0 everywhere else.
    Moreover, instead of degenerating into binomial systems that depend on lifting,
    our homotopy degenerates into facet subsystems that may be found 
    through an examination of the network topology.
%     For certain types of network topology, 
%     we can establish even stronger correspondence between the 
%     synchronization configurations of the original network
%     and the solutions of the facet systems.
%     % Indeed, if we assume the original network has generic coupling strength
%     % and natural frequency, then there is a one-to-one correspondence
%     % between the synchronization configurations and the disjoint union
%     % of solutions of all the facet systems. 
%     For instance, 
%     This correspondence as well as the underlying homotopy is
%     developed in the forth coming paper by the author
%     and Robert Davis.
\end{remark}

\begin{remark}[Pushing Kuramoto networks to the limit]
    As noted in the end of \cref{sec:example},
    the degeneration of a given network into facet subnetworks
    can be thought of as pushing the network to its breaking point.
    Since the construction of the homotopy $H$ 
    replaces constant terms in~\eqref{equ:kuramoto-unmixed} by $c_k / t$,
    decreasing $t$ from 1 to 0 has the effect of amplifying the differences
    in the natural frequencies.
    We expect this to make synchronization more difficult
    resulting in some real synchronization configurations
    degenerating into non-real complex configurations.
    As $t \to 0$, the differences in the natural frequencies are amplified to infinity,
    at which point the network is not even able to support 
    complex synchronization configurations.
    The network breaks apart, 
    and the facet subnetworks describe the pieces in the sense of limits.
\end{remark}

%=============================================================================
\section{Primitive facet subnetworks}\label{sec:primitive}

The general facet decomposition scheme outlined above decomposes a
Kuramoto network into a collection of smaller directed acyclic networks
involving the original oscillators --- the facet subnetworks.
Among these subnetworks, the basic building blocks are ``primitive'' subnetworks 
which are, in a sense, the smallest Kuramoto networks that could have a 
frequency synchronization configuration.

\begin{definition}
    In a given Kuramoto network,
    a facet subnetwork is called \textbf{primitive} if 
    the underlying graph $G$ is a weakly connected acyclic graph that contains 
    all the vertices ${0,\dots,n}$ and exactly $n = N-1$ directed edges.
\end{definition}

Since such a primitive subnetwork contains exactly $N$ node and $N-1$ directed edges, 
removing any edge creates disconnected components.
It is thus a minimum facet subnetwork.
Yet, as we demonstrate in this section, 
it is sophisticated enough to have a complex synchronization configuration,
and this configuration \emph{must be unique}.
Equally as important, this unique synchronization configuration can be 
\emph{computed quickly and easily}
using only $O(n)$ complex multiplication and division, 
and no additional memory is needed.
Once solved, the solutions of these primitive facet subsystems 
can be used to bootstrap the homotopy continuation procedure described in the previous section.

\begin{theorem}
    For $N > 1$,
    the facet subsystem~\eqref{equ:facet} supported by a primitive facet subnetwork 
    has a \emph{unique} complex solution for which $x_i \ne 0$ for each 
    $i=1,\dots,n$.
\end{theorem}

\begin{proof}
    Since $G$ has exactly $n$ directed edges, the corresponding 
    synchronization system has exactly $n$ nonconstant terms
    in each equation.
    Under the assumption of generic coefficients,
    we can reduce the system to
    \begin{equation}\label{equ:prim-red}
        c_k' = a_{k}' x_{i_k} / x_{j_k} \quad 
        \text{for } k = 1,\dots,n.
    \end{equation}
    via Gaussian elimination with no cancellation of the terms
    with $\{c_k'\}$ and $\{a_k'\}$ representing the resulting coefficients.
        The invertibility of the Gaussian elimination process
        ensures that all $a_k' \ne 0$.
    Having exactly $n = N-1$ edges and $N$ nodes also ensures that
    the underlying undirected graph is an undirected tree.
    Therefore, for each $i = 1,\dots,n$, there is a unique path
    between node $0$ and node $i$ through the nodes 
    $i_0,i_1,\dots,i_m$ with $i_0 = 0$ and $i_m = i$,
    where $m$ is the length of this path.
    \[
        \begin{aligned}
            c_1' &= a_{1}' x_{i_1}^{\pm 1} x_{i_0}^{\mp 1} \\
            c_2' &= a_{2}' x_{i_2}^{\pm 1} x_{i_1}^{\mp 1} \\
                &\vdots \\
            c_m' &= a_{m}' x_{i_m}^{\pm 1} x_{i_{m-1}}^{\mp 1}. \\
        \end{aligned}
    \]
    We should recall that node $0$ is the reference node,
    i.e., $x_{i_0} = x_0 = 1$.
    Solving this system via a forward substitution process, 
    we can determine the value of $x_{i_m} = x_{i} \ne 0$.
    Since the path between node $i$ and node $0$ is unique,
    there is no possibility of inconsistency in this process,
    and hence the value of each $x_i$ is uniquely determined.
    We can therefore conclude that a solution to the facet system exists,
    and it is unique.
\end{proof}

Aside from offering starting points for the homotopy continuation method
that will identify all complex synchronization configurations,
the existence of primitive facet subnetworks also directly contributes to
an estimation of the total number of complex synchronization configurations.
Since each primitive facet subsystem has a unique solution,
by ~\cref{thm:root-count}, the total number of complex synchronization configurations
is greater than or equal to the number of primitive facet subnetworks.

\begin{remark}[Computational complexity]
    From a computational view point,
    the true value of primitive subnetworks lies in
    the ease with which it can be solved.
    The above proof is constructive and suggests a practical algorithm
    for solving the reduced facet subsystem~\eqref{equ:prim-red}
    that only requires $O(n)$ complex multiplication and division, and no additional memory
    once the initial elimination step is done.
    This is of great importance in the homotopy method for solving
    the synchronization system described in the previous section,
    since the solutions to the facet subsystems are the starting points
    of solution paths that can lead to the desired synchronization configurations
    of the original Kuramoto network.
\end{remark}

\begin{remark}[Toric interpretation]
    The above theorem can also be interpreted 
    in the language of toric algebraic geometry as follows.
    The facet subsystem supported by a primitive facet subnetwork
    can be reduced to an equivalent binomial system
    whose associated exponent matrix
    (the matrix whose columns are the exponent vectors for the nonconstant terms)
    is a unimodular matrix.
    Consequently, it defines a (normal) irreducible toric variety of zero dimension
    which must be a single point.
    From this toric geometry view point, 
    we can also see that primitive subnetworks
    are the smallest building blocks in the proposed decomposition scheme.
\end{remark}

\begin{remark}[Real synchronization configurations]
    Recall that algebraic synchronization equations are obtained through 
    a change of variables $x_i = e^{\imag \theta_i}$ for $i=1,\dots,n$,
    so real solutions to the original Kuramoto equations corresponds to
    complex solutions to the algebraic system with
    $|x_i| = 1$ for $i=1,\dots,n$  
    (i.e., the real torus solutions in $(S^1)^n = (\{ z \in \C \mid |z| = 1\})^n$).
    Since a primitive facet subsystem \eqref{equ:prim-red} 
    can be solved using only complex multiplication and division,
    we can ensure $|x_i| = 1$ for $i=1,\dots,n$ 
    by choosing coefficients that have unit absolute values.
    In this case, the unique solution corresponds to a real solution.
    That is, there is always a way to choose the coefficients
    so that the unique solution of a given primitive facet subsystem 
    gives rise to a real solution.
    Under the interpretation of facet subnetworks as
    acyclic directed Kuramoto networks with one-way interaction (\cref{fig:dir-edge}), 
    this real solution can be viewed as a generalized real synchronization configuration
    that represents the limit behavior of the Kuramoto network
    as the differences in the natural frequencies are amplified to infinity.
\end{remark}

As the next section demonstrates,
primitive subnetworks appear naturally in the facet decomposition
of many types of Kuramoto networks, e.g., trees, cycles, and chordal graphs.
In some cases, all facet subnetworks are primitive.

%=============================================================================
\section{Examples}\label{sec:example}

In this section, we illustrate the facet decomposition scheme using concrete examples.
In all examples, the facets of the adjacency polytopes are computed using the
well test software package \tech{Polymake},\cite{polymake}
which can compute the list of all facets efficiently even for complicated convex polytopes
in very high dimension.

\begin{figure}[ht]
    \centering
    \subfloat[A tree]{
        \begin{tikzpicture}[scale=0.33]
            \node[draw,circle] (0) at ( 0.00, 3.00) {$0$};
            \node[draw,circle] (1) at (-2.00, 0.93) {$1$};
            \node[draw,circle] (2) at ( 0.00,-2.43) {$2$};
            \node[draw,circle] (3) at (-4.00,-2.43) {$3$};
            \node[draw,circle] (4) at ( 2.00, 0.93) {$4$};
            \path[draw,thick]  (1) -- (0);
            \path[draw,thick]  (2) -- (1);
            \path[draw,thick]  (1) -- (3);
            \path[draw,thick]  (4) -- (0);
        \end{tikzpicture}\label{fig:tree5-orig}}%
    \hspace{2ex}
    \subfloat[A cycle]{
        \begin{tikzpicture}[scale=0.33]
        \node[draw,circle] (0) at (1.83697019872e-16,3.0) {$0$};
        \node[draw,circle] (1) at (-2.85316954889,0.927050983125) {$1$};
        \node[draw,circle] (2) at (-1.76335575688,-2.42705098312) {$2$};
        \node[draw,circle] (3) at (1.76335575688,-2.42705098312) {$3$};
        \node[draw,circle] (4) at (2.85316954889,0.927050983125) {$4$};
        \path[draw,thick] (0)  --  (1) ;
        \path[draw,thick] (0)  --  (4) ;
        \path[draw,thick] (2)  --  (1) ;
        \path[draw,thick] (2)  --  (3) ;
        \path[draw,thick] (3)  --  (4) ;
        \end{tikzpicture}\label{fig:cycle5}}%
    \hspace{2ex}
    \subfloat[A chordal network]{
        \begin{tikzpicture}[scale=0.31]
        \node[draw,circle] (0) at ( 0, 3) {$0$};
        \node[draw,circle] (1) at (-3, 0) {$1$};
        \node[draw,circle] (2) at ( 0,-3) {$2$};
        \node[draw,circle] (3) at ( 3, 0) {$3$};
        \path[draw,thick] (0)  --  (1) ;
        \path[draw,thick] (1)  --  (2) ;
        \path[draw,thick] (2)  --  (3) ;
        \path[draw,thick] (3)  --  (0) ;
        \path[draw,thick] (0)  --  (2) ;
        \end{tikzpicture}\label{fig:chord4}}%
    % \hspace{3ex}
    % \subfloat[A wheel graph]{
    %     \begin{tikzpicture}[scale=0.40]
    %     \node[draw,circle] (0) at ( 0, 0) {$0$};
    %     \node[draw,circle] (1) at (-3, 3) {$1$};
    %     \node[draw,circle] (2) at (-3,-3) {$2$};
    %     \node[draw,circle] (3) at ( 3,-3) {$3$};
    %     \node[draw,circle] (4) at ( 3, 3) {$4$};
    %     \path[draw,thick] (0)  --  (1) ;
    %     \path[draw,thick] (1)  --  (2) ;
    %     \path[draw,thick] (2)  --  (3) ;
    %     \path[draw,thick] (3)  --  (4) ;
    %     \path[draw,thick] (0)  --  (2) ;
    %     \path[draw,thick] (0)  --  (3) ;
    %     \path[draw,thick] (0)  --  (4) ;
    %     \path[draw,thick] (1)  --  (4) ;
    %     \end{tikzpicture}\label{fig:wheel5-orig}}%
    \caption{Three small networks}
    \label{fig:small}
\end{figure}
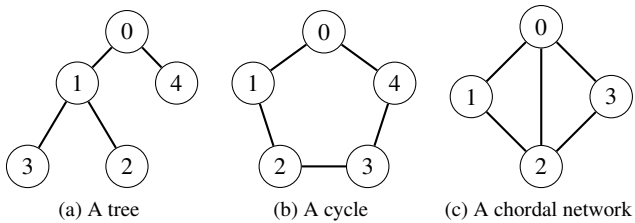
%-----------------------------------------------------------------------------
\subsection{A tree network}

% \begin{wrapfigure}[12]{r}{0.30\textwidth}
%     \centering
%     \caption{A tree of 5 nodes}
%     \label{fig:tree5-orig}
% \end{wrapfigure}
We start with the simplest types of networks --- trees.
For a tree graph containing five nodes, displayed in \Cref{fig:tree5-orig},
the corresponding adjacency polytope (\cref{def:ap}) has 16 facets
which produces 16 facet subnetworks, displayed in \Cref{fig:tree5}.
\emph{Every facet subnetwork is primitive in this case}.
Since each primitive facet subnetwork has 
a unique complex synchronization configuration,
we can conclude that the original tree network has at most 16
complex synchronization configurations.
This aligns with the well known result that on a tree network of $N$ oscillators,
there are at most $2^{N-1}$ complex synchronization configurations. 
    Adjacency polytope bounds for trees networks in general
    were analyzed in a closely related work\cite{ChenDavisMehta2018Counting}
    using convex geometry method.
    With the framework proposed in this paper,
    we can see that this bound actually has a topological origin
    --- an acyclic cover of a tree network.
In this case,
the topological constraints given in \cref{thm:facet-network}
actually determine the list of facet subnetworks.
That is, without actually computing the facets of the adjacency polytope,
one can easily enumerate all of the facet subnetworks
by listing the acyclic subnetworks.

Interestingly, this upper bound on the number of complex synchronization configurations 
coincide with the tightest possible upper bound 
on the number of \emph{real} configurations as
it is well known that there could also be as many as $2^{N-1}$ real synchronization configurations.
\cite{Baillieul1982}

\begin{figure}[h]
    \begin{tabular}{ccccc}
        \includegraphics[width=0.11\textwidth]{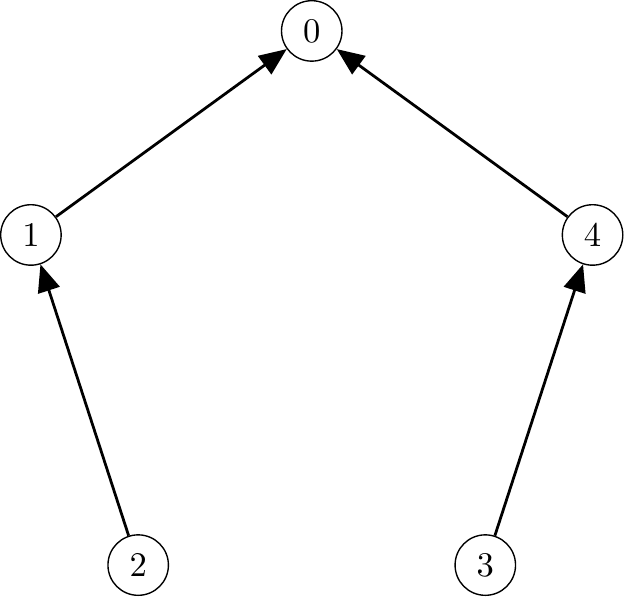} &
        \includegraphics[width=0.11\textwidth]{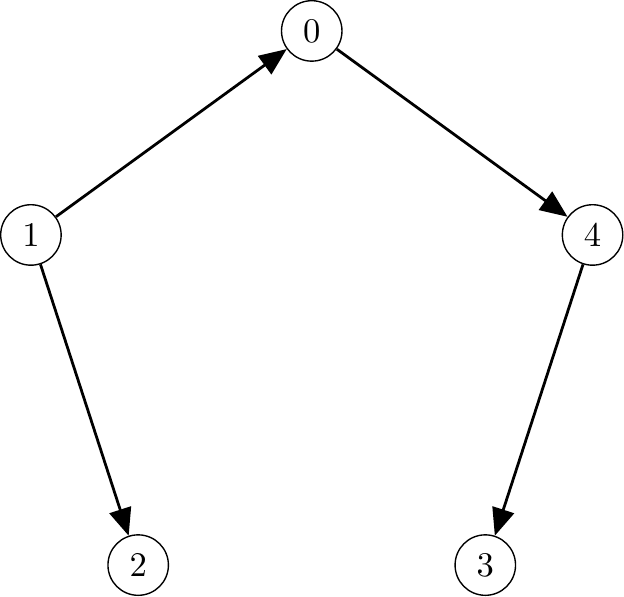} &
        \includegraphics[width=0.11\textwidth]{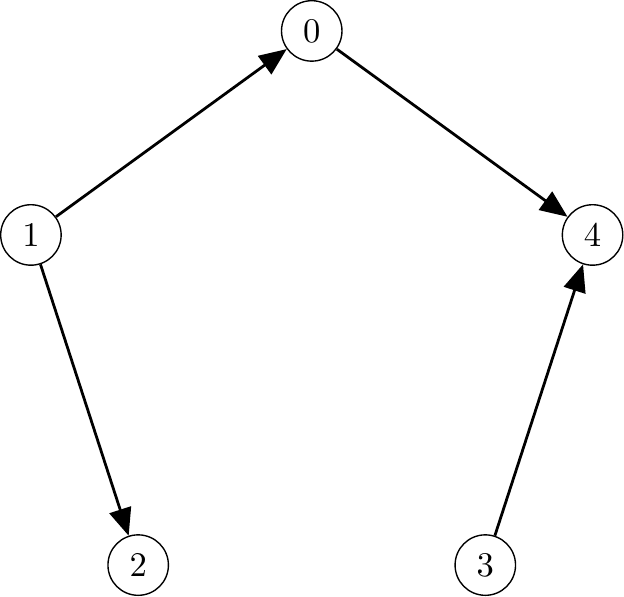} &
        \includegraphics[width=0.11\textwidth]{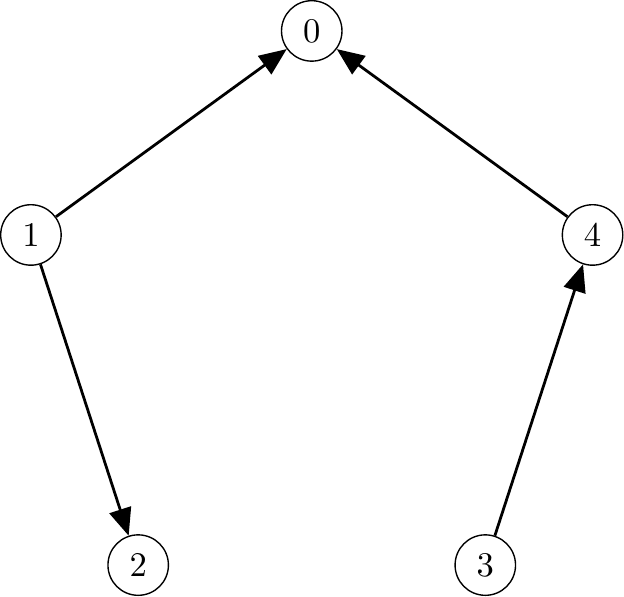} &
        \\[0.2ex]
        \includegraphics[width=0.11\textwidth]{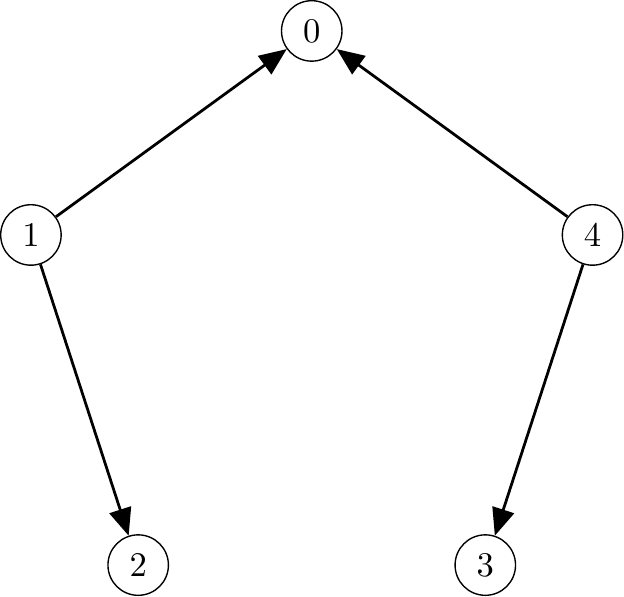} &
        \includegraphics[width=0.11\textwidth]{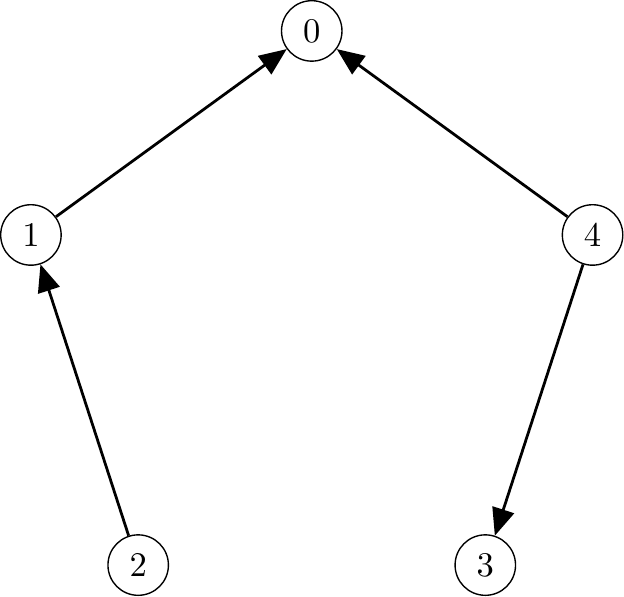} &
        \includegraphics[width=0.11\textwidth]{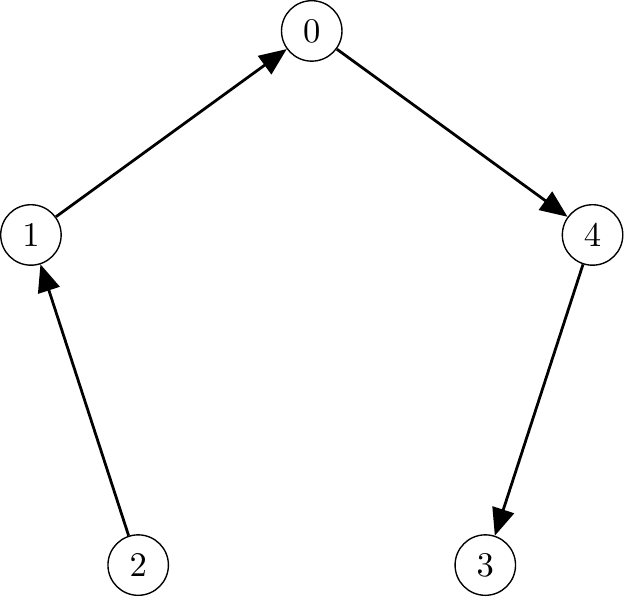} &
        \includegraphics[width=0.11\textwidth]{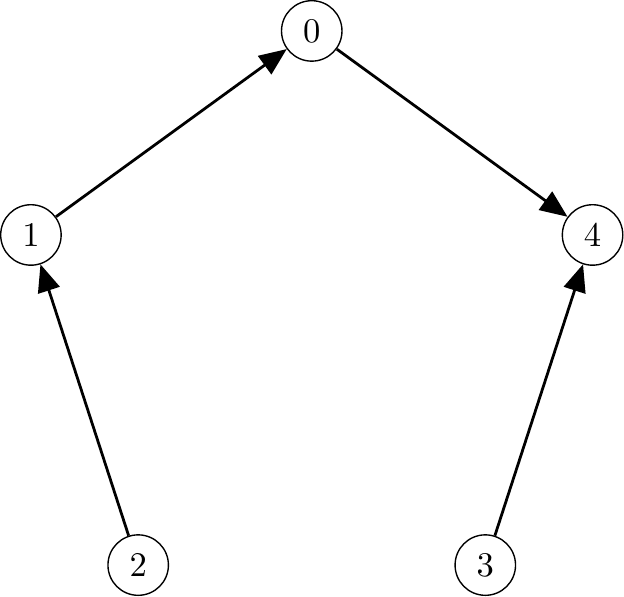} &
        \\[0.2ex]
        \includegraphics[width=0.11\textwidth]{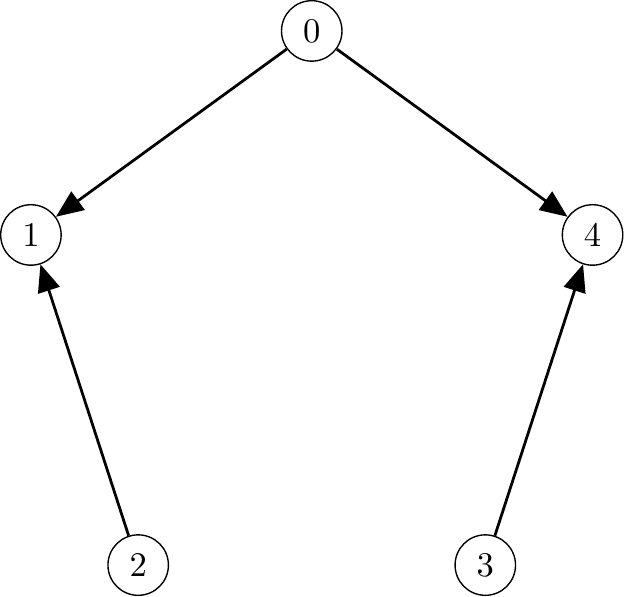} &
        \includegraphics[width=0.11\textwidth]{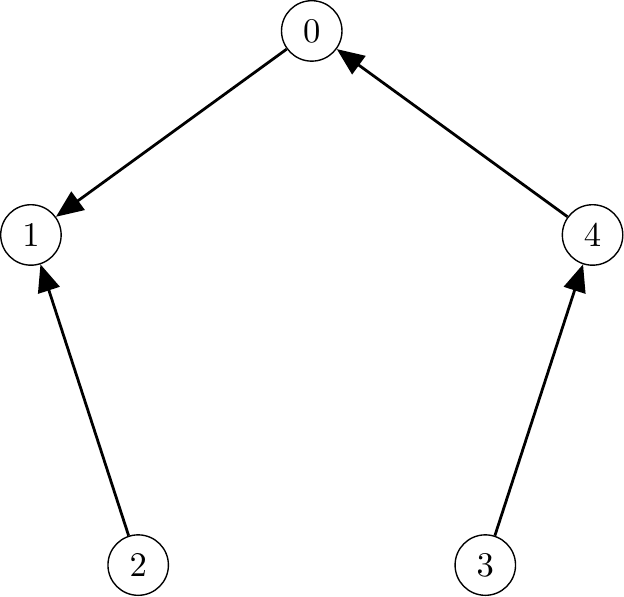} &
        \includegraphics[width=0.11\textwidth]{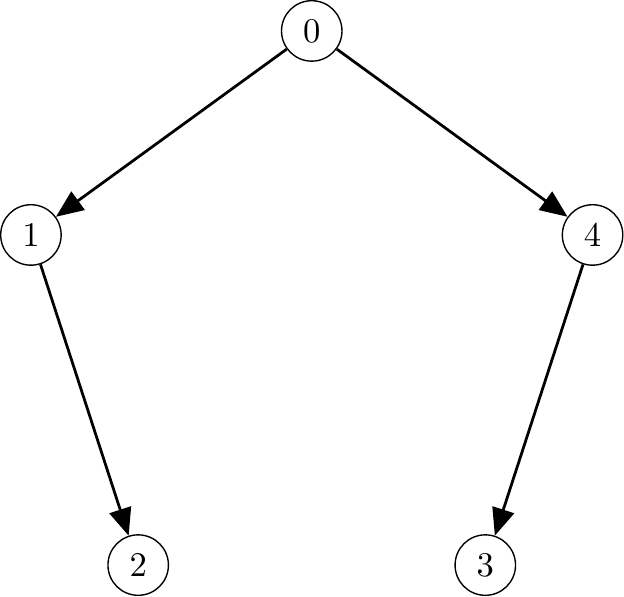} &
        \includegraphics[width=0.11\textwidth]{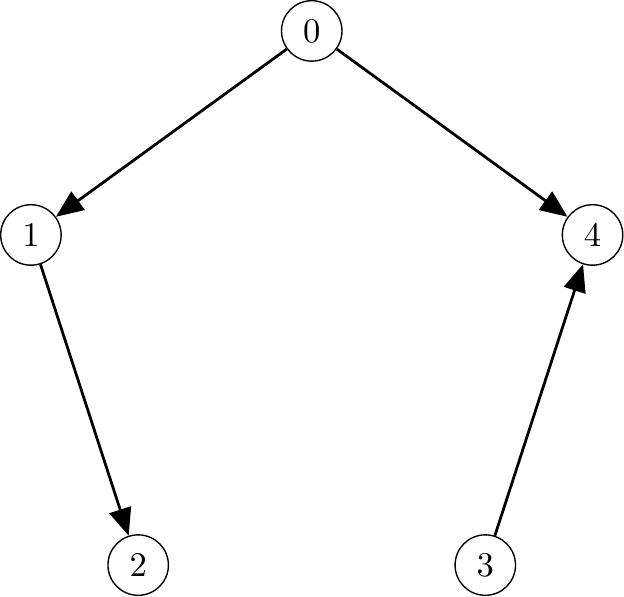} &
        \\[0.2ex]
        \includegraphics[width=0.11\textwidth]{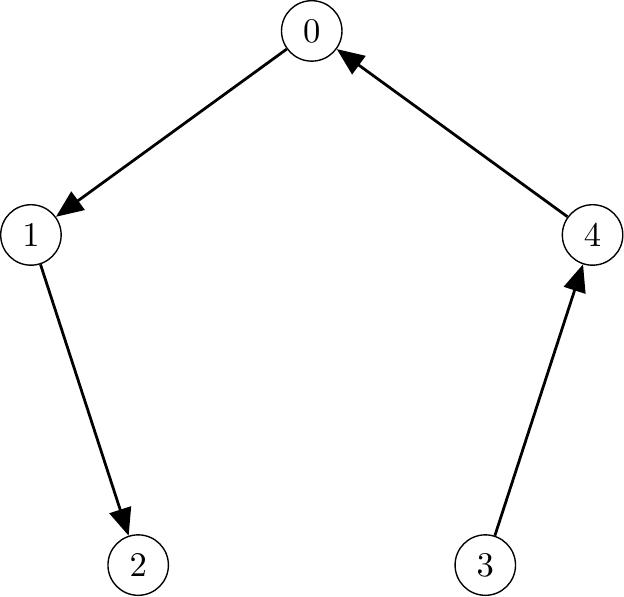} &
        \includegraphics[width=0.11\textwidth]{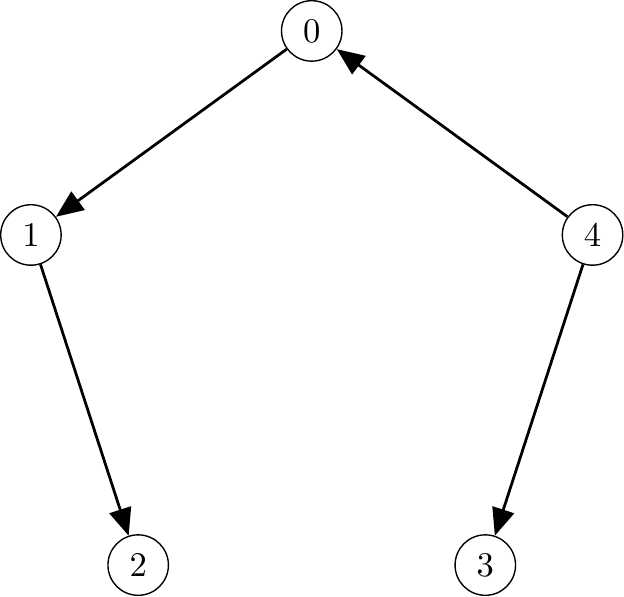} &
        \includegraphics[width=0.11\textwidth]{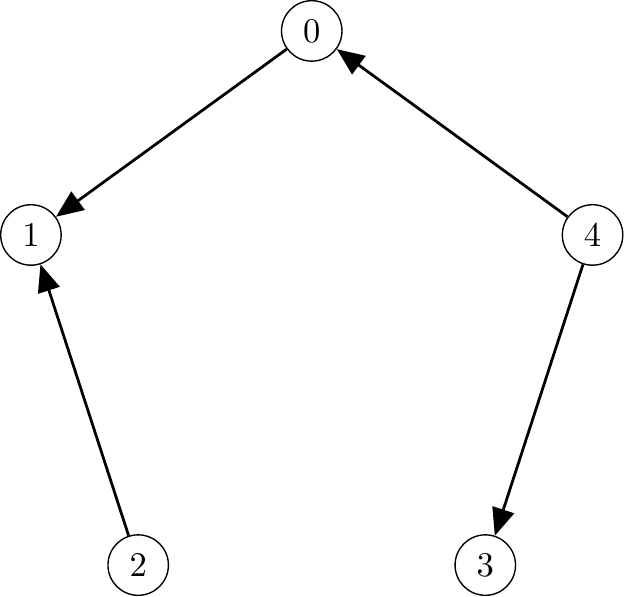} &
        \includegraphics[width=0.11\textwidth]{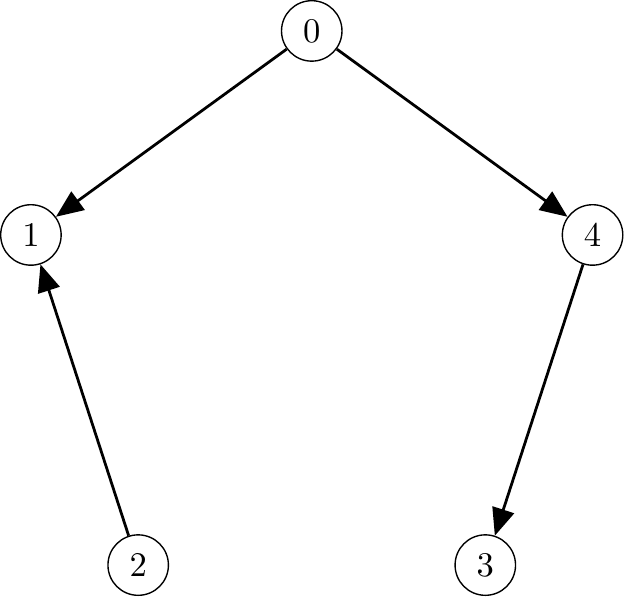} &
    \end{tabular}
    \caption{
        Facet subnetworks of a tree graph of five nodes,
        all of which are primitive
    }
    \label{fig:tree5}
\end{figure}

%-----------------------------------------------------------------------------
\subsection{A cycle network}

For a cycle network, there is a similar decomposition, but with more subnetworks.
This is consistent with the observation that cycle networks could have
more possible synchronization configurations than tree networks.
    Unlike the previous case,
    not all acyclic subgraphs give rise to facet subnetworks.
\Cref{fig:cycle5-dac} displays all 30 of the facet subnetworks from
a cycle graph of five nodes (\cref{fig:cycle5}). 
Each subnetwork is primitive and hence has a unique complex synchronization configuration.
    Consequently, a cycle network of five oscillators
    has at most 30 complex synchronization configurations.
    This number also coincide with the maximum number of 
    real synchronization configurations obtained by a recent study,\cite{Zachariah2018Distributions}
    i.e., all complex solutions could be real.
    
    The obvious symmetry shown in \Cref{fig:cycle5-dac} is not accidental.
    In addition to the transpose symmetries stated in \Cref{thm:facet-network}
    (for a facet subnetwork $G_F$, its transpose $G_F^\top$ is also a facet subnetwork),
    we can also see that the set of facet subnetworks 
    inherited the full range of cyclic symmetry from the original network.

\begin{figure}[ht]
    \begin{tabular}{cccccc}
        \includegraphics[width=0.09\textwidth]{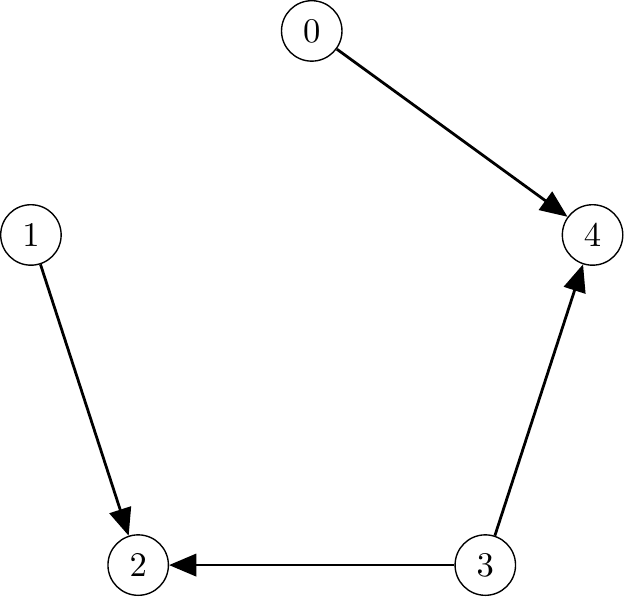} &
        \includegraphics[width=0.09\textwidth]{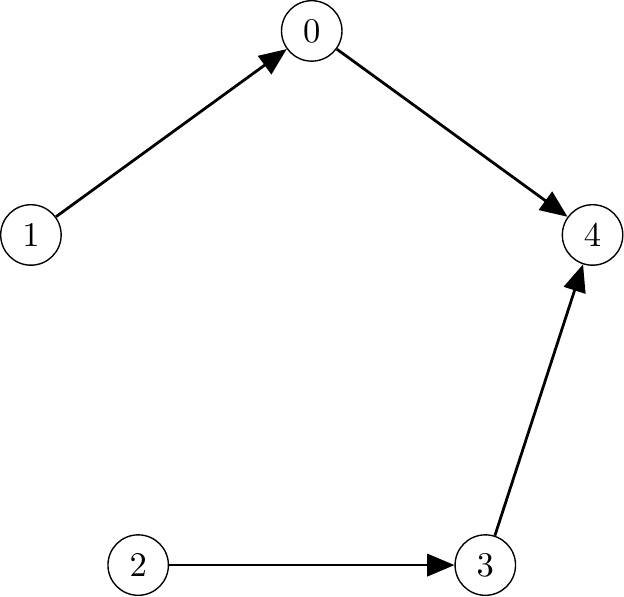} &
        \includegraphics[width=0.09\textwidth]{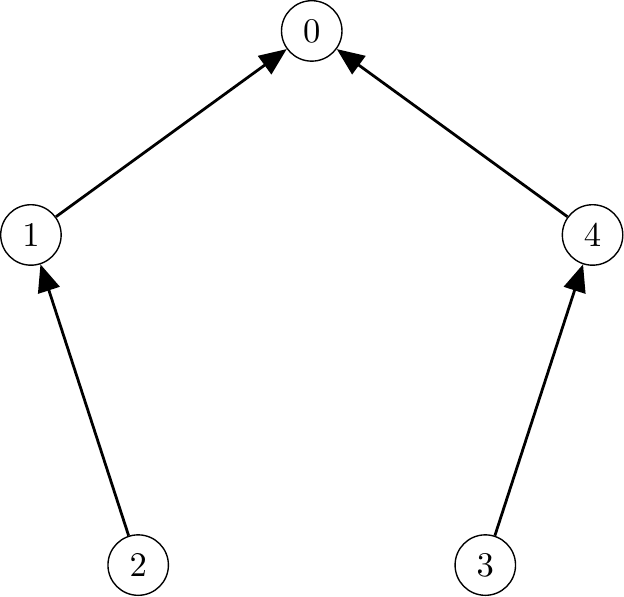} &
        \includegraphics[width=0.09\textwidth]{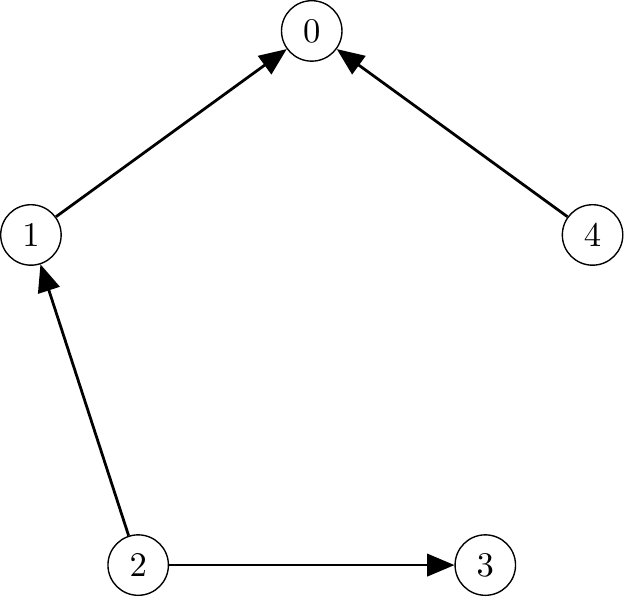} &
        \includegraphics[width=0.09\textwidth]{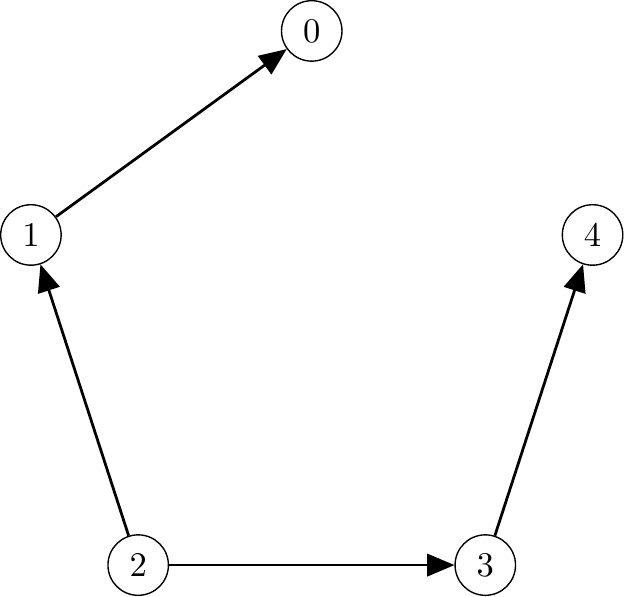} &
        \\[0.3ex]
        \includegraphics[width=0.09\textwidth]{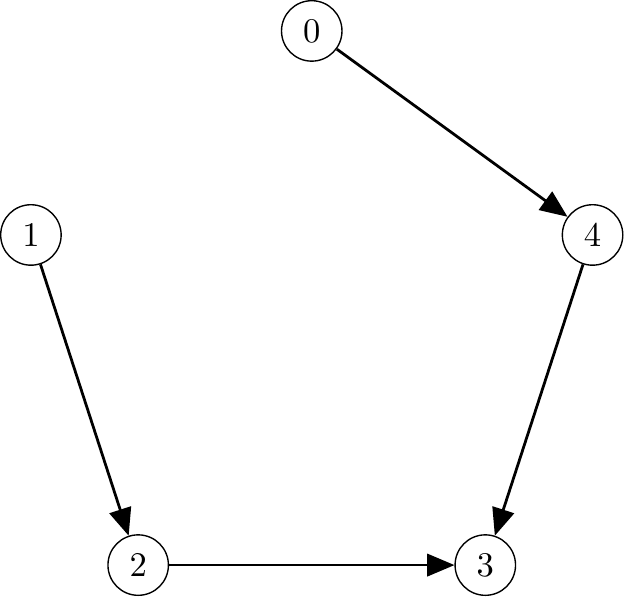} &
        \includegraphics[width=0.09\textwidth]{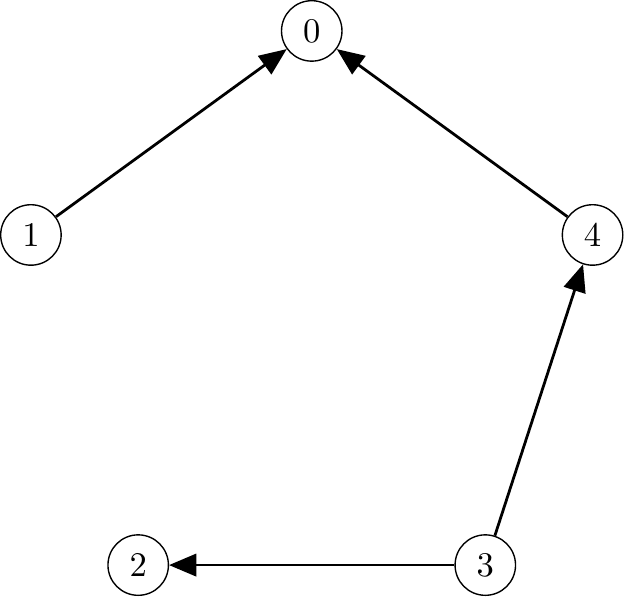} &
        \includegraphics[width=0.09\textwidth]{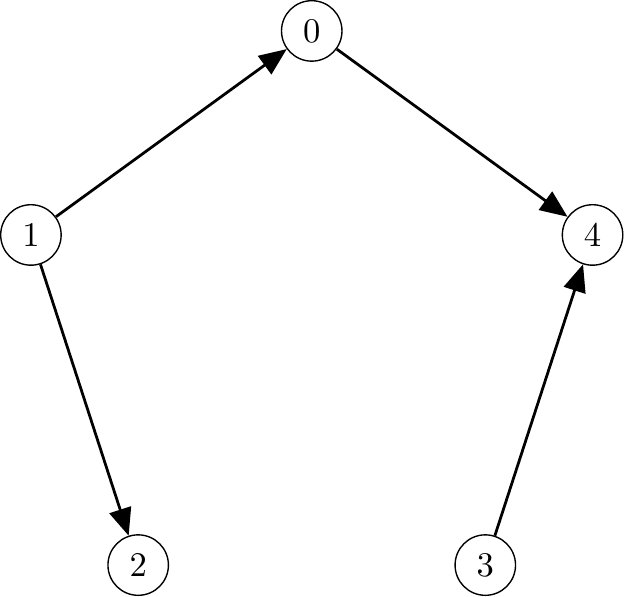} &
        \includegraphics[width=0.09\textwidth]{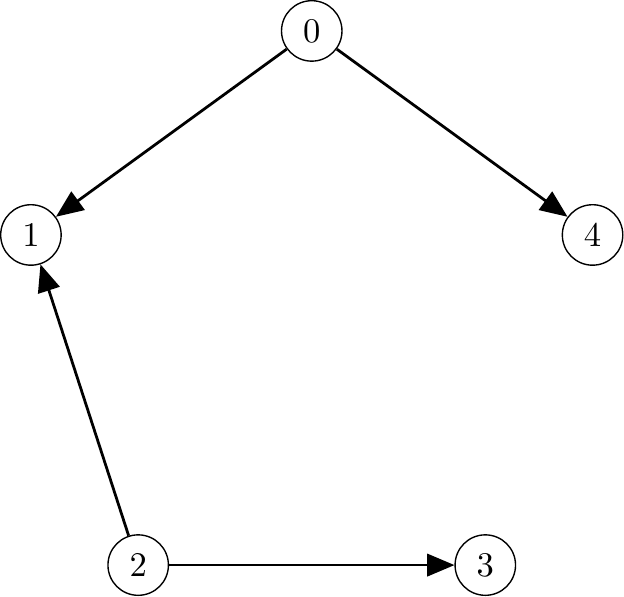} &
        \includegraphics[width=0.09\textwidth]{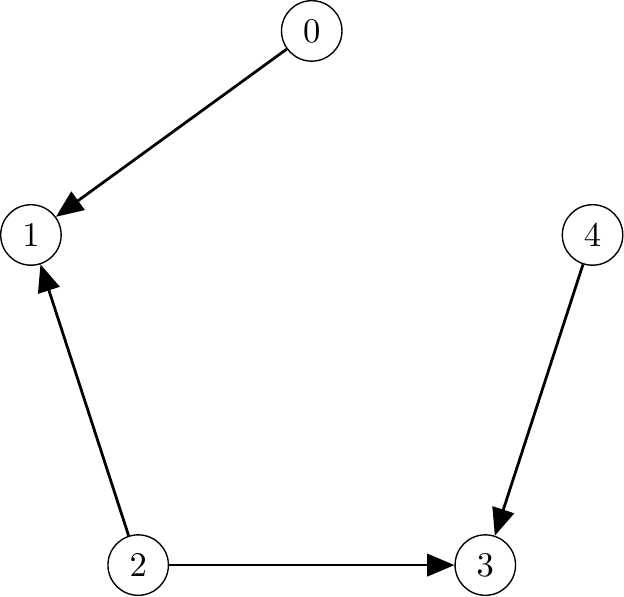} &
        \\[0.3ex]
        \includegraphics[width=0.09\textwidth]{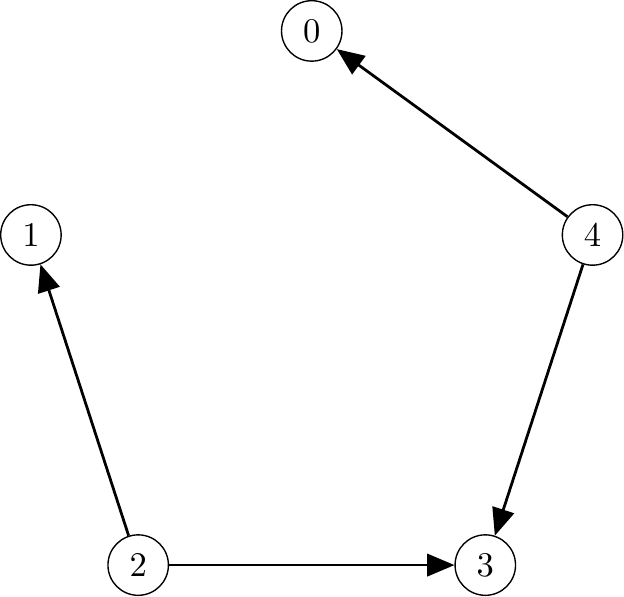} &
        \includegraphics[width=0.09\textwidth]{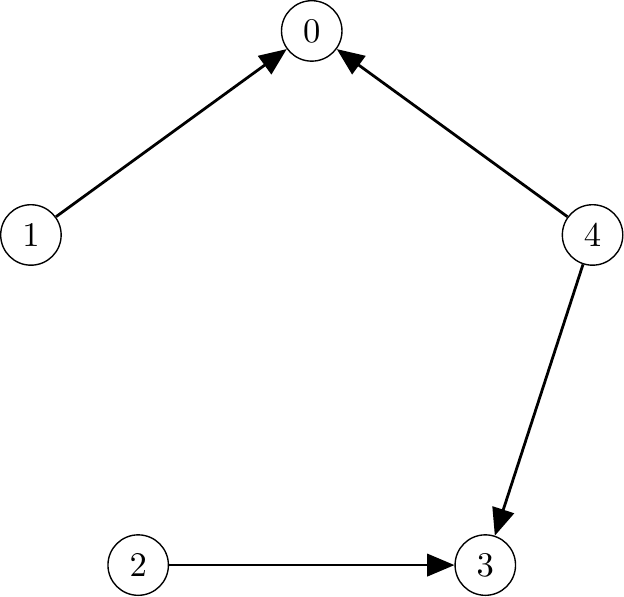} &
        \includegraphics[width=0.09\textwidth]{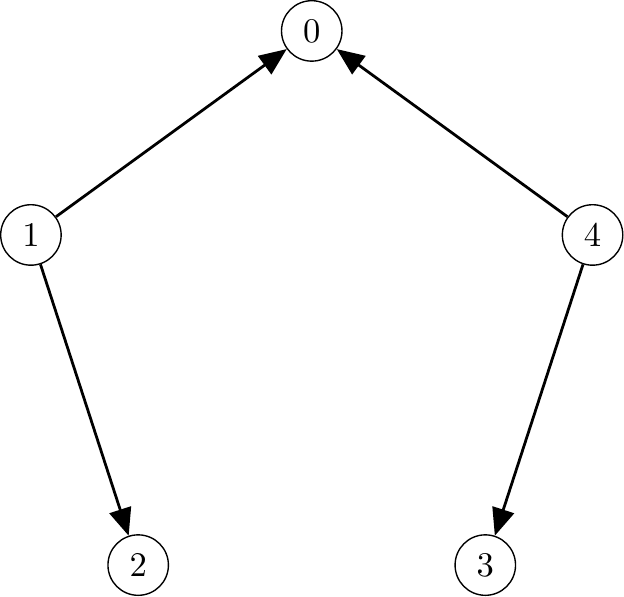} &
        \includegraphics[width=0.09\textwidth]{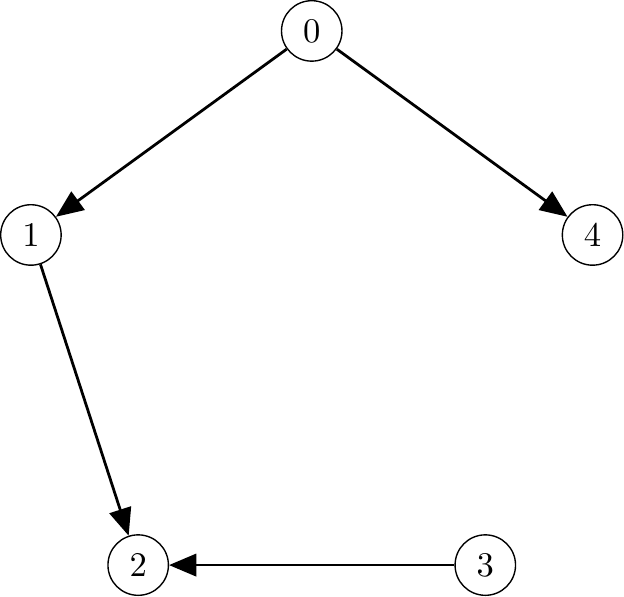} &
        \includegraphics[width=0.09\textwidth]{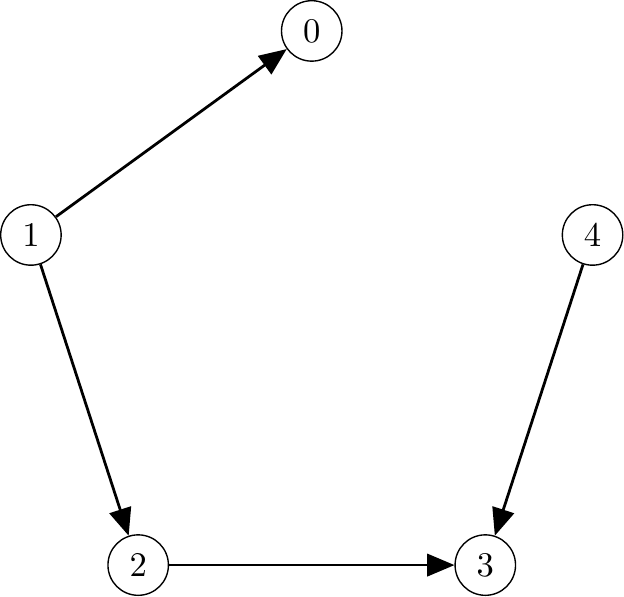} &
        
        \\[0.3ex]
        \includegraphics[width=0.09\textwidth]{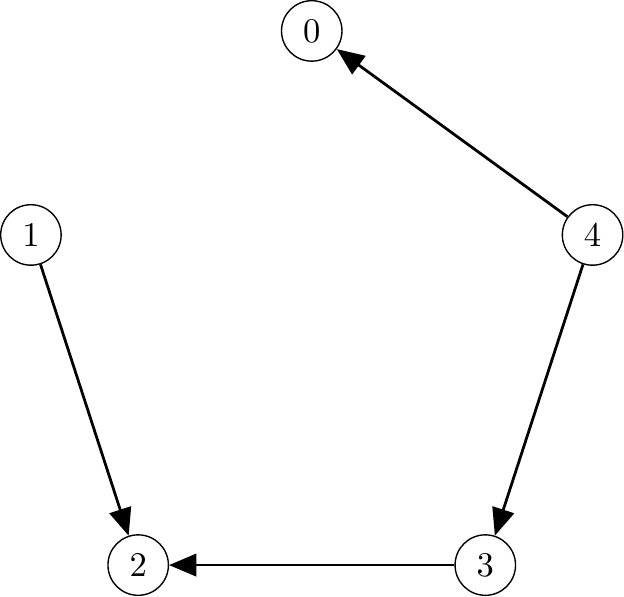} &
        \includegraphics[width=0.09\textwidth]{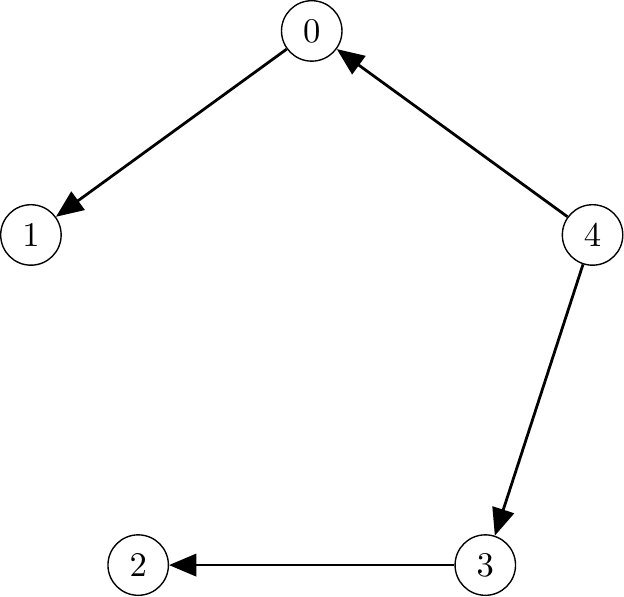} &
        \includegraphics[width=0.09\textwidth]{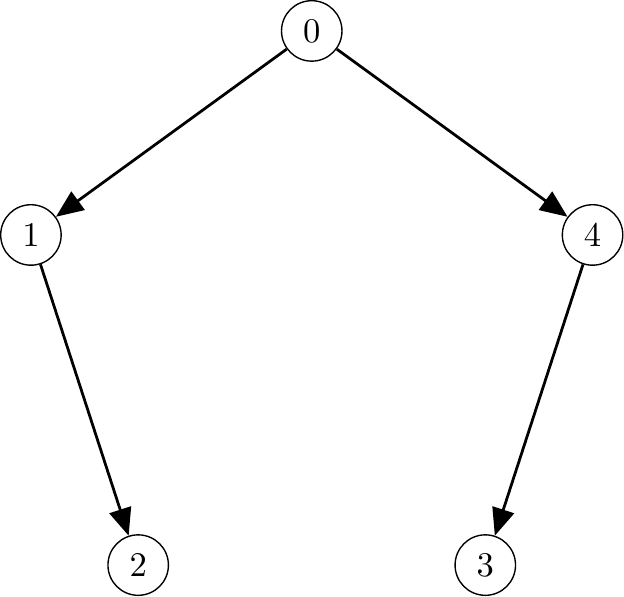} &
        \includegraphics[width=0.09\textwidth]{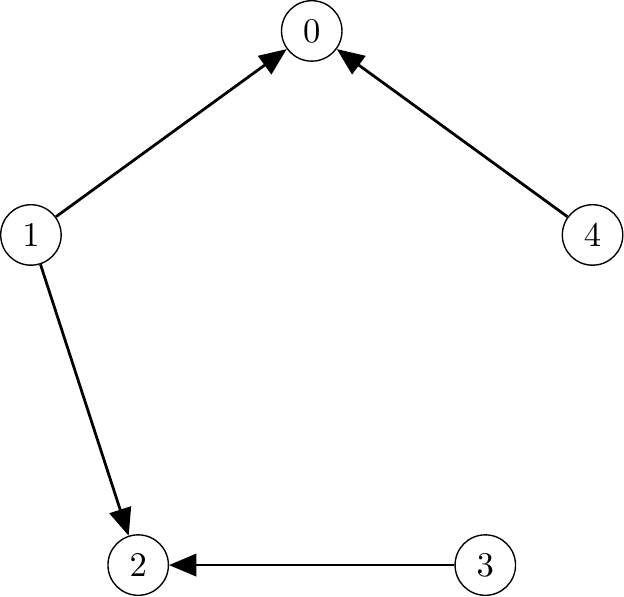} &
        \includegraphics[width=0.09\textwidth]{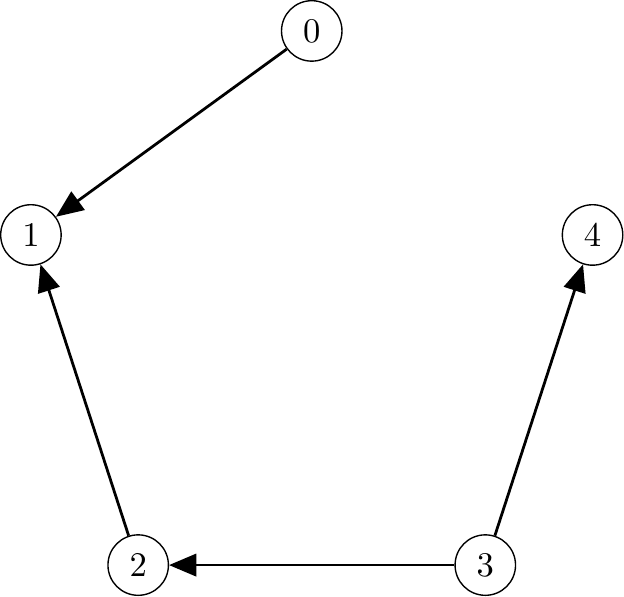} &
        
        \\[0.3ex]
        \includegraphics[width=0.09\textwidth]{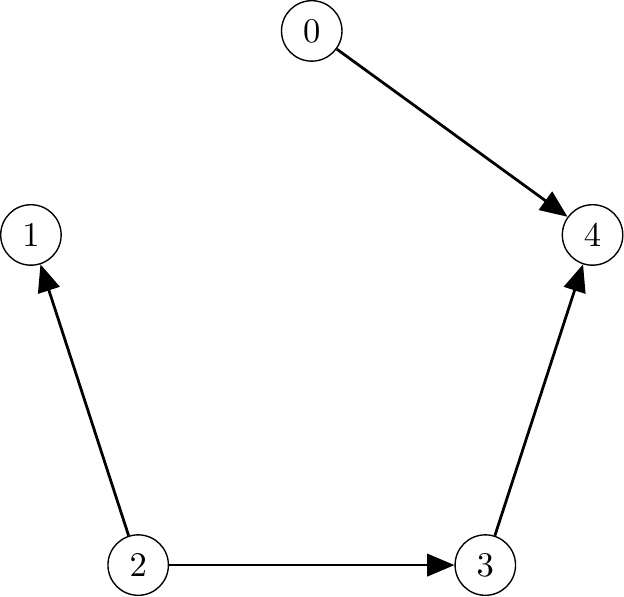} &
        \includegraphics[width=0.09\textwidth]{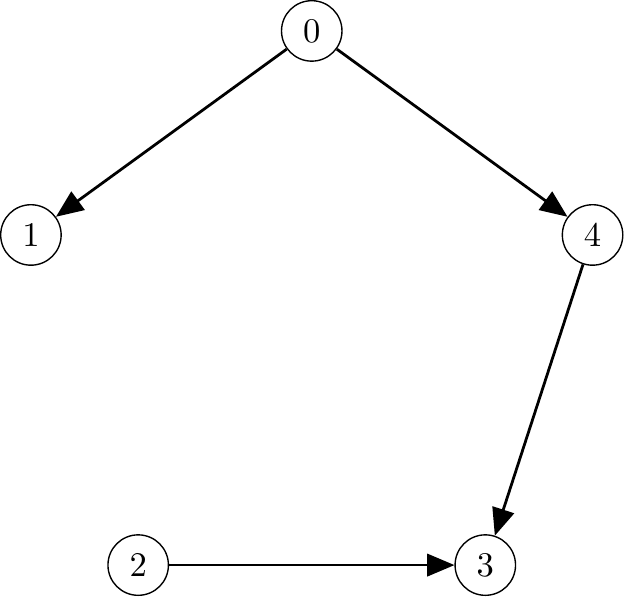} &
        \includegraphics[width=0.09\textwidth]{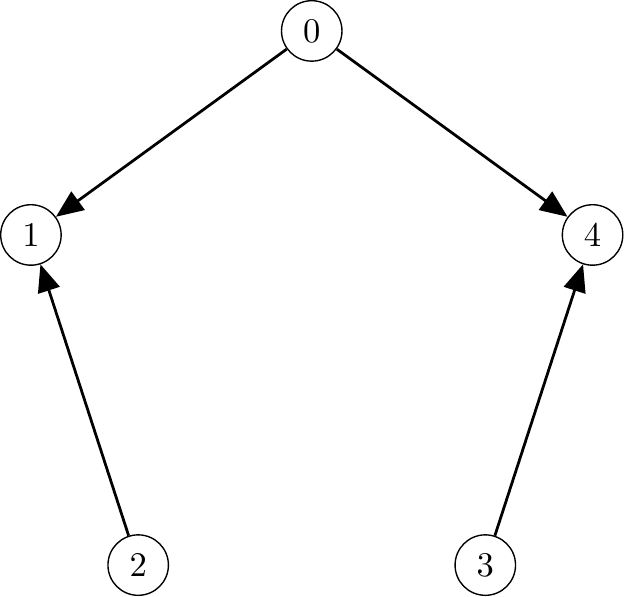} &
        \includegraphics[width=0.09\textwidth]{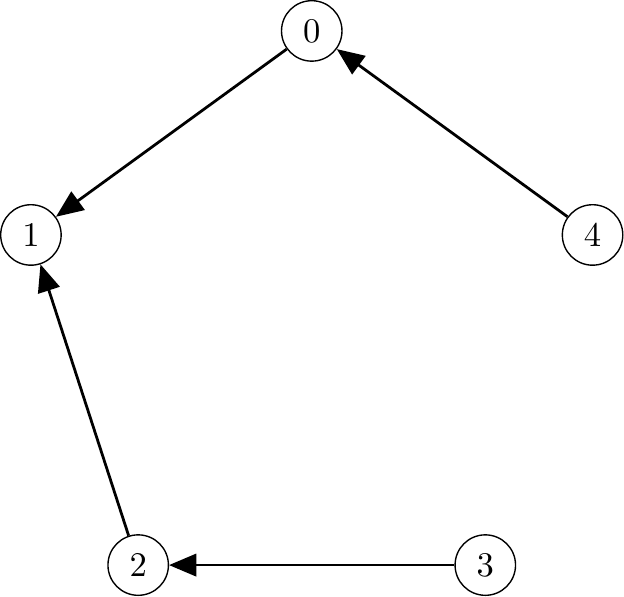} &
        \includegraphics[width=0.09\textwidth]{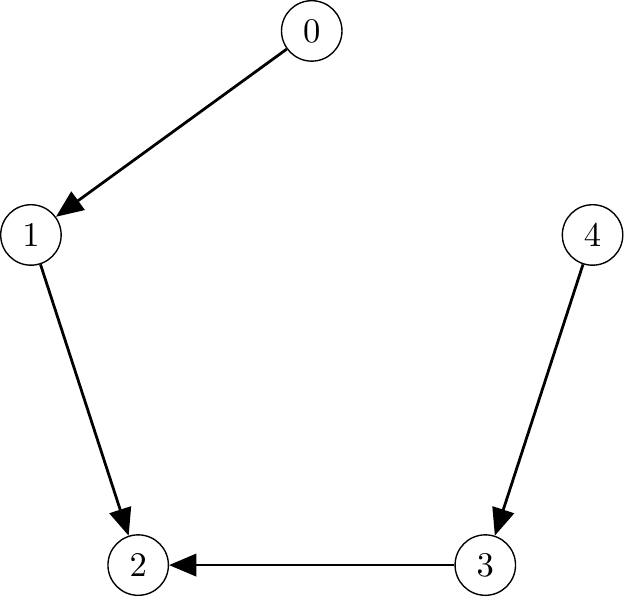} &
        \\[0.3ex]
        
        \includegraphics[width=0.09\textwidth]{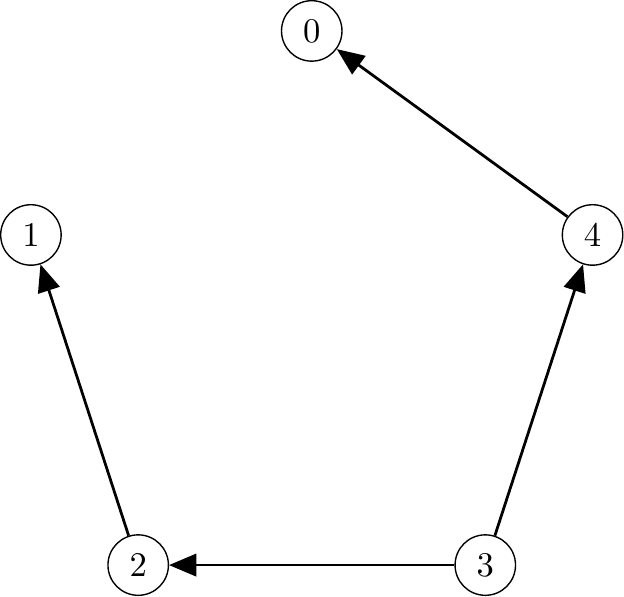} &
        \includegraphics[width=0.09\textwidth]{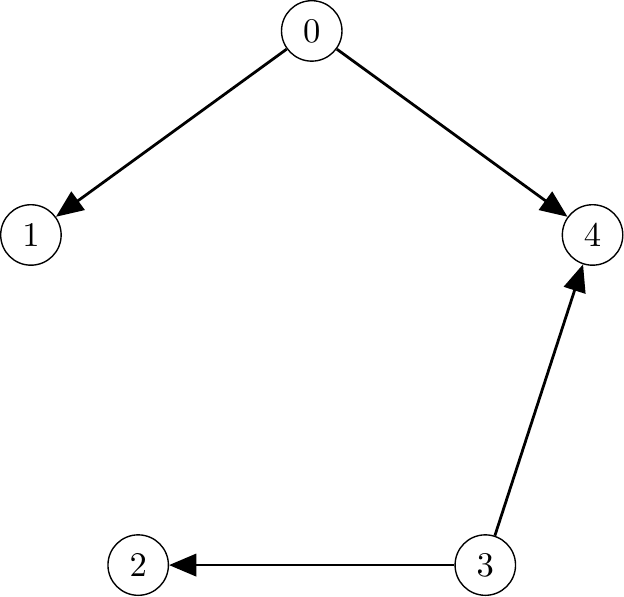} &
        \includegraphics[width=0.09\textwidth]{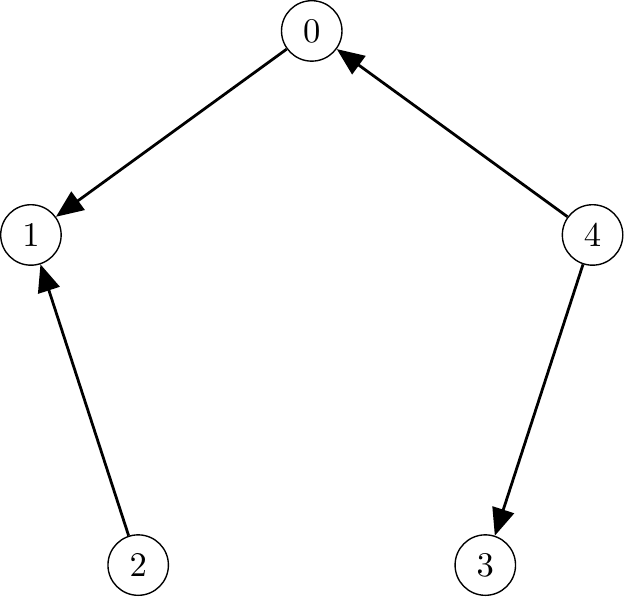} &
        \includegraphics[width=0.09\textwidth]{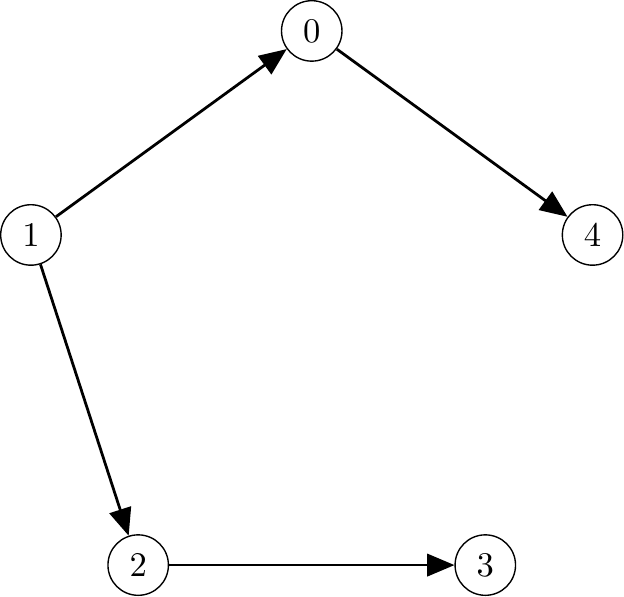} &
        \includegraphics[width=0.09\textwidth]{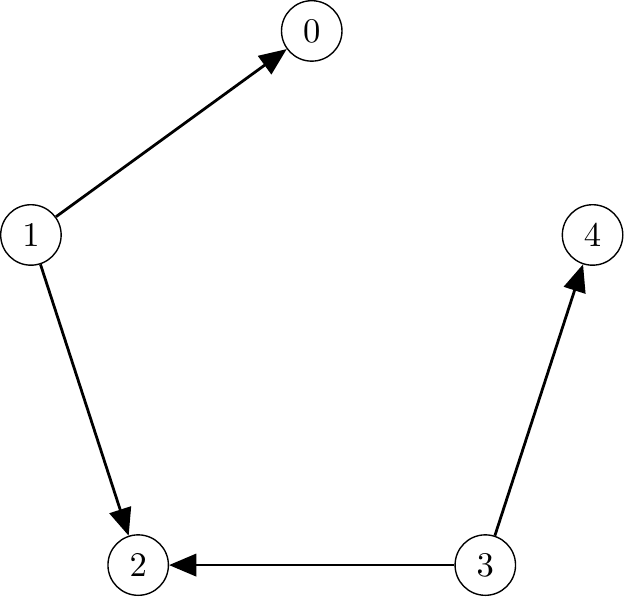} &
    \end{tabular}
    \caption{
        Facet subnetworks (all primitive) of a cycle of 5 oscillators.
    }
    \label{fig:cycle5-dac}
\end{figure}

\begin{remark}
        This computationally obtained root count for the Kuramoto equations
        induced by cycle networks has also been rigorously analyzed using techniques
        from convex geometry in a related work,\cite{ChenDavisMehta2018Counting}
        where an explicit formula for the root count is established.
    Moreover,
    % The cases with cycle networks is studied in detail in a
    a follow-up paper~\cite{ChenDavis2018Toric} by Robert Davis and the author
    demonstrated that for cycle networks with an \emph{odd} number of oscillators, 
    the facet decomposition scheme always produces the most refined 
    decomposition into primitive subnetworks,
    and there is an one-to-one correspondence between these primitive subnetworks
    and the complex synchronization configurations of the original network 
    under the genericity conditions on the coefficients.
    For cycle networks with an \emph{even} number of oscillators, however,
    the facet decomposition scheme alone is not sufficient to produce
    only primitive subnetworks.
    A refined decomposition scheme~\cite{ChenDavis2018Toric} is proposed
    that can further decompose cycle networks into only primitive subnetworks.
\end{remark}

%-----------------------------------------------------------------------------
\subsection{A chordal network}

We now consider a chordal network of four oscillators,
which consists of a cycle of four nodes together with a ``chord'' edge,
as displayed in~\Cref{fig:chord4}.
Alternatively, such a graph can also be considered as two cycle graphs 
sharing common edges.
\Cref{fig:chord4-dac} shows all 12 of the facet subnetworks.
Eight of them are primitive.
The remaining four are non-primitive but less complicated than the original network.
    In total, this chordal network of four oscillators has 
    16 complex synchronization configurations:
    The eight primitive facet subnetworks each contribute one configuration,
    and the four non-primitive subnetworks each contribute two.
    
    As in the previous case,
    the set of facet subnetworks exhibits symmetric structures.
    In additional to the transpose symmetry stated in \Cref{thm:facet-network},
    the set inherits from the original network the reflection symmetries given by
    $\{ 0 \mapsto 2 \}$, $\{ 1 \mapsto 3\}$, etc.

\begin{figure}[ht]
    \begin{tabular}{ccccc}
        \includegraphics[width=0.11\textwidth]{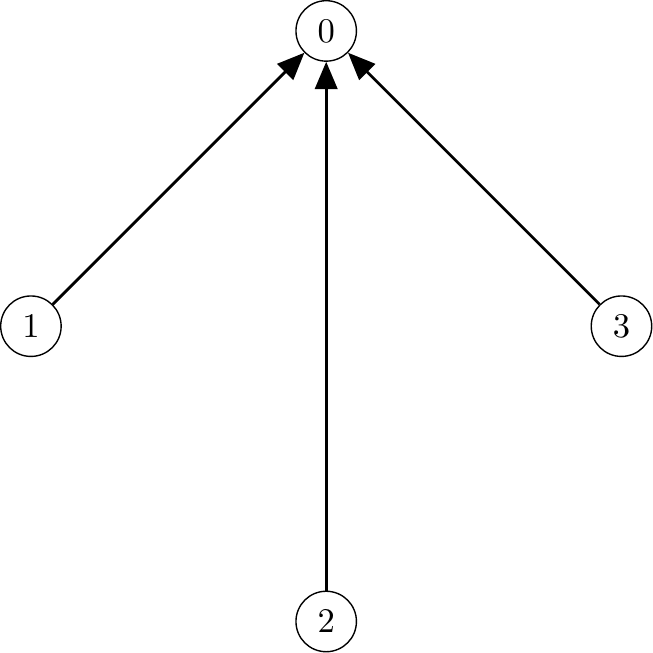} &
        \includegraphics[width=0.11\textwidth]{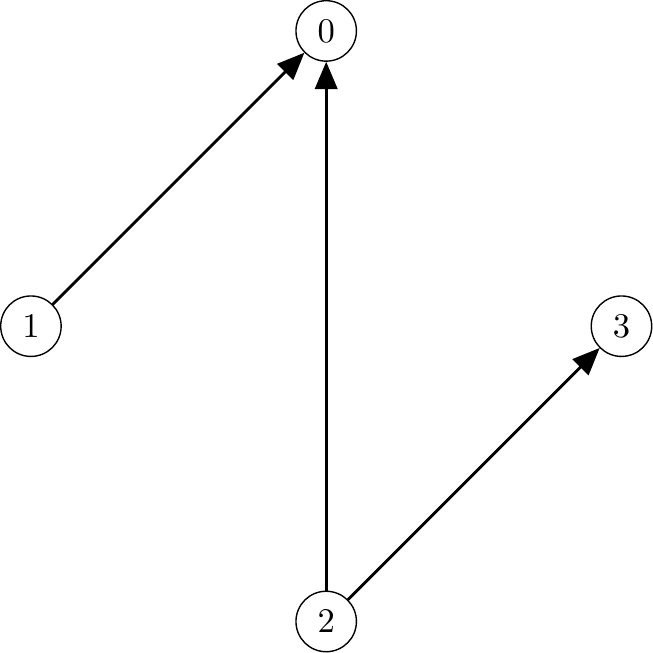} &
        \includegraphics[width=0.11\textwidth]{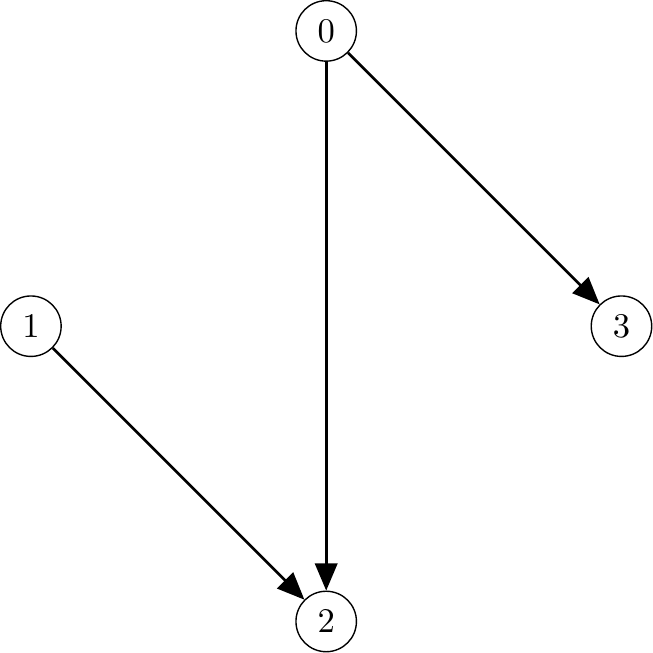} &
        \includegraphics[width=0.11\textwidth]{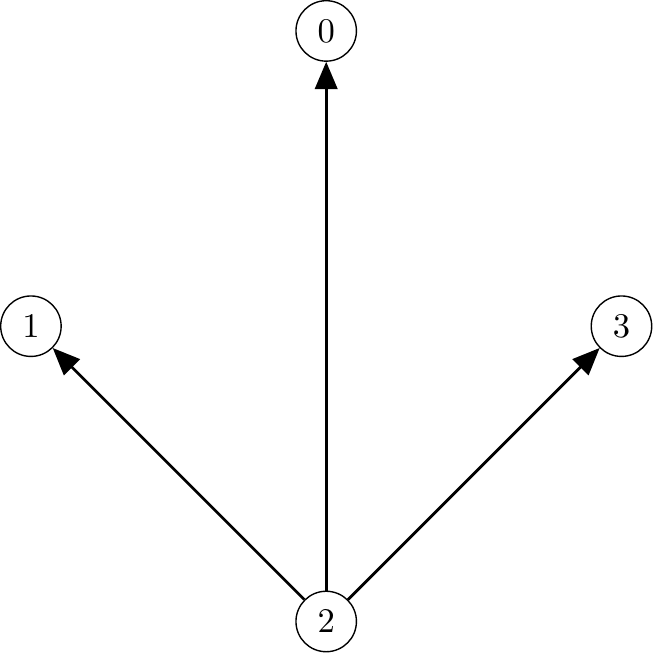} &
        \\[1ex]
        \includegraphics[width=0.11\textwidth]{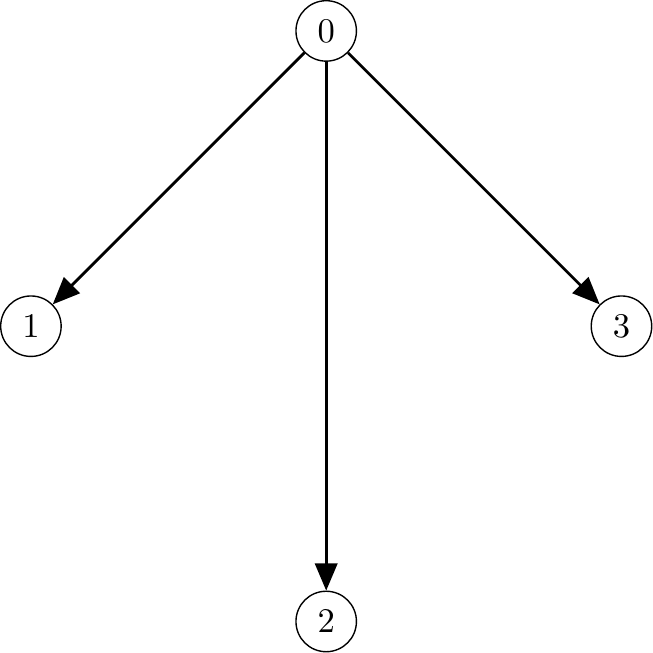} &
        \includegraphics[width=0.11\textwidth]{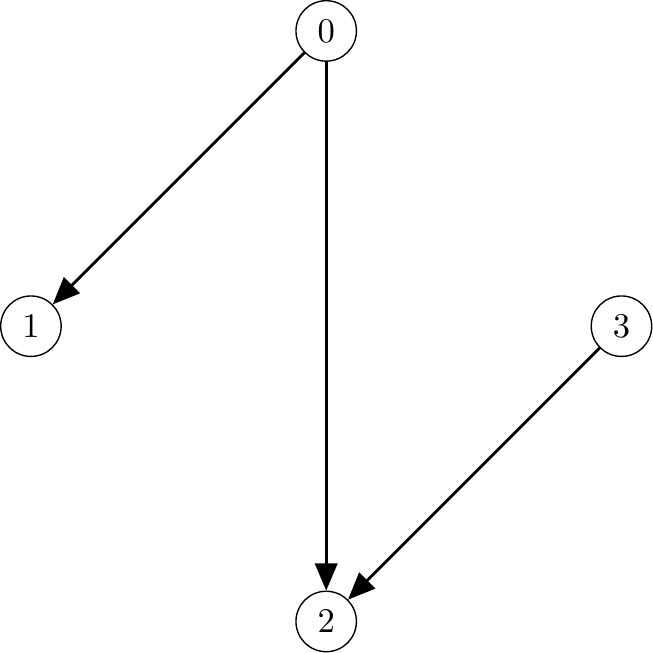} &
        \includegraphics[width=0.11\textwidth]{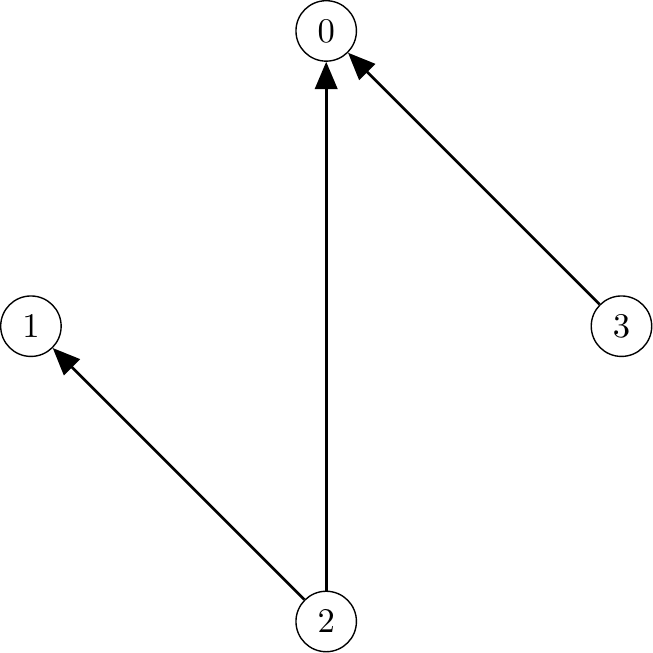} &
        \includegraphics[width=0.11\textwidth]{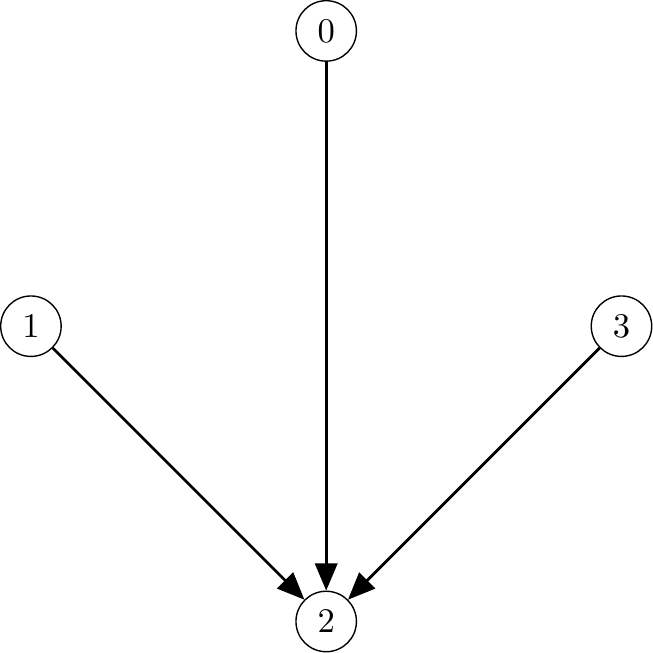} &
        \\[1ex]
        \includegraphics[width=0.11\textwidth]{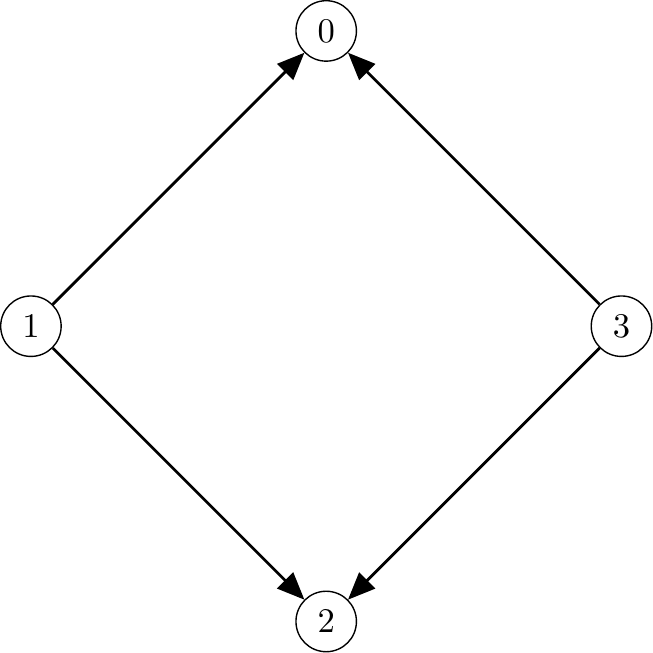} &
        \includegraphics[width=0.11\textwidth]{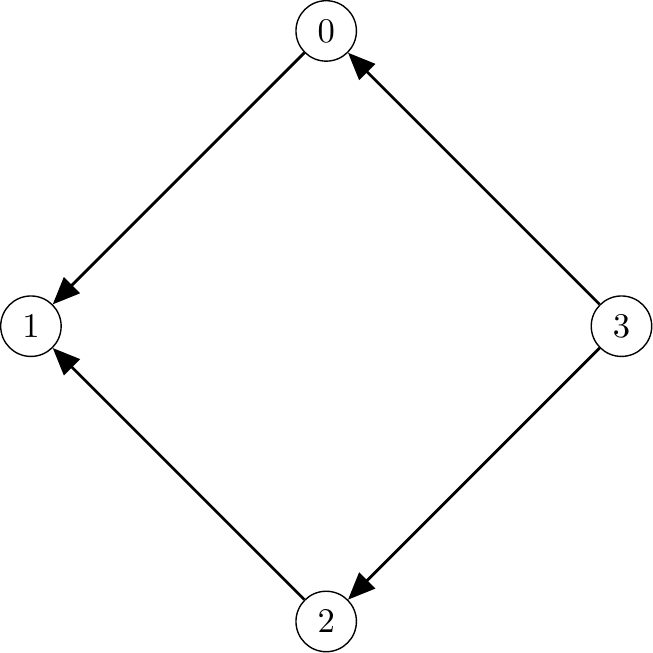} &
        \includegraphics[width=0.11\textwidth]{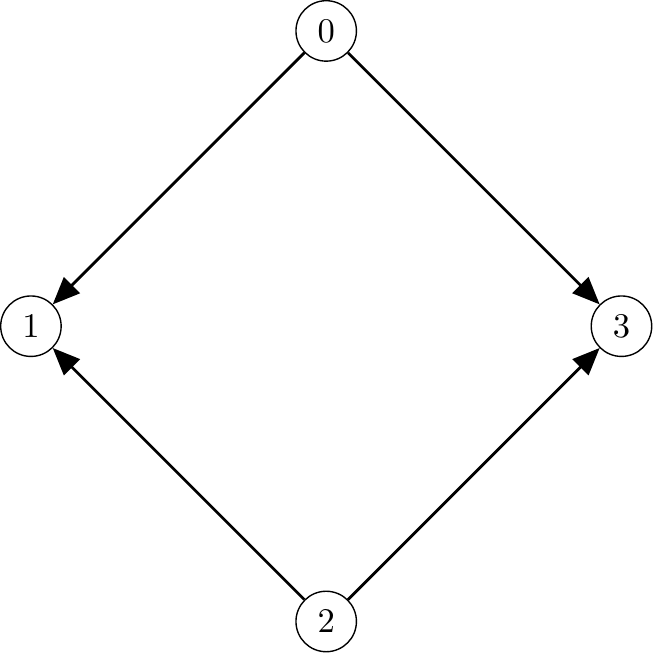} &
        \includegraphics[width=0.11\textwidth]{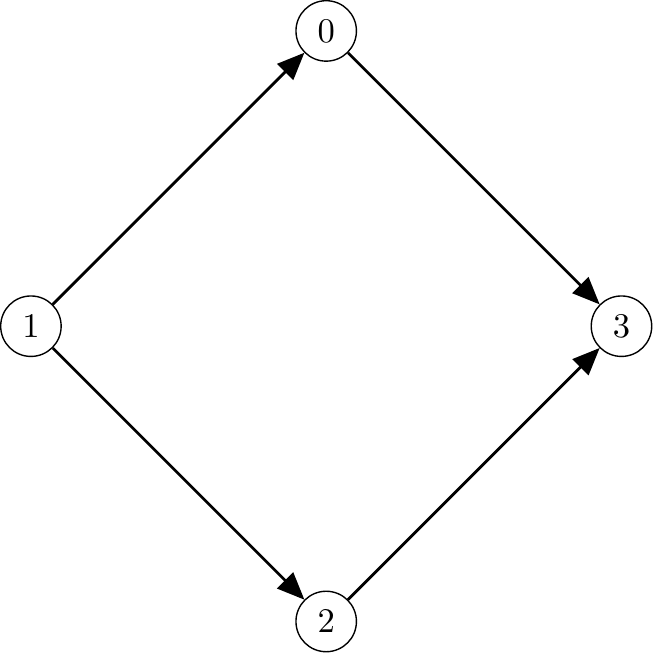} &
    \end{tabular}
    \caption{Facet subnetworks of a chordal graph of 4 nodes}
    \label{fig:chord4-dac}
\end{figure}

%-----------------------------------------------------------------------------
\subsection{A wheel network}

% Finally, we show the facet decomposition scheme for Kuramoto networks 
% whose underlying graph are wheel graphs.
\Cref{fig:wheel10} displays a ``wheel network'' of 10 oscillators 
which is simply a cycle network consisting of 9 oscillators
together with an additional central node that is 
directly connected to all other nodes.
This network has a total of 1598 facet subnetworks.
116 of them are primitive facet subnetworks
(see for example \Cref{fig:wheel-10-prim} and \Cref{fig:wheel-10-prim-mirror}).
The number of edges in the subnetworks range from 9 (primitive) to 13.
\Cref{fig:wheel-10-dac} includes some examples from each of these classes,
and \Cref{tab:wheel-10-sizes} summarizes the distribution of these subnetworks by size.

Transpose symmetry can be seen from \cref{fig:wheel-10-prim} and \cref{fig:wheel-10-prim-mirror}
--- one could be obtained from another by reversing the direction of all the edges.
Similar to the case of cycle networks,
the facet subnetworks exhibit cyclic symmetries inherited from the original network.

By computing the volume of each facet of $\adjp_G$,
the contribution of each facet subnetwork to the total root count can be obtained,
and they are summarized in \cref{tab:wheel-10-sizes}.
For instance, each of the 116 primitive subnetworks contribute one complex root
while each of the 414 non-primitive subnetworks having 10 edges contribute two complex roots.
In total, the maximum number of complex synchronization configurations
on this network is 8480,
and this bound is exact for generic choices of coefficients.

\begin{figure}[ht]
    \centering
    \begin{tikzpicture}[scale=0.5]
    \node[draw,circle] (0) at ( 0.0,0.0) {$0$};
    \node[draw,circle] (1) at (-1.2,2.3) {$1$};
    \node[draw,circle] (9) at ( 1.2,2.3) {$9$};
    \node[draw,circle] (2) at (-2.85316954889,0.927050983125) {$2$};
    \node[draw,circle] (3) at (-2.85316954889,-0.927050983125) {$3$};
    \node[draw,circle] (4) at (-1.76335575688,-2.42705098312) {$4$};
    \node[draw,circle] (5) at (-5.51091059616e-16,-3.0) {$5$};
    \node[draw,circle] (6) at (1.76335575688,-2.42705098312) {$6$};
    \node[draw,circle] (7) at (2.85316954889,-0.927050983125) {$7$};
    \node[draw,circle] (8) at (2.85316954889,0.927050983125) {$8$};
    \path[draw,thick] (0)  --  (1) ;
    \path[draw,thick] (0)  --  (2) ;
    \path[draw,thick] (0)  --  (3) ;
    \path[draw,thick] (0)  --  (4) ;
    \path[draw,thick] (0)  --  (5) ;
    \path[draw,thick] (0)  --  (6) ;
    \path[draw,thick] (0)  --  (7) ;
    \path[draw,thick] (0)  --  (8) ;
    \path[draw,thick] (0)  --  (9) ;
    \path[draw,thick] (1)  --  (2) ;
    \path[draw,thick] (2)  --  (3) ;
    \path[draw,thick] (3)  --  (4) ;
    \path[draw,thick] (4)  --  (5) ;
    \path[draw,thick] (5)  --  (6) ;
    \path[draw,thick] (6)  --  (7) ;
    \path[draw,thick] (7)  --  (8) ;
    \path[draw,thick] (8)  --  (9) ;
    \path[draw,thick] (9)  --  (1) ;
    \end{tikzpicture}
    \caption{A wheel network of 10 oscillators}
    \label{fig:wheel10}
\end{figure}
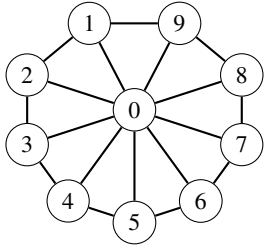

\begin{table}[ht]
    \centering
    \begin{tabular}{lrrrrr}
        \toprule
        Facet subnetwork size           &   9 &  10 &  11 &  12 &  13 \\ \midrule
        N.o. facet subnetworks          & 116 & 414 & 540 & 384 & 144 \\
        Root count per subnetwork &   1 &   2 &   4 &   8 &  16 \\ \bottomrule
    \end{tabular}
    \caption{
        Number of facet subnetworks and 
        the root counts contributed by each 
        classified by the size (n.o. edges) of subnetworks
    }
    \label{tab:wheel-10-sizes}
\end{table}

In spite of having a large number of potential complex synchronization configurations,
a modern homotopy continuation solver can still find all configurations quickly.
For instance, a modified version\footnote{%
    \textsf{Hom4PS-3} implements the general polyhedral homotopy method.\cite{HuberSturmfels1995Polyhedral}
    As noted in \Cref{rmk:polyhedral}, the proposed homotopy method
    can be considered as a highly specialized version of the unmixed form of
    the polyhedral homotopy method.
    The preliminary implementation used here combines \textsf{Hom4PS-3}
    with a pre-processor that analyzes the adjacency polytope,
    generates suitable facet systems as starting systems,
    and bootstraps the homotopy method.
    This pre-processor relies on facet information computed by
    \textsf{Polymake}.\cite{polymake}
    % This prototype implementation is unable to take advantage of
    % the symmetries and other special structure of the adjacency polytope,
    % it already offer
}of \textsf{Hom4PS-3} 
can solve the corresponding synchronization system 
in seconds on a moderate workstation.
More importantly, the homotopy method scales almost linearly as more processor cores are used.
\Cref{tab:hom4ps-time} displays the computation time consumed on a workstation equipped with 16 GB of memory and
\textsf{Intel Xeon E5-2690} processor running at 2.9 GHz and up to 8 cores.

\begin{figure}[ht]
    \centering
    \subfloat[A primitive subnetwork]{
        \label{fig:wheel-10-prim}
        \centering
        \includegraphics[width=0.14\textwidth]{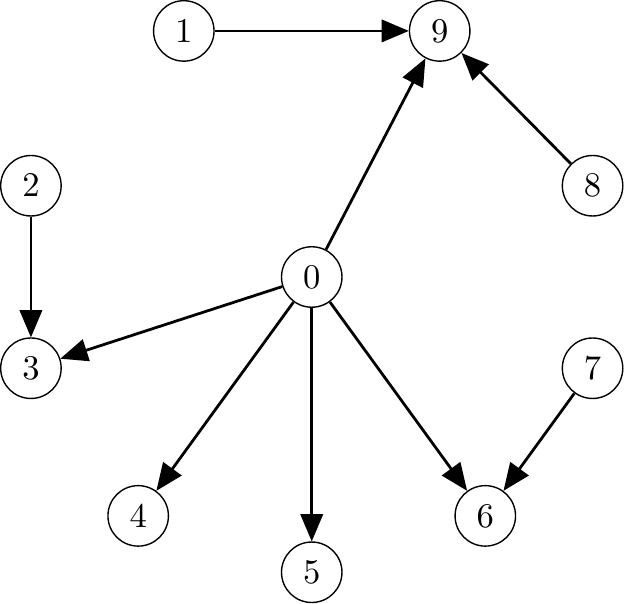}
    }\hspace{1ex}%
    \subfloat[Another primitive subnetwork]{
        \label{fig:wheel-10-prim-mirror}
        \centering
        \includegraphics[width=0.14\textwidth]{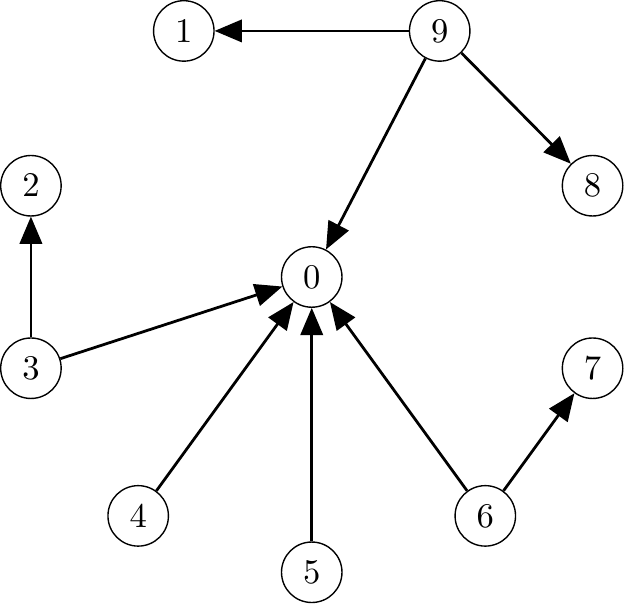}
    }\hspace{1ex}%
    \subfloat[A nonprimitive subnetwork]{
        \label{fig:wheel-10-10}
        \centering
        \includegraphics[width=0.14\textwidth]{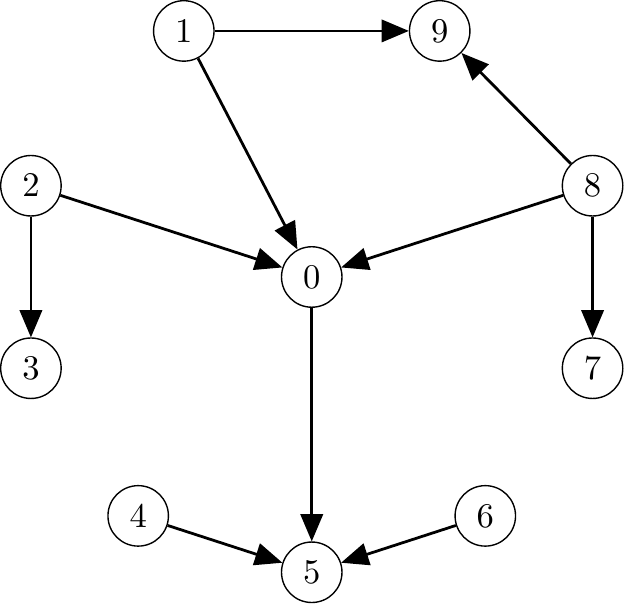}
    }
    
    \subfloat[A nonprimitive subnetwork]{
        \label{fig:wheel-10-11}
        \centering
        \includegraphics[width=0.14\textwidth]{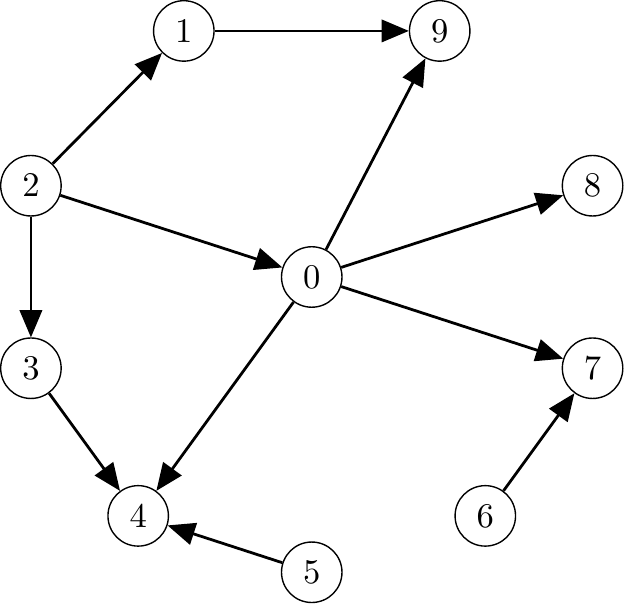}
    }\hspace{1ex}%
    \subfloat[A nonprimitive subnetwork]{
        \label{fig:wheel-10-12}
        \centering
        \includegraphics[width=0.14\textwidth]{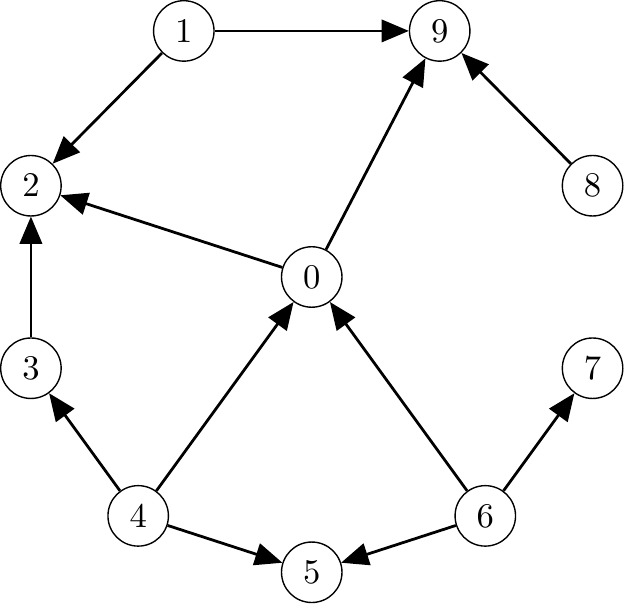}
    }\hspace{1ex}%
    \subfloat[A nonprimitive subnetwork]{
        \label{fig:wheel-10-13}
        \centering
        \includegraphics[width=0.14\textwidth]{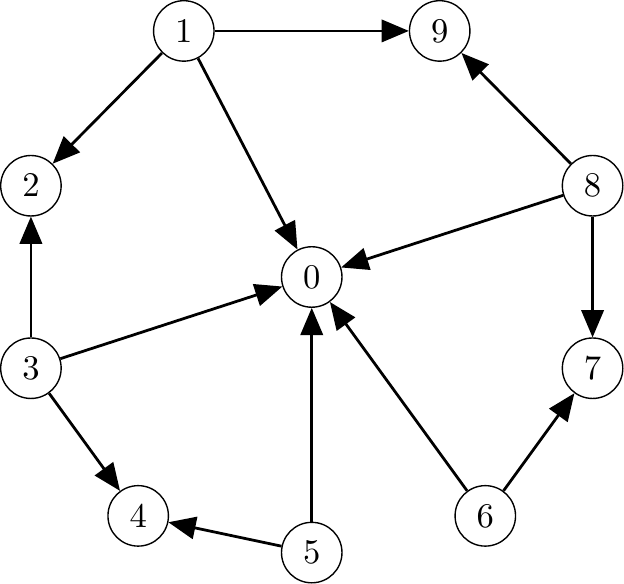}
    }
    \caption{Some subnetworks of a wheel network of 10 oscillators}
    \label{fig:wheel-10-dac}
\end{figure}

\begin{table}[h]
    \centering
    \begin{tabular}{lrrrrrr}
        Threads/cores used & 1     & 2    &  3    & 4    & 6    & 8    \\ \toprule
        Wall clock time    & 9.1s  & 4.6s &  3.3s & 2.6s & 1.8s & 1.4s \\ \bottomrule
    \end{tabular}
    \caption{
        Time consumed by a preliminary implementation based on a modified version of \textsf{Hom4PS-3}
        for solving the synchronization system derived from the Kuramoto network
        displayed in \Cref{fig:wheel10}
        using different numbers of processor cores.
    }
    \label{tab:hom4ps-time}
\end{table}

\section{Conclusion and future directions}\label{sec:conclusions}

This paper focuses on the study of frequency synchronization configurations
in Kuramoto models for networks of coupled oscillators.
It lays the foundation for a divide-and-conquer approach to analyzing
large and complex Kuramoto networks from the view point of algebraic geometry.
% The Kuramoto network is one of the most widely used models in the study of
% spontaneous synchronization phenomena in a number of independent fields.
% This simple model for a network of coupled oscillators
% exhibits rich and complex behaviors.
% One important problem in Kuramoto models is to compute and understand
% the set of all frequency synchronization configurations --- 
% special configurations in which all oscillators are tuned to the exact same frequency.
% These configurations form the solution set of a system of transcendental equations
% --- the synchronization equations.
% Despite its simple formulation, the problem of finding the complete set of
% frequency synchronizations remains a difficult problem especially for large networks.
The main contribution of this paper is a general decomposition scheme for Kuramoto networks
that can reduce a complicated Kuramoto network into a collection of simpler
directed acyclic subnetworks which are generalized Kuramoto networks
that allow one-way interactions among oscillators (c.f. Ref.\onlinecite{DelabaysJacquodDorfler2018}).
The challenge of finding all possible frequency synchronization configurations
is thus reduced to the problem of fully understanding 
these much simpler subnetworks.

Starting from a complex algebraic re-formulation of 
the underlying transcendental equations,
the proposed framework uses the geometric information
extracted from the ``adjacency polytope'' of the network,
which is a polytope that encodes the network topology.
This polytope is also the convex hull of the union of the Newton polytopes
derived from the algebraic synchronization equations.
The subnetworks are then in one-to-one correspondence with the
facets (maximal faces) of the adjacency polytope.
Since the adjacency polytopes are highly symmetric,
the facets can be enumerated efficiently.
Associated with each of these facet subnetworks
is a much simpler, generalized synchronization equation,
a ``facet subsystem,''
which only involves a fraction of the terms from the original
algebraic synchronization equations.
These subsystems are expected to be easier to solve than the original system.
% On one end of the spectrum,
The facet subsystems associated with primitive subnetworks
can be solved in linear time with no additional memory required.

From a computational point of view,
the proposed decomposition framework gives rise to a homotopy continuation method.
It induces a continuous deformation of the algebraic synchronization system
that can degenerate into a collection of simpler facet subsystems.
    Homotopy continuation methods have been proved to be
    highly scalable as each solution can be computed independently.
    The can therefore be parallelized easily on modern high performance computing hardware.

The resulting homotopy method can be viewed as a highly specialized polyhedral homotopy method
with fixed lifting functions (valuation functions).
In the case that all subnetworks are primitive,
this homotopy construction circumvents two of the main 
computational bottlenecks of the polyhedral homotopy method:
\begin{enumerate}
    \item 
        Since the facet subnetworks corresponds to facets of a well-known family of polytopes,
        they can be listed more easily, thereby allowing one to skip the costly step of ``mixed cells computation''.
    \item
        The starting system --- facet subsystems ---
        can be solved in \emph{linear time} and does not require the costly step of
        computing Smith Normal Forms or Hermite Normal Forms of exponent matrices.
\end{enumerate}

The examples in the \cref{sec:example} demonstrate that
% The proposed decomposition scheme is still incomplete:
% Depending on the network topology,
for certain types of networks,
some facet subnetworks may not be primitive.
% and the corresponding facet subsystems need to be solved via other means.
The next step is to develop algorithms 
in order to refine our facet-based decomposition efficiently 
so that \emph{all} resulting subnetworks are primitive.
Recent fellow-up works have presented promising developments in this direction.
For instance, in a recent paper by Robert Davis and the author,
the directed acyclic decomposition scheme proposed here is significantly refined
for cycle graphs so that all resulting subnetworks are primitive.
This refinement is equivalent to a triangulation process for the facets 
of the adjacency polytope,
but with the derivation of the explicit formula,
the complexity of the refinement step for each facet is only
linear in the number of oscillators.
The refinement scheme for other types of networks remains an open problem.

    In this work, we have expanded our domain to include complex synchronization configurations
    in order to take advantage of the powerful tools of complex algebraic geometry.
    These complex configurations can be of great value
    even when only real configurations are needed.
    First, the complex configurations necessarily include all real configurations.
    Through an examination of the imaginary parts of the phase angles,
    the extraneous (non-real) solutions can be filtered out easily.
    Similarly, by examining the eigenvalues of the Jacobian matrices
    at real solutions unstable solutions can also be filtered out 
    if only stable solutions are desired.
    More importantly, the collection of complex configurations forms a reservoir 
    of real configurations as one vary the parameters
    (natural frequencies and coupling strength).
    As the parameters move in the space of all possible parameters,
    for almost all parameters, the total number of complex configurations 
    (counting multiplicity) remains the same, 
    while the real configurations can increase or decrease.
    Non-real configurations may collide and form real configurations,
    and real configurations may degenerate into non-real ones.
    This phenomenon is a central question in numerical algebraic geometry
    and algebraic geometry in general.
    Numerical algorithms using all complex solutions as seeds for the exploration of the landscape of real solutions
    have been developed in the context of mechanical engineering.\cite{GriffinHauenstein2015Real}
    Similarly numerical algorithms targeting real solutions of certain stability types
    have been studied in the context of chemical clusters.\cite{ChenMehta2015Index}
    Combining these techniques and the framework developed in this work
    can shed new light on the 
    possible real or even stable synchronization configurations of Kuramoto models.

\section*{Acknowledgments}

This work was funded in part by the AMS-Simons Travel Grant and Auburn University at Montgomery Grant-In-Aid program
and NSF under Grant Nos. 1115587 and 1923099.

The author would like to thank 
Dhagash Mehta for introducing the interesting topic of the Kuramoto model;
Daniel Molzahn for the helpful discussions on the closely related power-flow equations;
and Wuwei Lin for asking the question of whether or not a Kuramoto network
can be decomposed into smaller subnetwork in a reasonable way 
--- a question that started this research.
The author also greatly benefited from discussions with Robert Davis.
Indeed, discussions with Robert Davis produced a significant refinement 
of the framework developed in this paper for cycle graphs,
and this extension is discussed in a separate follow-up paper~\cite{ChenDavis2018Toric}.

\vspace{0.2cm}
\textit{
    This article may be downloaded for personal use only. 
    Any other use requires prior permission of the author and AIP Publishing. 
    This article appeared in  
    Chaos: An Interdisciplinary Journal of Nonlinear Science 
    (Vol.29, Issue 9)
    and may be found at 
    \url{https://doi.org/10.1063/1.5097826}.
}

\bibliography{all}

\begin{thebibliography}{39}%
\makeatletter
\providecommand \@ifxundefined [1]{%
 \@ifx{#1\undefined}
}%
\providecommand \@ifnum [1]{%
 \ifnum #1\expandafter \@firstoftwo
 \else \expandafter \@secondoftwo
 \fi
}%
\providecommand \@ifx [1]{%
 \ifx #1\expandafter \@firstoftwo
 \else \expandafter \@secondoftwo
 \fi
}%
\providecommand \natexlab [1]{#1}%
\providecommand \enquote  [1]{``#1''}%
\providecommand \bibnamefont  [1]{#1}%
\providecommand \bibfnamefont [1]{#1}%
\providecommand \citenamefont [1]{#1}%
\providecommand \href@noop [0]{\@secondoftwo}%
\providecommand \href [0]{\begingroup \@sanitize@url \@href}%
\providecommand \@href[1]{\@@startlink{#1}\@@href}%
\providecommand \@@href[1]{\endgroup#1\@@endlink}%
\providecommand \@sanitize@url [0]{\catcode `\\12\catcode `\$12\catcode
  `\&12\catcode `\#12\catcode `\^12\catcode `\_12\catcode `\%12\relax}%
\providecommand \@@startlink[1]{}%
\providecommand \@@endlink[0]{}%
\providecommand \url  [0]{\begingroup\@sanitize@url \@url }%
\providecommand \@url [1]{\endgroup\@href {#1}{\urlprefix }}%
\providecommand \urlprefix  [0]{URL }%
\providecommand \Eprint [0]{\href }%
\providecommand \doibase [0]{http://dx.doi.org/}%
\providecommand \selectlanguage [0]{\@gobble}%
\providecommand \bibinfo  [0]{\@secondoftwo}%
\providecommand \bibfield  [0]{\@secondoftwo}%
\providecommand \translation [1]{[#1]}%
\providecommand \BibitemOpen [0]{}%
\providecommand \bibitemStop [0]{}%
\providecommand \bibitemNoStop [0]{.\EOS\space}%
\providecommand \EOS [0]{\spacefactor3000\relax}%
\providecommand \BibitemShut  [1]{\csname bibitem#1\endcsname}%
\let\auto@bib@innerbib\@empty
%</preamble>
\bibitem [{\citenamefont {D{\"{o}}rfler}\ and\ \citenamefont
  {Bullo}(2014)}]{dorfler_synchronization_2014}%
  \BibitemOpen
  \bibfield  {author} {\bibinfo {author} {\bibfnamefont {F.}~\bibnamefont
  {D{\"{o}}rfler}}\ and\ \bibinfo {author} {\bibfnamefont {F.}~\bibnamefont
  {Bullo}},\ }\bibfield  {title} {\enquote {\bibinfo {title} {{Synchronization
  in complex networks of phase oscillators: A survey}},}\ }\href {\doibase
  10.1016/j.automatica.2014.04.012} {\bibfield  {journal} {\bibinfo  {journal}
  {Automatica}\ }\textbf {\bibinfo {volume} {50}},\ \bibinfo {pages}
  {1539--1564} (\bibinfo {year} {2014})}\BibitemShut {NoStop}%
\bibitem [{\citenamefont {Kuramoto}(1975)}]{Kuramoto1975Self}%
  \BibitemOpen
  \bibfield  {author} {\bibinfo {author} {\bibfnamefont {Y.}~\bibnamefont
  {Kuramoto}},\ }\bibfield  {title} {\enquote {\bibinfo {title}
  {{Self-entrainment of a population of coupled non-linear oscillators}},}\ \
  }(\bibinfo  {publisher} {Springer Berlin Heidelberg},\ \bibinfo {year}
  {1975})\ pp.\ \bibinfo {pages} {420--422}\BibitemShut {NoStop}%
\bibitem [{\citenamefont {Kuramoto}(2012)}]{Kuramoto2012Chemical}%
  \BibitemOpen
  \bibfield  {author} {\bibinfo {author} {\bibfnamefont {Y.}~\bibnamefont
  {Kuramoto}},\ }\href {https://books.google.com/books?id=tcTyCAAAQBAJ} {\emph
  {\bibinfo {title} {{Chemical Oscillations, Waves, and Turbulence}}}}\
  (\bibinfo  {publisher} {Springer Science \& Business Media},\ \bibinfo {year}
  {2012})\ p.\ \bibinfo {pages} {165}\BibitemShut {NoStop}%
\bibitem [{\citenamefont {Baillieul}\ and\ \citenamefont
  {Byrnes}(1982)}]{Baillieul1982}%
  \BibitemOpen
  \bibfield  {author} {\bibinfo {author} {\bibfnamefont {J.}~\bibnamefont
  {Baillieul}}\ and\ \bibinfo {author} {\bibfnamefont {C.~I.}\ \bibnamefont
  {Byrnes}},\ }\bibfield  {title} {\enquote {\bibinfo {title} {{Geometric
  Critical Point Analysis of Lossless Power System Models}},}\ }\href {\doibase
  10.1109/TCS.1982.1085093} {\bibfield  {journal} {\bibinfo  {journal} {IEEE
  Transactions on Circuits and Systems}\ }\textbf {\bibinfo {volume} {29}},\
  \bibinfo {pages} {724--737} (\bibinfo {year} {1982})}\BibitemShut {NoStop}%
\bibitem [{\citenamefont {Kuramoto}(1984)}]{Kuramoto1984Cooperative}%
  \BibitemOpen
  \bibfield  {author} {\bibinfo {author} {\bibfnamefont {Y.}~\bibnamefont
  {Kuramoto}},\ }\bibfield  {title} {\enquote {\bibinfo {title} {{Cooperative
  Dynamics of Oscillator Community}},}\ }\href {\doibase 10.1143/PTPS.79.223}
  {\bibfield  {journal} {\bibinfo  {journal} {Progress of Theoretical Physics
  Supplement}\ }\textbf {\bibinfo {volume} {79}},\ \bibinfo {pages} {223--240}
  (\bibinfo {year} {1984})}\BibitemShut {NoStop}%
\bibitem [{\citenamefont {Strogatz}(2000)}]{Strogatz2000}%
  \BibitemOpen
  \bibfield  {author} {\bibinfo {author} {\bibfnamefont {S.~H.}\ \bibnamefont
  {Strogatz}},\ }\bibfield  {title} {\enquote {\bibinfo {title} {{From Kuramoto
  to Crawford: exploring the onset of synchronization in populations of coupled
  oscillators}},}\ }\href {\doibase 10.1016/S0167-2789(00)00094-4} {\bibfield
  {journal} {\bibinfo  {journal} {Physica D: Nonlinear Phenomena}\ }\textbf
  {\bibinfo {volume} {143}},\ \bibinfo {pages} {1--20} (\bibinfo {year}
  {2000})}\BibitemShut {NoStop}%
\bibitem [{\citenamefont {Aeyels}\ and\ \citenamefont
  {Rogge}(2004)}]{AeyelsRogge2004Existence}%
  \BibitemOpen
  \bibfield  {author} {\bibinfo {author} {\bibfnamefont {D.}~\bibnamefont
  {Aeyels}}\ and\ \bibinfo {author} {\bibfnamefont {J.~A.}\ \bibnamefont
  {Rogge}},\ }\bibfield  {title} {\enquote {\bibinfo {title} {{Existence of
  Partial Entrainment and Stability of Phase Locking Behavior of Coupled
  Oscillators}},}\ }\href {\doibase 10.1143/PTP.112.921} {\bibfield  {journal}
  {\bibinfo  {journal} {Progress of Theoretical Physics}\ }\textbf {\bibinfo
  {volume} {112}},\ \bibinfo {pages} {921--942} (\bibinfo {year}
  {2004})}\BibitemShut {NoStop}%
\bibitem [{\citenamefont {Mirollo}\ and\ \citenamefont
  {Strogatz}(2005)}]{MirolloStrogatz2005Spectrum}%
  \BibitemOpen
  \bibfield  {author} {\bibinfo {author} {\bibfnamefont {R.~E.}\ \bibnamefont
  {Mirollo}}\ and\ \bibinfo {author} {\bibfnamefont {S.~H.}\ \bibnamefont
  {Strogatz}},\ }\bibfield  {title} {\enquote {\bibinfo {title} {{The spectrum
  of the locked state for the Kuramoto model of coupled oscillators}},}\ }\href
  {\doibase 10.1016/j.physd.2005.01.017} {\bibfield  {journal} {\bibinfo
  {journal} {Physica D: Nonlinear Phenomena}\ }\bibinfo {series}
  {Synchronization and Pattern Formation in Nonlinear Systems: New Developments
  and Future PerspectivesA Special Issue dedicated to Professor Yushiki
  Kuramoto},\ \textbf {\bibinfo {volume} {205}},\ \bibinfo {pages} {249--266}
  (\bibinfo {year} {2005})}\BibitemShut {NoStop}%
\bibitem [{\citenamefont {Mehta}\ \emph {et~al.}(2015)\citenamefont {Mehta},
  \citenamefont {Daleo}, \citenamefont {D{\"{o}}rfler},\ and\ \citenamefont
  {Hauenstein}}]{MehtaDaleoDorflerHauenstein2015Algebraic}%
  \BibitemOpen
  \bibfield  {author} {\bibinfo {author} {\bibfnamefont {D.}~\bibnamefont
  {Mehta}}, \bibinfo {author} {\bibfnamefont {N.~S.}\ \bibnamefont {Daleo}},
  \bibinfo {author} {\bibfnamefont {F.}~\bibnamefont {D{\"{o}}rfler}}, \ and\
  \bibinfo {author} {\bibfnamefont {J.~D.}\ \bibnamefont {Hauenstein}},\
  }\bibfield  {title} {\enquote {\bibinfo {title} {{Algebraic geometrization of
  the kuramoto model: Equilibria and stability analysis}},}\ }\href {\doibase
  10.1063/1.4919696} {\bibfield  {journal} {\bibinfo  {journal} {Chaos}\
  }\textbf {\bibinfo {volume} {25}},\ \bibinfo {pages} {053103} (\bibinfo
  {year} {2015})}\BibitemShut {NoStop}%
\bibitem [{\citenamefont {Chen}\ \emph {et~al.}(2016)\citenamefont {Chen},
  \citenamefont {Marecek}, \citenamefont {Mehta},\ and\ \citenamefont
  {Niemerg}}]{ChenMarecekMehtaNeimerg2019Three}%
  \BibitemOpen
  \bibfield  {author} {\bibinfo {author} {\bibfnamefont {T.}~\bibnamefont
  {Chen}}, \bibinfo {author} {\bibfnamefont {J.}~\bibnamefont {Marecek}},
  \bibinfo {author} {\bibfnamefont {D.}~\bibnamefont {Mehta}}, \ and\ \bibinfo
  {author} {\bibfnamefont {M.}~\bibnamefont {Niemerg}},\ }\bibfield  {title}
  {\enquote {\bibinfo {title} {{Three Formulations of the Kuramoto Model as a
  System of Polynomial Equations}},}\ }\href {http://arxiv.org/abs/1603.05905}
  {\  (\bibinfo {year} {2016})},\ \Eprint {http://arxiv.org/abs/1603.05905}
  {arXiv:1603.05905} \BibitemShut {NoStop}%
\bibitem [{\citenamefont {Rogge}\ and\ \citenamefont
  {Aeyels}(2004)}]{RoggeAeyels2004Stability}%
  \BibitemOpen
  \bibfield  {author} {\bibinfo {author} {\bibfnamefont {J.~A.}\ \bibnamefont
  {Rogge}}\ and\ \bibinfo {author} {\bibfnamefont {D.}~\bibnamefont {Aeyels}},\
  }\bibfield  {title} {\enquote {\bibinfo {title} {{Stability of phase locking
  in a ring of unidirectionally coupled oscillators}},}\ }\href {\doibase
  10.1088/0305-4470/37/46/004} {\bibfield  {journal} {\bibinfo  {journal}
  {Journal of Physics A: Mathematical and General}\ }\textbf {\bibinfo {volume}
  {37}},\ \bibinfo {pages} {11135--11148} (\bibinfo {year} {2004})}\BibitemShut
  {NoStop}%
\bibitem [{\citenamefont {Ochab}\ and\ \citenamefont
  {G{\'{o}}ra}(2010)}]{OchabGora2010Synchronization}%
  \BibitemOpen
  \bibfield  {author} {\bibinfo {author} {\bibfnamefont {J.}~\bibnamefont
  {Ochab}}\ and\ \bibinfo {author} {\bibfnamefont {P.~F.}\ \bibnamefont
  {G{\'{o}}ra}},\ }\href
  {https://www.actaphys.uj.edu.pl/fulltext?series=Sup&vol=3&page=453} {\enquote
  {\bibinfo {title} {{SYNCHRONIZATION OF COUPLED OSCILLATORS IN A LOCAL
  ONE-DIMENSIONAL KURAMOTO MODEL *}},}\ }\bibinfo {type} {Tech. Rep.}\ \bibinfo
  {number} {2}\ (\bibinfo {year} {2010})\BibitemShut {NoStop}%
\bibitem [{\citenamefont {Delabays}, \citenamefont {Coletta},\ and\
  \citenamefont {Jacquod}(2016)}]{DelabaysColettaJacquod2016Multistability}%
  \BibitemOpen
  \bibfield  {author} {\bibinfo {author} {\bibfnamefont {R.}~\bibnamefont
  {Delabays}}, \bibinfo {author} {\bibfnamefont {T.}~\bibnamefont {Coletta}}, \
  and\ \bibinfo {author} {\bibfnamefont {P.}~\bibnamefont {Jacquod}},\
  }\bibfield  {title} {\enquote {\bibinfo {title} {{Multistability of
  phase-locking and topological winding numbers in locally coupled Kuramoto
  models on single-loop networks}},}\ }\href {\doibase 10.1063/1.4943296}
  {\bibfield  {journal} {\bibinfo  {journal} {Journal of Mathematical Physics}\
  }\textbf {\bibinfo {volume} {57}},\ \bibinfo {pages} {032701} (\bibinfo
  {year} {2016})}\BibitemShut {NoStop}%
\bibitem [{\citenamefont {Delabays}, \citenamefont {Coletta},\ and\
  \citenamefont {Jacquod}(2017)}]{DelabaysColettaJacquod2017Multistability}%
  \BibitemOpen
  \bibfield  {author} {\bibinfo {author} {\bibfnamefont {R.}~\bibnamefont
  {Delabays}}, \bibinfo {author} {\bibfnamefont {T.}~\bibnamefont {Coletta}}, \
  and\ \bibinfo {author} {\bibfnamefont {P.}~\bibnamefont {Jacquod}},\
  }\bibfield  {title} {\enquote {\bibinfo {title} {{Multistability of
  phase-locking in equal-frequency Kuramoto models on planar graphs}},}\ }\href
  {\doibase 10.1063/1.4978697} {\bibfield  {journal} {\bibinfo  {journal}
  {Journal of Mathematical Physics}\ }\textbf {\bibinfo {volume} {58}},\
  \bibinfo {pages} {032703} (\bibinfo {year} {2017})}\BibitemShut {NoStop}%
\bibitem [{\citenamefont {Manik}, \citenamefont {Timme},\ and\ \citenamefont
  {Witthaut}(2017)}]{ManikTimmeWitthaut2017Cycle}%
  \BibitemOpen
  \bibfield  {author} {\bibinfo {author} {\bibfnamefont {D.}~\bibnamefont
  {Manik}}, \bibinfo {author} {\bibfnamefont {M.}~\bibnamefont {Timme}}, \ and\
  \bibinfo {author} {\bibfnamefont {D.}~\bibnamefont {Witthaut}},\ }\bibfield
  {title} {\enquote {\bibinfo {title} {{Cycle flows and multistability in
  oscillatory networks}},}\ }\href {\doibase 10.1063/1.4994177} {\bibfield
  {journal} {\bibinfo  {journal} {Chaos: An Interdisciplinary Journal of
  Nonlinear Science}\ }\textbf {\bibinfo {volume} {27}},\ \bibinfo {pages}
  {083123} (\bibinfo {year} {2017})}\BibitemShut {NoStop}%
\bibitem [{\citenamefont {Chen}(2019)}]{Chen2019Unmixing}%
  \BibitemOpen
  \bibfield  {author} {\bibinfo {author} {\bibfnamefont {T.}~\bibnamefont
  {Chen}},\ }\bibfield  {title} {\enquote {\bibinfo {title} {{Unmixing the
  Mixed Volume Computation}},}\ }\href {\doibase 10.1007/s00454-019-00078-x}
  {\bibfield  {journal} {\bibinfo  {journal} {Discrete and Computational
  Geometry}\ } (\bibinfo {year} {2019}),\
  10.1007/s00454-019-00078-x}\BibitemShut {NoStop}%
\bibitem [{\citenamefont {Chen}, \citenamefont {Davis},\ and\ \citenamefont
  {Mehta}(2018)}]{ChenDavisMehta2018Counting}%
  \BibitemOpen
  \bibfield  {author} {\bibinfo {author} {\bibfnamefont {T.}~\bibnamefont
  {Chen}}, \bibinfo {author} {\bibfnamefont {R.}~\bibnamefont {Davis}}, \ and\
  \bibinfo {author} {\bibfnamefont {D.}~\bibnamefont {Mehta}},\ }\bibfield
  {title} {\enquote {\bibinfo {title} {{Counting Equilibria of the Kuramoto
  Model Using Birationally Invariant Intersection Index}},}\ }\href {\doibase
  10.1137/17M1145665} {\bibfield  {journal} {\bibinfo  {journal} {SIAM Journal
  on Applied Algebra and Geometry}\ }\textbf {\bibinfo {volume} {2}},\ \bibinfo
  {pages} {489--507} (\bibinfo {year} {2018})}\BibitemShut {NoStop}%
\bibitem [{\citenamefont {Sommese}\ and\ \citenamefont
  {Wampler}(2005)}]{SommeseWampler2005Numerical}%
  \BibitemOpen
  \bibfield  {author} {\bibinfo {author} {\bibfnamefont {A.~J.}\ \bibnamefont
  {Sommese}}\ and\ \bibinfo {author} {\bibfnamefont {C.~W.}\ \bibnamefont
  {Wampler}},\ }\href {\doibase 10.1142/9789812567727} {\emph {\bibinfo {title}
  {{The Numerical Solution of Systems of Polynomials Arising in Engineering and
  Science}}}}\ (\bibinfo  {publisher} {WORLD SCIENTIFIC},\ \bibinfo {year}
  {2005})\BibitemShut {NoStop}%
\bibitem [{\citenamefont {Chen}\ and\ \citenamefont
  {Mehta}(2018)}]{ChenMehta2018Network}%
  \BibitemOpen
  \bibfield  {author} {\bibinfo {author} {\bibfnamefont {T.}~\bibnamefont
  {Chen}}\ and\ \bibinfo {author} {\bibfnamefont {D.}~\bibnamefont {Mehta}},\
  }\bibfield  {title} {\enquote {\bibinfo {title} {{On the Network Topology
  Dependent Solution Count of the Algebraic Load Flow Equations}},}\ }\href
  {\doibase 10.1109/TPWRS.2017.2724030} {\bibfield  {journal} {\bibinfo
  {journal} {IEEE Transactions on Power Systems}\ }\textbf {\bibinfo {volume}
  {33}},\ \bibinfo {pages} {1451--1460} (\bibinfo {year} {2018})},\ \Eprint
  {http://arxiv.org/abs/1512.04987} {arXiv:1512.04987} \BibitemShut {NoStop}%
\bibitem [{\citenamefont {Matsui}\ \emph {et~al.}(2011)\citenamefont {Matsui},
  \citenamefont {Higashitani}, \citenamefont {Nagazawa}, \citenamefont
  {Ohsugi},\ and\ \citenamefont {Hibi}}]{Matsui2011Roots}%
  \BibitemOpen
  \bibfield  {author} {\bibinfo {author} {\bibfnamefont {T.}~\bibnamefont
  {Matsui}}, \bibinfo {author} {\bibfnamefont {A.}~\bibnamefont {Higashitani}},
  \bibinfo {author} {\bibfnamefont {Y.}~\bibnamefont {Nagazawa}}, \bibinfo
  {author} {\bibfnamefont {H.}~\bibnamefont {Ohsugi}}, \ and\ \bibinfo {author}
  {\bibfnamefont {T.}~\bibnamefont {Hibi}},\ }\bibfield  {title} {\enquote
  {\bibinfo {title} {{Roots of Ehrhart polynomials arising from graphs}},}\
  }\href {\doibase 10.1007/s10801-011-0290-8} {\bibfield  {journal} {\bibinfo
  {journal} {Journal of Algebraic Combinatorics}\ }\textbf {\bibinfo {volume}
  {34}},\ \bibinfo {pages} {721--749} (\bibinfo {year} {2011})}\BibitemShut
  {NoStop}%
\bibitem [{\citenamefont {Higashitani}, \citenamefont {Kummer},\ and\
  \citenamefont {Micha{\l}ek}(2016)}]{Higashitani2016Interlacing}%
  \BibitemOpen
  \bibfield  {author} {\bibinfo {author} {\bibfnamefont {A.}~\bibnamefont
  {Higashitani}}, \bibinfo {author} {\bibfnamefont {M.}~\bibnamefont {Kummer}},
  \ and\ \bibinfo {author} {\bibfnamefont {M.}~\bibnamefont {Micha{\l}ek}},\
  }\bibfield  {title} {\enquote {\bibinfo {title} {{Interlacing Ehrhart
  Polynomials of Reflexive Polytopes}},}\ }\href {\doibase
  10.1007/s00029-017-0350-6} {\  (\bibinfo {year} {2016}),\
  10.1007/s00029-017-0350-6},\ \Eprint {http://arxiv.org/abs/1612.07538}
  {arXiv:1612.07538} \BibitemShut {NoStop}%
\bibitem [{\citenamefont {Delucchi}\ and\ \citenamefont
  {Hoessly}(2016)}]{DelucchiHoessly2016Fundamental}%
  \BibitemOpen
  \bibfield  {author} {\bibinfo {author} {\bibfnamefont {E.}~\bibnamefont
  {Delucchi}}\ and\ \bibinfo {author} {\bibfnamefont {L.}~\bibnamefont
  {Hoessly}},\ }\bibfield  {title} {\enquote {\bibinfo {title} {{Fundamental
  polytopes of metric trees via parallel connections of matroids}},}\ }\href
  {http://arxiv.org/abs/1612.05534} {\  (\bibinfo {year} {2016})},\ \Eprint
  {http://arxiv.org/abs/1612.05534} {arXiv:1612.05534} \BibitemShut {NoStop}%
\bibitem [{\citenamefont {Bernshtein}(1975)}]{Bernshtein1975Number}%
  \BibitemOpen
  \bibfield  {author} {\bibinfo {author} {\bibfnamefont {D.~N.}\ \bibnamefont
  {Bernshtein}},\ }\bibfield  {title} {\enquote {\bibinfo {title} {{The number
  of roots of a system of equations}},}\ }\href@noop {} {\bibfield  {journal}
  {\bibinfo  {journal} {Functional Analysis and its Applications}\ }\textbf
  {\bibinfo {volume} {9}},\ \bibinfo {pages} {183--185} (\bibinfo {year}
  {1975})}\BibitemShut {NoStop}%
\bibitem [{\citenamefont {Delabays}, \citenamefont {Jacquod},\ and\
  \citenamefont {D{\"{o}}rfler}(2018)}]{DelabaysJacquodDorfler2018}%
  \BibitemOpen
  \bibfield  {author} {\bibinfo {author} {\bibfnamefont {R.}~\bibnamefont
  {Delabays}}, \bibinfo {author} {\bibfnamefont {P.}~\bibnamefont {Jacquod}}, \
  and\ \bibinfo {author} {\bibfnamefont {F.}~\bibnamefont {D{\"{o}}rfler}},\
  }\bibfield  {title} {\enquote {\bibinfo {title} {{The Kuramoto model on
  directed and signed graphs}},}\ }\href {http://arxiv.org/abs/1807.11410} {\
  (\bibinfo {year} {2018})},\ \Eprint {http://arxiv.org/abs/1807.11410}
  {arXiv:1807.11410} \BibitemShut {NoStop}%
\bibitem [{\citenamefont {Fukuda}, \citenamefont {Liebling},\ and\
  \citenamefont {Margot}(1997)}]{Fukuda1997}%
  \BibitemOpen
  \bibfield  {author} {\bibinfo {author} {\bibfnamefont {K.}~\bibnamefont
  {Fukuda}}, \bibinfo {author} {\bibfnamefont {T.~M.}\ \bibnamefont
  {Liebling}}, \ and\ \bibinfo {author} {\bibfnamefont {F.}~\bibnamefont
  {Margot}},\ }\bibfield  {title} {\enquote {\bibinfo {title} {{Analysis of
  backtrack algorithms for listing all vertices and all faces of a convex
  polyhedron}},}\ }\href {\doibase 10.1016/0925-7721(95)00049-6} {\bibfield
  {journal} {\bibinfo  {journal} {Computational Geometry: Theory and
  Applications}\ }\textbf {\bibinfo {volume} {8}},\ \bibinfo {pages} {1--12}
  (\bibinfo {year} {1997})}\BibitemShut {NoStop}%
\bibitem [{\citenamefont {B{\"{u}}eler}, \citenamefont {Enge},\ and\
  \citenamefont {Fukuda}(2000)}]{BuelerEngeFukuda2000Exact}%
  \BibitemOpen
  \bibfield  {author} {\bibinfo {author} {\bibfnamefont {B.}~\bibnamefont
  {B{\"{u}}eler}}, \bibinfo {author} {\bibfnamefont {A.}~\bibnamefont {Enge}},
  \ and\ \bibinfo {author} {\bibfnamefont {K.}~\bibnamefont {Fukuda}},\
  }\bibfield  {title} {\enquote {\bibinfo {title} {{Exact Volume Computation
  for Polytopes: A Practical Study}},}\ }in\ \href
  {http://link.springer.com/chapter/10.1007/978-3-0348-8438-9_6} {\emph
  {\bibinfo {booktitle} {Polytopes — Combinatorics and Computation}}},\
  \bibinfo {series and number} {\bibinfo {series} {{DMV} {Seminar}}\
  No.~\bibinfo {number} {29}},\ \bibinfo {editor} {edited by\ \bibinfo {editor}
  {\bibfnamefont {G.}~\bibnamefont {Kalai}}\ and\ \bibinfo {editor}
  {\bibfnamefont {G.~M.}\ \bibnamefont {Ziegler}}}\ (\bibinfo  {publisher}
  {Birkh{\"{a}}user Basel},\ \bibinfo {year} {2000})\ pp.\ \bibinfo {pages}
  {131--154}\BibitemShut {NoStop}%
\bibitem [{\citenamefont {Avis}\ and\ \citenamefont {Jordan}(2018)}]{Avis2018}%
  \BibitemOpen
  \bibfield  {author} {\bibinfo {author} {\bibfnamefont {D.}~\bibnamefont
  {Avis}}\ and\ \bibinfo {author} {\bibfnamefont {C.}~\bibnamefont {Jordan}},\
  }\bibfield  {title} {\enquote {\bibinfo {title} {{mplrs: A scalable parallel
  vertex/facet enumeration code}},}\ }\href {\doibase
  10.1007/s12532-017-0129-y} {\bibfield  {journal} {\bibinfo  {journal}
  {Mathematical Programming Computation}\ }\textbf {\bibinfo {volume} {10}},\
  \bibinfo {pages} {267--302} (\bibinfo {year} {2018})}\BibitemShut {NoStop}%
\bibitem [{\citenamefont {Gawrilow}\ and\ \citenamefont
  {Joswig}(2000)}]{polymake}%
  \BibitemOpen
  \bibfield  {author} {\bibinfo {author} {\bibfnamefont {E.}~\bibnamefont
  {Gawrilow}}\ and\ \bibinfo {author} {\bibfnamefont {M.}~\bibnamefont
  {Joswig}},\ }\bibfield  {title} {\enquote {\bibinfo {title} {{polymake: a
  Framework for Analyzing Convex Polytopes}},}\ }in\ \href {\doibase
  10.1007/978-3-0348-8438-9_2} {\emph {\bibinfo {booktitle} {Polytopes —
  Combinatorics and Computation}}}\ (\bibinfo  {publisher} {Birkh{\"{a}}user
  Basel},\ \bibinfo {address} {Basel},\ \bibinfo {year} {2000})\ pp.\ \bibinfo
  {pages} {43--73}\BibitemShut {NoStop}%
\bibitem [{\citenamefont {Li}, \citenamefont {Sauer},\ and\ \citenamefont
  {Yorke}(1989)}]{LiSauerYorke1989Cheater}%
  \BibitemOpen
  \bibfield  {author} {\bibinfo {author} {\bibfnamefont {T.-Y.}\ \bibnamefont
  {Li}}, \bibinfo {author} {\bibfnamefont {T.}~\bibnamefont {Sauer}}, \ and\
  \bibinfo {author} {\bibfnamefont {J.~A.}\ \bibnamefont {Yorke}},\ }\bibfield
  {title} {\enquote {\bibinfo {title} {{The cheater's homotopy: an efficient
  procedure for solving systems of polynomial equations}},}\ }\href@noop {}
  {\bibfield  {journal} {\bibinfo  {journal} {SIAM Journal on Numerical
  Analysis}\ ,\ \bibinfo {pages} {1241--1251}} (\bibinfo {year}
  {1989})}\BibitemShut {NoStop}%
\bibitem [{\citenamefont {Morgan}\ and\ \citenamefont
  {Sommese}(1989)}]{MorganSommese1989Coefficient}%
  \BibitemOpen
  \bibfield  {author} {\bibinfo {author} {\bibfnamefont {A.~P.}\ \bibnamefont
  {Morgan}}\ and\ \bibinfo {author} {\bibfnamefont {A.~J.}\ \bibnamefont
  {Sommese}},\ }\bibfield  {title} {\enquote {\bibinfo {title}
  {{Coefficient-parameter polynomial continuation}},}\ }\href {\doibase
  10.1016/0096-3003(89)90099-4} {\bibfield  {journal} {\bibinfo  {journal}
  {Applied Mathematics and Computation}\ }\textbf {\bibinfo {volume} {29}},\
  \bibinfo {pages} {123--160} (\bibinfo {year} {1989})}\BibitemShut {NoStop}%
\bibitem [{\citenamefont {Kouchnirenko}(1976)}]{Kushnirenko1976Polyedres}%
  \BibitemOpen
  \bibfield  {author} {\bibinfo {author} {\bibfnamefont {A.~G.}\ \bibnamefont
  {Kouchnirenko}},\ }\bibfield  {title} {\enquote {\bibinfo {title}
  {{Poly{\`{e}}dres de Newton et nombres de Milnor}},}\ }\href {\doibase
  10.1007/BF01389769} {\bibfield  {journal} {\bibinfo  {journal} {Inventiones
  Mathematicae}\ }\textbf {\bibinfo {volume} {32}},\ \bibinfo {pages} {1--31}
  (\bibinfo {year} {1976})}\BibitemShut {NoStop}%
\bibitem [{\citenamefont {Huber}\ and\ \citenamefont
  {Sturmfels}(1995)}]{HuberSturmfels1995Polyhedral}%
  \BibitemOpen
  \bibfield  {author} {\bibinfo {author} {\bibfnamefont {B.}~\bibnamefont
  {Huber}}\ and\ \bibinfo {author} {\bibfnamefont {B.}~\bibnamefont
  {Sturmfels}},\ }\bibfield  {title} {\enquote {\bibinfo {title} {{A polyhedral
  method for solving sparse polynomial systems}},}\ }\href {\doibase
  10.1090/S0025-5718-1995-1297471-4} {\bibfield  {journal} {\bibinfo  {journal}
  {Mathematics of Computation}\ }\textbf {\bibinfo {volume} {64}},\ \bibinfo
  {pages} {1541--1555} (\bibinfo {year} {1995})}\BibitemShut {NoStop}%
\bibitem [{\citenamefont {Allgower}(1981)}]{allgower_survey_1981}%
  \BibitemOpen
  \bibfield  {author} {\bibinfo {author} {\bibfnamefont {E.~L.}\ \bibnamefont
  {Allgower}},\ }\bibfield  {title} {\enquote {\bibinfo {title} {{A survey of
  homotopy methods for smooth mappings}},}\ }in\ \href
  {http://link.springer.com/chapter/10.1007/BFb0090675} {\emph {\bibinfo
  {booktitle} {Numerical Solution of Nonlinear Equations}}},\ \bibinfo {series
  and number} {\bibinfo {series} {Lecture Notes in Mathematics}\ No.\ \bibinfo
  {number} {878}},\ \bibinfo {editor} {edited by\ \bibinfo {editor}
  {\bibfnamefont {E.~L.}\ \bibnamefont {Allgower}}, \bibinfo {editor}
  {\bibfnamefont {K.}~\bibnamefont {Glashoff}}, \ and\ \bibinfo {editor}
  {\bibfnamefont {H.-O.}\ \bibnamefont {Peitgen}}}\ (\bibinfo  {publisher}
  {Springer Berlin Heidelberg},\ \bibinfo {year} {1981})\ pp.\ \bibinfo {pages}
  {1--29}\BibitemShut {NoStop}%
\bibitem [{\citenamefont {Chen}\ and\ \citenamefont
  {Li}(2015)}]{ChenLi2015Homotopy}%
  \BibitemOpen
  \bibfield  {author} {\bibinfo {author} {\bibfnamefont {T.}~\bibnamefont
  {Chen}}\ and\ \bibinfo {author} {\bibfnamefont {T.-Y.}\ \bibnamefont {Li}},\
  }\bibfield  {title} {\enquote {\bibinfo {title} {{Homotopy continuation
  method for solving systems of nonlinear and polynomial equations}},}\ }\href
  {\doibase 10.4310/CIS.2015.v15.n2.a1} {\bibfield  {journal} {\bibinfo
  {journal} {Commun. Inf. Syst.}\ }\textbf {\bibinfo {volume} {15}},\ \bibinfo
  {pages} {119--307} (\bibinfo {year} {2015})}\BibitemShut {NoStop}%
\bibitem [{\citenamefont {Davidenko}(1953)}]{davidenko_new_1953}%
  \BibitemOpen
  \bibfield  {author} {\bibinfo {author} {\bibfnamefont {D.~F.}\ \bibnamefont
  {Davidenko}},\ }\bibfield  {title} {\enquote {\bibinfo {title} {{On a new
  method of numerical solution of systems of nonlinear equations}},}\ }in\
  \href@noop {} {\emph {\bibinfo {booktitle} {Dokl. Akad. Nauk SSSR}}},\
  Vol.~\bibinfo {volume} {88}\ (\bibinfo {year} {1953})\ pp.\ \bibinfo {pages}
  {601--602}\BibitemShut {NoStop}%
\bibitem [{\citenamefont {Zachariah}\ \emph {et~al.}(2018)\citenamefont
  {Zachariah}, \citenamefont {Charles}, \citenamefont {Boston},\ and\
  \citenamefont {Lesieutre}}]{Zachariah2018Distributions}%
  \BibitemOpen
  \bibfield  {author} {\bibinfo {author} {\bibfnamefont {A.}~\bibnamefont
  {Zachariah}}, \bibinfo {author} {\bibfnamefont {Z.}~\bibnamefont {Charles}},
  \bibinfo {author} {\bibfnamefont {N.}~\bibnamefont {Boston}}, \ and\ \bibinfo
  {author} {\bibfnamefont {B.}~\bibnamefont {Lesieutre}},\ }\bibfield  {title}
  {\enquote {\bibinfo {title} {{Distributions of the Number of Solutions to the
  Network Power Flow Equations}},}\ }in\ \href {\doibase
  10.1109/ISCAS.2018.8351675} {\emph {\bibinfo {booktitle} {2018 IEEE
  International Symposium on Circuits and Systems (ISCAS)}}}\ (\bibinfo
  {publisher} {IEEE},\ \bibinfo {year} {2018})\ pp.\ \bibinfo {pages}
  {1--5}\BibitemShut {NoStop}%
\bibitem [{\citenamefont {Chen}\ and\ \citenamefont
  {Davis}(2018)}]{ChenDavis2018Toric}%
  \BibitemOpen
  \bibfield  {author} {\bibinfo {author} {\bibfnamefont {T.}~\bibnamefont
  {Chen}}\ and\ \bibinfo {author} {\bibfnamefont {R.}~\bibnamefont {Davis}},\
  }\bibfield  {title} {\enquote {\bibinfo {title} {{A toric deformation method
  for solving Kuramoto equations}},}\ }\href {http://arxiv.org/abs/1810.05690}
  {\  (\bibinfo {year} {2018})},\ \Eprint {http://arxiv.org/abs/1810.05690}
  {arXiv:1810.05690} \BibitemShut {NoStop}%
\bibitem [{\citenamefont {Griffin}\ and\ \citenamefont
  {Hauenstein}(2015)}]{GriffinHauenstein2015Real}%
  \BibitemOpen
  \bibfield  {author} {\bibinfo {author} {\bibfnamefont {Z.~A.}\ \bibnamefont
  {Griffin}}\ and\ \bibinfo {author} {\bibfnamefont {J.~D.}\ \bibnamefont
  {Hauenstein}},\ }\bibfield  {title} {\enquote {\bibinfo {title} {{Real
  solutions to systems of polynomial equations and parameter continuation}},}\
  }\href {\doibase 10.1515/advgeom-2015-0004} {\bibfield  {journal} {\bibinfo
  {journal} {Advances in Geometry}\ }\textbf {\bibinfo {volume} {15}},\
  \bibinfo {pages} {173--187} (\bibinfo {year} {2015})}\BibitemShut {NoStop}%
\bibitem [{\citenamefont {Chen}\ and\ \citenamefont
  {Mehta}(2015)}]{ChenMehta2015Index}%
  \BibitemOpen
  \bibfield  {author} {\bibinfo {author} {\bibfnamefont {T.}~\bibnamefont
  {Chen}}\ and\ \bibinfo {author} {\bibfnamefont {D.}~\bibnamefont {Mehta}},\
  }\bibfield  {title} {\enquote {\bibinfo {title} {{An index-resolved
  fixed-point homotopy and potential energy landscapes}},}\ }\href
  {http://arxiv.org/abs/1504.06622 http://www.arxiv.org/pdf/1504.06622.pdf}
  {\bibfield  {journal} {\bibinfo  {journal} {arXiv:1504.06622 [cond-mat]}\ }
  (\bibinfo {year} {2015})}\BibitemShut {NoStop}%
\end{thebibliography}%
\end{document}